\documentclass[11pt]{amsart}

\usepackage[dvipsnames]{xcolor}

\usepackage[
		backend=biber,
		style=alphabetic,
		doi=true,
		isbn=false,
		url=true
	]{biblatex}
\DeclareFieldFormat[
		article,
		book,
		inbook,
		incollection,
		inproceedings,
		patent,
		thesis,
		unpublished
	]{title}{\emph{#1\isdot}}
\DeclareFieldFormat{journaltitle}{#1\isdot}
\AtEveryBibitem{
	\clearfield{note}
	\iffieldundef{doi}{}{\clearfield{url}}
	\iffieldequalstr{eprinttype}{arxiv}{\clearfield{url}}{}
	\iffieldequalstr{eprinttype}{arXiv}{\clearfield{url}}{}
}

\usepackage{amssymb, mathrsfs, mathtools, braket, mleftright}

\usepackage{tikz}
\usetikzlibrary{cd}

\usepackage[
		colorlinks,
		linkcolor=Maroon,
		citecolor=ForestGreen,
		urlcolor=MidnightBlue
	]{hyperref}

\usepackage[nameinlink]{cleveref}

\crefname{section}{Section}{Sections}
\crefname{subsection}{Section}{Sections}
\crefname{subsubsection}{Section}{Sections}

\makeatletter
\renewcommand{\subsubsection}{%
	\@startsection{subsubsection}{3}{\z@}%
		{.5\linespacing\@plus.7\linespacing}{-.5em}%
		{\normalfont\bfseries}}
\makeatother

\newtheorem{thm}{Theorem}[section]
\crefname{thm}{Theorem}{Theorems}
\newtheorem{prop}[thm]{Proposition}
\crefname{prop}{Proposition}{Propositions}
\newtheorem{lem}[thm]{Lemma}
\crefname{lem}{Lemma}{Lemmas}
\newtheorem{cor}[thm]{Corollary}
\crefname{cor}{Corollary}{Corollaries}

\crefname{conj}{Conjecture}{Conjectures}

\crefname{ques}{Question}{Questions}

\theoremstyle{definition}
\newtheorem{defi}[thm]{Definition}
\crefname{defi}{Definition}{Definitions}
\newtheorem{setting}[thm]{Setting}
\crefname{setting}{Setting}{Settings}
\newtheorem{exa}[thm]{Example}
\crefname{exa}{Example}{Examples}
\newtheorem{claim}{Claim}
\crefname{claim}{Claim}{Claims}

\theoremstyle{remark}
\newtheorem{rem}[thm]{Remark}
\crefname{rem}{Remark}{Remarks}

\crefname{enumi}{}{}
\creflabelformat{enumi}{#2(#1)#3}

\crefname{enumii}{}{}
\creflabelformat{enumii}{#2(#1)#3}

\numberwithin{equation}{section}

\newcommand{\longhookrightarrow}{\lhook\joinrel\longrightarrow}

\newcommand{\abs}[1]{\lvert #1\rvert}
\newcommand{\card}{\abs}
\newcommand{\bigcard}[1]{\bigl\lvert #1\bigr\rvert}
\newcommand{\innprod}[2]{\langle #1,#2\rangle}
\newcommand{\isom}{\cong}
\newcommand{\isomarrow}{\xrightarrow{\;\sim\;}}
\newcommand{\lineq}{\sim}

\newcommand{\NEb}{\overline{\mathrm{NE}}}
\DeclareMathOperator{\Nef}{Nef}
\DeclareMathOperator{\Mov}{Mov}

\newcommand{\restr}[2]{#1\rvert_{#2}}
\newcommand{\restrdiv}[2]{#1\rvert_{#2}}

\newcommand{\PP}{\mathbb{P}}
\renewcommand{\AA}{\mathbb{A}}
\newcommand{\ZZ}{\mathbb{Z}}
\newcommand{\RR}{\mathbb{R}}
\newcommand{\RRp}{\RR_{\ge0}}
\newcommand{\QQ}{\mathbb{Q}}
\newcommand{\CC}{\mathbb{C}}
\newcommand{\Ob}[1]{\mathcal{O}_{#1}}
\newcommand{\Ic}{\mathcal{I}}
\newcommand{\Rc}{\mathcal{R}}

\DeclareMathOperator{\Cone}{Cone}
\DeclareMathOperator{\Bl}{Bl}
\DeclareMathOperator{\codim}{codim}

\DeclareMathOperator{\Pic}{Pic}
\DeclareMathOperator{\Div}{div}
\DeclareMathOperator{\Spec}{Spec}
\DeclareMathOperator{\Proj}{Proj}
\DeclareMathOperator{\Exc}{Exc}

\DeclareMathOperator{\relint}{rel\,int}
\DeclareMathOperator{\Hom}{Hom}
\DeclareMathOperator{\Ker}{Ker}

\newcommand{\sset}[2]{\set{#1|#2}}

\title{Relative cone of curves and contractions of a successive blowup}
\author{Yuto Masamura}
\address{Graduate School of Mathematical Sciences, The University of Tokyo, 3-8-1 Komaba, Meguro-ku, Tokyo 153-8914, Japan}
\email{masamura@ms.u-tokyo.ac.jp}
\subjclass[2020]{Primary 14E30; Secondary 14M25, 14C20, 14E05}
\keywords{cone of curves, extremal contraction, blowup}

\begin{document}

\begin{abstract}
	Let $X$ be a normal variety in arbitrary characteristic, and let $\tilde X$ be the successive blowup along subvarieties $Z_1,\dotsc,Z_n\subseteq X$ of codimension at least two that have simple normal crossings and satisfy $Z_h\not\supseteq Z_i$ whenever $h<i$.
	We prove that the relative cone of curves $\NEb(\tilde X/X)$ is generated by the classes of finitely many \emph{elementary curves}, and that every face admits a contraction over $X$.
	Moreover, we prove that every small face contraction admits a $D$-flip for every $\RR$-Cartier divisor $D$ that is relatively anti-ample.
	We also compute the relative Cox sheaf and prove its finite generation.
\end{abstract}

\maketitle

\tableofcontents

\section{Introduction}\label{sec:intro}

The cone of curves describes the numerical classes of effective curves on a variety and is a basic tool for studying contractions in birational geometry.
For a klt pair $(X,B)$ projective over a base $S$ in characteristic zero, the cone and contraction theorems give precise control of the $(K_X+B)$-negative region of the relative cone of curves $\NEb(X/S)$.
The cone is locally rational polyhedral in this region, and every $(K_X+B)$-negative extremal face admits a contraction.
The minimal model program proceeds by contracting $(K_X+B)$-negative extremal rays and, when a contraction is small, replacing it by a flip \cite{KollarMori}.
These results, however, do not describe the entire relative cone of curves, which need not be polyhedral outside the $(K_X+B)$-negative region.
Moreover, in positive characteristic, cone and contraction theorems for klt pairs are not known in all dimensions.
It is therefore natural to study classes of projective morphisms for which, in arbitrary characteristic, one can describe the entire relative cone of curves, contract every face, and construct flips for all small face contractions.
The toric setting provides the basic model.
For smooth projective toric varieties, Batyrev described the cone of curves in terms of primitive collections \cite{Batyrev}.
Toric Mori theory was developed by Reid and Fujino--Sato \cite{ReidToric,FujinoSato}.

In this paper, we study the relative birational geometry of the following class of successive blowups.
Let $X$ be a normal variety in arbitrary characteristic, and let $Z_1,\dotsc,Z_n\subseteq X$ be subvarieties with simple normal crossings (\cref{def:snc}).
Assume that $c_i\coloneq\codim_X Z_i\ge2$ for every $i$ and that $Z_h\not\supseteq Z_i$ whenever $h<i$.
Starting with $X^{(0)}\coloneq X$, define successively
\[
	X^{(i)}\coloneq\Bl_{Z_i^{(i-1)}}X^{(i-1)}\qquad(1\le i\le n),
\]
where $Z_i^{(i-1)}\subseteq X^{(i-1)}$ denotes the strict transform of $Z_i$.
Set $\tilde X\coloneq X^{(n)}$, and write $\pi\colon\tilde X\to X$ for the composite morphism.
We determine the structure of the relative cone of curves $\NEb(\tilde X/X)$ and study the contractions of its faces and the flips of its small face contractions.

For a single blowup $\tilde X=\Bl_ZX$, the relative cone of curves $\NEb(\tilde X/X)$ is the ray generated by the class of a line in an exceptional fiber.
For two centers, $\NEb(\tilde X/X)$ always has two extremal rays, but the corresponding contractions can be divisorial or small, depending on how the centers meet.
For example, if the centers are disjoint, the contractions of both rays are divisorial blowdowns.
If some irreducible component $Z\subseteq Z_1\cap Z_2$ has excess dimension, meaning that
\[
	\codim_X Z<c_1+c_2,
\]
then the contraction other than the blowdown $\tilde X\to X^{(1)}$ is small if and only if $Z_1\not\subseteq Z_2$.
Such small contractions already appear in Kawamata's four-dimensional examples \cite{Kawamata}.
Tsukioka extended this construction to arbitrary smooth projective varieties of dimension at least four \cite{Tsukioka}, and Yoshida and the author later treated more general centers and intersection patterns \cite{MasamuraYoshida}.
These constructions are all carried out over $\CC$.
The present paper extends this two-center picture to an arbitrary number of centers with simple normal crossings, in arbitrary characteristic.

Successive blowups also arise in other parts of birational geometry.
They appear in Casagrande's structure theorem for Fano fourfolds with large Picard number \cite{Casagrande} and in Chaput--Fu's resolutions of stratified Mukai flops \cite{ChaputFu}.
The wonderful models of De Concini--Procesi and Li are likewise constructed by successive blowups, with centers forming a building set \cite{DCP,LiLi}.

To formulate our cone theorem, we introduce a class of $\pi$-contracted curves defined only by their intersection numbers with the exceptional divisors.
For each $1\le i\le n$, let $E_i$ denote the pullback to $\tilde X$ of the exceptional divisor of the $i$-th blowup.
For a $\pi$-contracted curve $C\subseteq\tilde X$, consider its intersection vector
\[
	(C_1,\dotsc,C_n),\qquad C_i\coloneq E_i\cdot C.
\]
This vector determines the relative numerical class $[C]\in N_1(\tilde X/X)$.
Let $i_0$ be the smallest index for which $C_{i_0}\ne0$, and call $i_0$ the \emph{source} of $C$.
We say that $C$ is \emph{elementary} if $C_{i_0}=-1$ and
\[
	-\bigcard{\sset{h<i}{C_h=1}}\le C_i\le1\qquad(i>i_0).
\]
See \cref{ss:tower} for further details.
The condition $C_{i_0}=-1$ means geometrically that $C$ maps birationally onto a line in a fiber of the $i_0$-th blowup $X^{(i_0)}\to X^{(i_0-1)}$.
These conditions depend only on $n$, and only finitely many integer vectors satisfy them, although some of these vectors may not be represented by curves on $\tilde X$.

\begin{thm}[{\cref{thm:cone}}]\label{introthm:cone}
	Let $X$ be a normal variety over an algebraically closed field, and let $Z_1,\dotsc,Z_n\subseteq X$ be subvarieties of codimension at least two.
	Assume that $Z_1,\dotsc,Z_n$ have simple normal crossings and that $Z_h\not\supseteq Z_i$ whenever $h<i$.

	Then the relative cone of curves $\NEb(\tilde X/X)$ is generated by the classes of finitely many elementary curves.
	In particular, every extremal ray is spanned by the class of an elementary curve.
\end{thm}

The numerical conditions defining elementary curves are sharp when $X$ and the centers are allowed to vary.
Indeed, every nonzero integer vector satisfying these conditions occurs as an extremal generator for a suitable choice of $X$ and centers $Z_1,\dotsc,Z_n$ (\cref{prop:realization}).
We also bound the number of extremal rays in terms of the number $n$ of centers alone, independently of the dimension of $X$ and of the configuration of the centers (\cref{prop:raycount}).
The number of extremal rays is at most $2(n-1)!$.

The cone theorem gives each face of $\NEb(\tilde X/X)$ a $\pi$-nef supporting Cartier divisor.
To construct the corresponding contraction, it suffices to prove that this divisor is $\pi$-semiample.
Our second main result proves the stronger statement that every $\pi$-nef Cartier divisor on $\tilde X$ is $\pi$-base-point-free.

\begin{thm}[{\cref{thm:bpf,cor:contract}}]\label{introthm:contract}
	In the setting of \cref{introthm:cone}, the following hold.
	\begin{enumerate}
		\item Let $D$ be a Cartier divisor on $\tilde X$.
			If $D$ is $\pi$-nef, then $D$ is $\pi$-base-point-free.
		\item For every face $F\subseteq\NEb(\tilde X/X)$, the contraction $g_F\colon\tilde X\to\tilde X_F$ of $F$ over $X$ exists.
	\end{enumerate}
\end{thm}

Since $\pi$ is birational, every contraction over $X$ obtained above is birational as well.
We identify the target of each face contraction as the normalized blowup of an explicit ideal on $X$.

\begin{thm}[{\cref{thm:targetn}}]\label{introthm:target}
	In the setting of \cref{introthm:cone}, let $F$ be a face of $\NEb(\tilde X/X)$, and let $g_F\colon\tilde X\to\tilde X_F$ be its contraction.
	Let $a=(a_1,\dotsc,a_n)\in\ZZ_{\ge0}^n$ be a vector such that $-\sum_i a_iE_i$ is a supporting divisor of $F$.
	Set $\Ic^{(a)}\coloneq\Ic_{Z_1}^{a_1}\cap\dotsb\cap\Ic_{Z_n}^{a_n}$.
	Then there is an isomorphism
	\[
		\tilde X_F\isom\overline{\Bl}_{\Ic^{(a)}}X
	\]
	over $X$, where $\overline{\Bl}_{\Ic^{(a)}}X$ denotes the normalization of the blowup $\Bl_{\Ic^{(a)}}X$.
\end{thm}

Moreover, small face contractions admit $D$-flips.

\begin{thm}[{\cref{thm:flip}}]\label{introthm:flip}
	In the setting of \cref{introthm:cone}, let $F$ be a face of $\NEb(\tilde X/X)$ whose contraction $g_F\colon\tilde X\to\tilde X_F$ is small.
	Let $D$ be an $\RR$-Cartier divisor on $\tilde X$ such that $-D$ is $g_F$-ample.
	Then the $D$-flip
	\[
		g_F^+\colon\tilde X^+\longrightarrow\tilde X_F
	\]
	of $g_F$ exists.
\end{thm}

If $F=R$ is an extremal ray, then the flipped model is independent of $D$, up to unique isomorphism over $\tilde X_R$ (\cref{lem:flipindependence}).
If $X$ is $\QQ$-factorial, then $\tilde X$ is also $\QQ$-factorial (\cref{lem:inherit}).
For an extremal ray $R$, the construction then gives a flip, a flop, or an anti-flip according to whether $K_{\tilde X}$ is negative, numerically trivial, or positive on $R$.

In the two-center construction of \cite{MasamuraYoshida}, the flip is obtained by reversing the order of the blowups.
For an arbitrary number of centers, we characterize when the successive blowup along the same centers in another order gives the flip (\cref{thm:reorderflip}).
Note that already for three centers, a flop need not arise from any reordering of the original centers (\cref{exa:branch}).

We also describe the exceptional locus of the contraction $g_R$ of an extremal ray $R$ (\cref{prop:locus}).
It is smooth and contained in $\bigcap_{E_i\cdot R<0}E_i$, and its image in $X$ is a union of connected components of $\bigcap_{E_i\cdot R\ne0}Z_i$.

In \cref{sec:examples}, we recover the two-center construction of \cite{MasamuraYoshida}, including the contraction target and exceptional locus (\cref{cor:MY}).
We then illustrate our results with two examples involving three centers, describing the contraction targets and exceptional loci in each case.

Finally, we consider the relative Cox sheaf of $\pi$.

\begin{thm}[{\cref{thm:cox}}]\label{introthm:cox}
	In the setting of \cref{introthm:cone}, the relative Cox sheaf of $\pi$
	\[
		\Rc\coloneq\bigoplus_{m\in\ZZ^n}\pi_*\Ob{\tilde X}\biggl(\sum_im_iE_i\biggr)
	\]
	is a finitely generated $\Ob X$-algebra.
\end{thm}

If $X$ is $\QQ$-factorial and quasi-projective, then $\pi$ is a Mori dream morphism in the sense of Ohta \cite{Ohta} (\cref{cor:mdm}).
Over $\CC$, Mori dream theory then gives the rational polyhedrality of $\NEb(\tilde X/X)$ and the existence of contractions and flips \cite[\S5]{Ohta}.
In contrast, our results hold in arbitrary characteristic, and they describe the cone by elementary curves and the targets of the contractions by normalized blowups of explicit ideals.

We now outline the proofs of these results.
In the two-center construction of \cite{MasamuraYoshida}, the relative nef cone is computed directly from the geometry of the fibers.
For an arbitrary number of centers, we use local toric models of the entire blowup sequence.
The simple normal crossings condition ensures that this sequence is locally an \'etale base change of a sequence of coordinate blowups of $\AA^d$ (\cref{sec:model}).
We track wall curves through the corresponding star subdivisions and compute their intersection numbers with the exceptional divisors step by step.
These computations, together with bounds from primitive relations, give the inequalities defining elementary curves.

For contractions, we compare curve classes on $\tilde X$ and its coordinate models using intersections with the exceptional divisors.
This comparison transfers relative nefness to the toric models.
Toric base-point-freeness then gives relative base-point-freeness on $\tilde X$, and hence contractions of all faces in arbitrary characteristic.
For small face contractions, flat base change reduces finite generation of the flip algebra to the toric case.
The same local models reduce the finite generation of the relative Cox sheaf to the toric computation of W\l{}odarczyk \cite{WlodarczykCox}.

\subsubsection*{Organization of the paper}

In \cref{sec:setup}, we fix notation, introduce elementary curves, and recall the required toric geometry.
In \cref{sec:model}, we construct local coordinate models for the blowup sequence.
In \cref{sec:cone}, we prove \cref{introthm:cone} and study the extremal rays.
In \cref{sec:contractions}, we construct contractions of faces and flips, prove \cref{introthm:contract,introthm:target,introthm:flip}, and give examples.
In \cref{sec:cox}, we prove \cref{introthm:cox} on the relative Cox sheaf and study small modifications of $\tilde X$ over $X$.

\subsubsection*{Acknowledgements}

The author is grateful to his Ph.D. advisor, Keiji Oguiso, for his invaluable guidance and encouragement.
He also thanks Cinzia Casagrande for her insightful comments and suggestions.
The author was supported by JSPS KAKENHI Grant Number JP24KJ0856.

The author used the AI assistants Claude (Anthropic) and GPT (OpenAI) to help with his research and the writing of this paper.
The author is responsible for the content of the paper.
This research was supported in part by access to OpenAI models provided through the ChatGPT program for academic researchers.

\section{Setup and preliminaries}\label{sec:setup}

In this section, we fix notation and develop the basic geometry of the successive blowup construction used throughout the paper, including its elementary curves and its toric coordinate model.

\subsection{Notation and conventions}\label{ss:conventions}

We work over an algebraically closed field $k$ of arbitrary characteristic.
A \emph{variety} is an integral separated scheme of finite type over $k$.

For a locally free sheaf $F$ of finite rank on a scheme $X$, we use Grothendieck's convention: $\PP_X(F)$ parametrizes invertible quotients of $F$.

For a proper birational morphism $g\colon X\to Y$ of normal varieties, we endow the exceptional locus $\Exc(g)$ with its reduced induced scheme structure.
We call $g$ \emph{small} if $\Exc(g)$ has codimension at least two in $X$.

Let $f\colon X\to S$ be a projective morphism of varieties with $X$ normal.
The \emph{relative Picard group} is $\Pic(X/S)\coloneq\Pic(X)/f^*\Pic(S)$.
We write $N_1(X/S)$ for the $\RR$-vector space of $f$-contracted curves modulo numerical equivalence over $S$.
We write $N^1(X/S)$ for the $\RR$-vector space of $\RR$-Cartier divisors on $X$ modulo numerical equivalence over $S$.
The \emph{relative cone of curves} $\NEb(X/S)\subseteq N_1(X/S)$ is the closure of the convex cone generated by the classes $[C]$ of $f$-contracted curves $C\subseteq X$.
The \emph{relative Picard number} is $\rho(X/S)\coloneq\dim_\RR N^1(X/S)$.
When $S=\Spec k$, we omit $S$ from this notation and write $N_1(X)$, $N^1(X)$, $\NEb(X)$, and $\rho(X)$.

A \emph{face} of $\NEb(X/S)$ is a convex subcone $F\subseteq\NEb(X/S)$ with the following property: if $\alpha_1,\alpha_2\in\NEb(X/S)$ and $\alpha_1+\alpha_2\in F$, then $\alpha_1,\alpha_2\in F$.
A one-dimensional face is called an \emph{extremal ray}.
A \emph{supporting divisor} of a face $F$ is an $f$-nef $\RR$-Cartier divisor $D$ such that
\[
	F=\sset{\alpha\in\NEb(X/S)}{D\cdot\alpha=0}.
\]
Thus $D$ is nonnegative on $\NEb(X/S)$ and vanishes precisely on $F$.

The \emph{contraction of a face $F$ over $S$}, if it exists, is a projective surjective morphism of varieties $g_F\colon X\to X_F$ over $S$ such that $(g_F)_*\Ob X=\Ob{X_F}$ and $g_F$ contracts precisely those $f$-contracted curves $C\subseteq X$ whose numerical classes $[C]$ lie in $F$.
It is unique up to a unique isomorphism of the target commuting with the morphisms from $X$.
We refer to \cite{KollarMori} for the standard terminology concerning cones of curves and contractions.

Let $g_R\colon X\to X_R$ be a birational extremal ray contraction.
We call $g_R$ \emph{divisorial} if $\Exc(g_R)$ is a prime divisor.
We impose neither $\QQ$-factoriality nor $K_X$-negativity in this definition.
If $X$ is $\QQ$-factorial, then $g_R$ is divisorial as soon as $\Exc(g_R)$ contains a codimension-one component.
Without $\QQ$-factoriality, $\Exc(g_R)$ can contain more than one prime divisor \cite[Example~4.1]{FujinoCompletions}, so $g_R$ need not be either divisorial or small in our terminology.
In \cref{set:main}, however, we will see that every extremal ray contraction of $\NEb(\tilde X/X)$ is either divisorial or small, even without assuming $\QQ$-factoriality (\cref{cor:typen}).

For a fan $\Sigma$ in a lattice $N$ with dual lattice $M$, we write $X_\Sigma$ for the associated toric variety, obtained by gluing the affine toric varieties $\Spec k[\sigma^\vee\cap M]$ for $\sigma\in\Sigma$ along the open subsets corresponding to common faces.

\subsection{Successive blowups and elementary curves}\label{ss:tower}

We now describe the successive blowup construction and define elementary curves, whose numerical classes will be shown to generate the relative cone of curves.

\begin{defi}\label{def:snc}
	Let $X$ be a variety and let $Z_1,\dotsc,Z_n\subseteq X$ be (closed) subvarieties.
	The subvarieties $Z_i$ \emph{have simple normal crossings} if, at every closed point $p\in \bigcup_i Z_i$, there is a regular system of parameters $x_1,\dotsc,x_d$ of $\Ob{X,p}$ and there are subsets $J_i\subseteq\{1,\dotsc,d\}$ such that every $Z_i$ containing $p$ satisfies
	\[
		\Ic_{Z_i,p}=(x_j\mid j\in J_i).
	\]
\end{defi}

Note that if $Z_1,\dotsc,Z_n$ are subvarieties with simple normal crossings, then $X$ is smooth along $\bigcup_i Z_i$, each $Z_i$ is smooth, and, for every nonempty subset $I\subseteq\{1,\dotsc,n\}$, the intersection $\bigcap_{i\in I} Z_i$ is either empty or smooth.

We will work in the following setup.

\begin{setting}\label{set:main}
	Let $X$ be a normal variety, and let $Z_1,\dotsc,Z_n\subseteq X$ be subvarieties.
	Assume that
	\begin{itemize}
		\item $c_i\coloneq\codim_X Z_i\ge 2$,
		\item $Z_h\not\supseteq Z_i$ for $h<i$, and
		\item $Z_1,\dotsc,Z_n$ have simple normal crossings.
	\end{itemize}
	Let
	\[
	\begin{tikzcd}
		\tilde X\coloneq X^{(n)}\arrow[r,"\pi_n"]
			& X^{(n-1)}\arrow[r]
			& \dotsb\arrow[r]
			& X^{(1)}\arrow[r,"\pi_1"]
			& X^{(0)}\coloneq X
	\end{tikzcd}
	\]
	be the sequence of blowups defined inductively as follows:
	\begin{itemize}
		\item $X^{(0)}\coloneq X$,
		\item $Z_h^{(0)}\coloneq Z_h$ for $1\le h\le n$,
		\item $X^{(i)}\coloneq\Bl_{Z_i^{(i-1)}}X^{(i-1)}$, and
		\item $Z_h^{(i)}\coloneq(\pi_i)^{-1}_*Z_h^{(i-1)}$ for $h>i$.
	\end{itemize}
	Set $\pi\coloneq\pi_1\circ\dotsb\circ\pi_n\colon\tilde X\to X$.
	For each $i$, let $E_i$ denote the pullback to $\tilde X$ of the exceptional divisor
	$E_i^{(i)}\subseteq X^{(i)}$ of $\pi_i$.
\end{setting}

We first record that the simple normal crossings condition is preserved under the successive blowups in \cref{set:main}.

\begin{lem}\label{lem:inherit}
	In \cref{set:main}, the subvarieties $Z_{i+1}^{(i)},\dotsc,Z_n^{(i)}\subseteq X^{(i)}$ have simple normal crossings for $1\le i<n$.
	In particular, $\tilde X$ is normal.
	If $X$ is $\QQ$-factorial (resp. smooth), then $\tilde X$ is $\QQ$-factorial (resp. smooth).
\end{lem}

\begin{proof}
	We first prove the simple normal crossings assertion for $i=1$ and then apply the same argument inductively.

	Away from $Z_1$, the blowup $\pi_1$ is an isomorphism, so we work locally near a closed point $p\in Z_1$.
	Shrinking $X$ around $p$, we may assume that $Z_h=\emptyset$ whenever $p\notin Z_h$.
	By the simple normal crossings assumption, after shrinking $X$ further, we may choose \'etale coordinates $x_1,\dotsc,x_d$ on $X$ such that
	\[
		Z_h=V(x_j\mid j\in J_h)
	\]
	for every nonempty $Z_h$, where $J_h\subseteq\{1,\dotsc,d\}$.

	Fix $j_0\in J_1$, and let $U_{j_0}\subseteq X^{(1)}$ be the standard blowup chart corresponding to the generator $x_{j_0}$ of $\Ic_{Z_1}$.
	On this chart, we have \'etale coordinates
	\[
		y_{j_0}=x_{j_0},\qquad
		y_j=\frac{x_j}{x_{j_0}}\quad(j\in J_1\setminus\{j_0\}),\qquad
		y_j=x_j\quad(j\notin J_1).
	\]
	The exceptional divisor on $U_{j_0}$ is defined by $y_{j_0}=0$.
	For each $h\ge2$ with $Z_h\ne\emptyset$, saturating the pullback of $\Ic_{Z_h}$ with respect to $y_{j_0}$ gives
	\[
		Z_h^{(1)}\cap U_{j_0}
		=\begin{cases}
			\emptyset & j_0\in J_h,\\
			V(y_j\mid j\in J_h) & j_0\notin J_h.
		\end{cases}
	\]
	Thus, on every standard blowup chart, the nonempty strict transforms of $Z_2,\dotsc,Z_n$ are coordinate subspaces.
	Since $p\in Z_1$ was arbitrary, $Z_2^{(1)},\dotsc,Z_n^{(1)}$ have simple normal crossings, and the first assertion follows by induction.

	Choose a smooth open neighborhood $U\subseteq X$ of $\bigcup_i Z_i$.
	By the first assertion, the restriction of the blowup sequence to $U$ consists of blowups along smooth centers, so $\pi^{-1}(U)$ is smooth.
	Since $\pi$ is an isomorphism over $X\setminus\bigcup_i Z_i$, the assertions about normality, $\QQ$-factoriality, and smoothness follow.
\end{proof}

\begin{lem}\label{lem:divisordecomposition}
	In \cref{set:main}, let $D$ be a Cartier divisor on $\tilde X$.
	Then $D_0\coloneq\pi_*D$ is a Cartier divisor on $X$, and there are unique integers $a_1,\dotsc,a_n$ such that
	\[
		D=\pi^*D_0-\sum_{i=1}^n a_iE_i.
	\]
\end{lem}

\begin{proof}
	The variety $X$ is smooth on a neighborhood of $\bigcup_iZ_i$, so the Weil divisor $D_0$ is Cartier there.
	Over $X\setminus\bigcup_iZ_i$, the morphism $\pi$ is an isomorphism, so $D_0$ is Cartier there as well.
	Hence $D_0$ is Cartier on $X$.

	Since $D-\pi^*D_0$ is supported on the exceptional prime divisors $E_1,\dotsc,E_n$, it is a unique integral combination of the $E_i$.
	This proves the assertion.
\end{proof}

To formulate the numerical conditions below, we first introduce some notation.

\begin{defi}
	In \cref{set:main}, let $C\subseteq\tilde X$ be a curve contracted by $\pi$.
	\begin{itemize}
		\item For $1\le i\le n$, let $C_i\coloneq E_i\cdot C$.
		\item The \emph{source} of $C$ is the minimal $i_0$ such that $C_{i_0}\ne0$.
	\end{itemize}
\end{defi}

The divisor classes $[E_1],\dotsc,[E_n]$ form a basis of $N^1(\tilde X/X)$.
Thus the intersection pairing gives an isomorphism
\[
	N_1(\tilde X/X)\isomarrow\RR^n,\qquad
	\alpha\longmapsto(E_1\cdot\alpha,\dotsc,E_n\cdot\alpha).
\]
We therefore identify the relative numerical class $[C]$ of a $\pi$-contracted curve $C$ with its intersection vector $(C_1,\dotsc,C_n)$.

For $0\le i\le n$, let $C^{(i)}\subseteq X^{(i)}$ denote the image of $C$.
The source $i_0$ is the unique index for which $C^{(i_0)}$ is a curve and $C^{(i_0-1)}$ is a point.

\begin{defi}\label{def:elementary}
	In \cref{set:main}, a $\pi$-contracted curve $C\subseteq\tilde X$ with source $i_0$ is called an \emph{elementary curve} if $C_{i_0}=-1$ and, for each $i>i_0$,
	\[
		-\bigcard{\sset{h<i}{C_h=1}}\le C_i\le1.
	\]
\end{defi}

Geometrically, since $C_{i_0}=-1$, the induced map $C\to C^{(i_0)}$ is birational, and the curve $C^{(i_0)}$ is a line in a fiber of the blowup $\pi_{i_0}\colon X^{(i_0)}\to X^{(i_0-1)}$.
Furthermore, if $C_i<0$ with $i>i_0$, then the line $C^{(i_0)}$ is contained in $Z_i^{(i_0)}$.

In \cref{set:main}, for each $i$, let $\ell_i\in N_1(\tilde X/X)$ denote the class whose intersection vector has entry $-1$ in position $i$ and $0$ elsewhere.

\begin{lem}\label{lem:fiberclass}
	In \cref{set:main}, for every $i$, the class $\ell_i$ is represented by an elementary curve.
\end{lem}

\begin{proof}
	Choose a closed point $p\in Z_i$ outside every center that does not contain $Z_i$.
	Shrinking $X$ around $p$, we may assume that $Z_h=\emptyset$ for $h<i$ and that every nonempty $Z_h$ with $h>i$ contains $Z_i$.
	Thus $X^{(i-1)}=X$.
	By the simple normal crossings assumption, we may also choose \'etale coordinates $x_1,\dotsc,x_d$ on $X$ such that $\Ic_{Z_h}=(x_j\mid j\in J_h)$ for every nonempty $Z_h$.

	The generators $x_j$ for $j\in J_i$ give homogeneous coordinates $y_j$ for $j\in J_i$ on the fiber $\pi_i^{-1}(p)\isom\PP^{c_i-1}$.
	For each $h>i$ with $Z_h\ne\emptyset$, we have $J_h\subseteq J_i$ and
	\[
		Z_h^{(i)}\cap\pi_i^{-1}(p)=V(y_j\mid j\in J_h)\subseteq\pi_i^{-1}(p),
	\]
	so the intersection is a linear subspace of codimension $\card{J_h}=c_h\ge2$.

	A general line in $\pi_i^{-1}(p)$ avoids these finitely many linear subspaces, so all later blowups are isomorphisms near its strict transform $C\subseteq\tilde X$.
	We have $E_i\cdot C=-1$, and $E_h\cdot C=0$ for $h\ne i$ since $C$ is disjoint from $E_h$.
	Thus $C$ is an elementary curve of class $\ell_i$.
\end{proof}

The following example shows how a negative entry $C_i<0$ with $i>i_0$ arises from a section of a later exceptional divisor.

\begin{exa}\label{exa:forced}
	Let $X=\AA^4_{x_1,x_2,x_3,x_4}$ and
	\[
		Z_1=V(x_1,x_2),\qquad
		Z_2=V(x_1,x_3,x_4),\qquad
		Z_3=V(x_3,x_4).
	\]
	The exceptional fiber $C^{(1)}\coloneq\pi_1^{-1}(0)\isom\PP^1$ meets $Z_2^{(1)}$ in a single point and is contained in $Z_3^{(1)}$.
	After blowing up $Z_2^{(1)}$, the strict transform $C^{(2)}$ is still contained in $Z_3^{(2)}$, and
	\[
		N_{Z_3^{(2)}/X^{(2)}}\Bigr\rvert_{C^{(2)}}
		\isom\Ob{\PP^1}(-1)^{\oplus2}.
	\]
	Thus the third exceptional divisor admits a section $C\subseteq E_3$ over $C^{(2)}$ such that $E_3\cdot C=-1$.
	Its numerical class is
	\[
		[C]=(C_1,C_2,C_3)=(-1,1,-1).
	\]
	Hence $C$ is an elementary curve with source $1$.
\end{exa}

More generally, \cref{prop:realization} shows that every nonzero integer vector satisfying the numerical conditions in \cref{def:elementary} occurs as the intersection vector of an elementary curve spanning an extremal ray for a suitable choice of $X$ and centers $Z_1,\dotsc,Z_n$.

\subsection{Coordinate blowups and toric geometry}\label{ss:fan}

We next consider the coordinate model of \cref{set:main}, obtained by successively blowing up coordinate subspaces of affine space.
It will provide the basic local model for the relative geometry of $\tilde X\to X$.

\begin{setting}\label{set:coordinate}
	Let $d\ge2$ be an integer, and let $J_1,\dotsc,J_n\subseteq\{1,\dotsc,d\}$ be subsets satisfying
	\[
		\card{J_i}\ge2,\qquad
		J_h\not\subseteq J_i\quad(h<i).
	\]
	For $1\le i\le n$, set
	\[
		W_i\coloneq V(x_j\mid j\in J_i)\subseteq\AA^d.
	\]
	The subvarieties $W_1,\dotsc,W_n$ have simple normal crossings, and
	\[
		\codim_{\AA^d}W_i=\card{J_i}\ge2,\qquad
		W_h\not\supseteq W_i\quad(h<i).
	\]
	Thus they satisfy the assumptions of \cref{set:main}.

	Let
	\[
	\begin{tikzcd}
		\tilde Y\coloneq Y^{(n)}
			\arrow[r]
			&Y^{(n-1)}
			\arrow[r]
			&\dotsb
			\arrow[r]
			&Y^{(1)}
			\arrow[r]
			&Y^{(0)}\coloneq\AA^d
	\end{tikzcd}
	\]
	be the resulting sequence of blowups, and let $\pi_Y\colon\tilde Y\to\AA^d$ denote their composite.
	For each $i$, let $F_i$ denote the pullback to $\tilde Y$ of the exceptional divisor of $Y^{(i)}\to Y^{(i-1)}$.

	We now describe this construction torically.
	Let $N\coloneq\ZZ^d$ with standard basis $e_1,\dotsc,e_d$, and let $\Sigma=\Sigma^{(0)}$ be the fan of $\AA^d$, consisting of the faces of
	\[
		\Cone(e_1,\dotsc,e_d)\subseteq N_\RR.
	\]
	For $1\le i\le n$, set
	\[
		\sigma_i\coloneq\Cone(e_j\mid j\in J_i),\qquad
		v_i\coloneq\sum_{j\in J_i}e_j,\qquad
		\rho_{v_i}\coloneq\Cone(v_i).
	\]
	Let $\Sigma^{(i)}$ be the fan obtained from $\Sigma$ by successive star subdivisions at $v_1,\dotsc,v_i$.
	As we verify below, at the $i$-th step the cone $\sigma_i$ is still a cone of $\Sigma^{(i-1)}$, and the subdivision is precisely the star subdivision along $\sigma_i$.
	Set $\tilde\Sigma\coloneq\Sigma^{(n)}$.
\end{setting}

We first check that the cones corresponding to the blowup centers survive until the stages at which they are subdivided.

\begin{lem}\label{lem:conesurvive}
	In \cref{set:coordinate}, let $0\le h\le n$ and $1\le i\le n$.
	Then the following are equivalent:
	\begin{enumerate}
		\item\label{item:h<i} $h<i$,
		\item\label{item:sigma_i_in_Sigma^h} $\sigma_i\in\Sigma^{(h)}$, and
		\item\label{item:exists_sigma_containing_sigma_i} there exists a cone $\sigma\in\Sigma^{(h)}$ such that $\sigma_i\subseteq\sigma$.
	\end{enumerate}
\end{lem}

\begin{proof}
	Assume first that $h<i$.
	For every $h'\le h$, the assumption $J_{h'}\not\subseteq J_i$ gives $\sigma_{h'}\not\subseteq\sigma_i$.
	Hence none of the first $h$ star subdivisions subdivides $\sigma_i$.
	Therefore $\sigma_i\in\Sigma^{(h)}$, proving that
	\cref{item:h<i} implies \cref{item:sigma_i_in_Sigma^h}.

	The implication
	\cref{item:sigma_i_in_Sigma^h}$\Rightarrow$\cref{item:exists_sigma_containing_sigma_i} is immediate by taking $\sigma\coloneq\sigma_i$.

	Finally, assume that $h\ge i$.
	The $i$-th star subdivision introduces the ray $\rho_{v_i}$ through the relative interior of $\sigma_i$.
	Consequently, no cone of $\Sigma^{(i)}$ contains $\sigma_i$.
	The same remains true in every subsequent refinement $\Sigma^{(h)}$.
	Thus \cref{item:exists_sigma_containing_sigma_i} cannot hold.
	This proves
	\cref{item:exists_sigma_containing_sigma_i}$\Rightarrow$\cref{item:h<i}.
\end{proof}

We can now identify the successive blowup with the toric sequence of star subdivisions.

\begin{lem}\label{lem:validstar}
	In \cref{set:coordinate}, let $1\le i\le n$.
	\begin{enumerate}
		\item\label{item:Yi_toric}
			The variety $Y^{(i)}$ is the toric variety $X_{\Sigma^{(i)}}$ of the fan $\Sigma^{(i)}$.
		\item\label{item:pii_toric}
			The morphism $Y^{(i)}\to Y^{(i-1)}$ is the toric morphism associated with the star subdivision along $\sigma_i$.
		\item\label{item:Wi_toric}
			For $h>i$, the subvariety $W_h^{(i)}\subseteq Y^{(i)}$ is the orbit closure $V(\sigma_h)$ in $X_{\Sigma^{(i)}}$ corresponding to $\sigma_h\in\Sigma^{(i)}$.
		\item\label{item:Fi}
			On $\tilde Y=X_{\tilde\Sigma}$, the divisor $F_i$ is the torus-invariant divisor $D_{v_i}=V(\rho_{v_i})$.
	\end{enumerate}
\end{lem}

\begin{proof}
	We prove \cref{item:Yi_toric,item:pii_toric,item:Wi_toric} inductively on $i$.
	For $i=0$, we have $Y^{(0)}=X_\Sigma=\AA^d$ and $W_h^{(0)}=V(\sigma_h)$ for every $h$, which gives the initial step.
	Let $i\ge1$ and assume the identifications in \cref{item:Yi_toric,item:Wi_toric} at stage $i-1$.

	The blowup $Y^{(i)}$ of $Y^{(i-1)}=X_{\Sigma^{(i-1)}}$ along the smooth torus-invariant subvariety
	\[
		W_i^{(i-1)}=V(\sigma_i)
	\]
	is the toric morphism associated with the star subdivision along $\sigma_i$, which introduces the ray $\rho_{v_i}$.
	Thus $Y^{(i)}=X_{\Sigma^{(i)}}$, proving \cref{item:Yi_toric,item:pii_toric}.

	Let $h>i$.
	By \cref{lem:conesurvive}, we have $\sigma_h\in\Sigma^{(i-1)}$. Moreover, the assumption $J_i\not\subseteq J_h$ gives $\sigma_i\not\subseteq\sigma_h$.
	Hence the orbit $O(\sigma_h)$ is disjoint from the center $W_i^{(i-1)}=V(\sigma_i)$ of the $i$-th blowup.
	Therefore the $i$-th blowup is an isomorphism over $O(\sigma_h)$.
	Hence the strict transform of $W_h^{(i-1)}=V(\sigma_h)$ is the closure of the same orbit $O(\sigma_h)$ in $X_{\Sigma^{(i)}}$.
	Thus
	\[
		W_h^{(i)}=V(\sigma_h)\quad
		\text{in}\quad Y^{(i)}=X_{\Sigma^{(i)}}.
	\]
	This proves \cref{item:Wi_toric} and completes the induction.

	At the $i$-th step, the exceptional divisor is $D_{v_i}$.
	Since no later blowup center is contained in its strict transform, the pullback of $D_{v_i}$ through all subsequent blowups coincides with its strict transform.
	Therefore $F_i=D_{v_i}$ on $\tilde Y=X_{\tilde\Sigma}$.
\end{proof}

In \cref{set:coordinate}, for each coordinate $j\in\{1,\dotsc,d\}$, set
\[
	I_j\coloneq\sset{i\in\{1,\dotsc,n\}}{j\in J_i},
\]
which records precisely which blowup centers involve the coordinate $x_j$.

The rays of $\tilde\Sigma$ are precisely the coordinate rays $\rho_{e_j}$ and the exceptional rays $\rho_{v_i}$.
Let $D_{e_j}$ and $D_{v_i}=F_i$ denote the corresponding invariant divisors, respectively.
In particular, $D_{e_j}$ is the strict transform of the coordinate hyperplane $V(x_j)\subseteq\AA^d$.

\begin{lem}\label{lem:D_e_j}
	In \cref{set:coordinate}, let $j\in\{1,\dotsc,d\}$.
	Then the following linear equivalence holds on $\tilde Y$:
	\[
		D_{e_j}
		\lineq-\sum_{i\in I_j}F_i.
	\]
\end{lem}

\begin{proof}
	Let $M\coloneq\Hom(N,\ZZ)$, and let $m_1,\dotsc,m_d$ be the basis dual to $e_1,\dotsc,e_d$.
	The toric principal divisor formula gives
	\begin{align*}
		\Div_{\tilde Y}(\chi^{m_j})
		&=\sum_{\rho\in\tilde\Sigma(1)}\innprod{m_j}{u_\rho}D_\rho\\
		&=\sum_{l=1}^d\innprod{m_j}{e_l}D_{e_l}+\sum_{i=1}^n\innprod{m_j}{v_i}D_{v_i}\\
		&=D_{e_j}+\sum_{i\in I_j}F_i.
	\end{align*}
	Since this divisor is principal, the assertion follows.
\end{proof}

\section{Local coordinate models}\label{sec:model}

In this section, we construct local coordinate models for the successive blowup $\pi\colon\tilde X\to X$ using successive blowups along coordinate subspaces of affine space.
We show that, locally near the centers, the original blowup sequence is obtained from these models by \'etale base change.
These models allow us to apply toric geometry to the study of the relative cone of curves and its contractions.

\begin{defi}\label{def:coordmodel}
	In \cref{set:main}, let $p\in\bigcup_i Z_i$ be a closed point.
	Choose coordinates $x_1,\dotsc,x_d$ at $p$ as in \cref{def:snc}, so that
	\[
		\Ic_{Z_i,p}=(x_j\mid j\in J_i)
	\]
	for every $Z_i$ containing $p$.
	After extending the $x_j$ to functions on an open neighborhood $U_p$ of $p$ and shrinking $U_p$ if necessary, we may assume that
	\[
		\phi_p\coloneq(x_1,\dotsc,x_d)\colon U_p\longrightarrow\AA^d,
		\qquad \phi_p(p)=0,
	\]
	is \'etale, that
	\[
		Z_i\cap U_p=\phi_p^{-1}(V(x_j\mid j\in J_i))
	\]
	for every $Z_i$ containing $p$, and that $U_p$ meets no $Z_i$ not containing $p$.

	For each $i$, set
	\[
		W_i\coloneq
		\begin{cases}
			V(x_j\mid j\in J_i)\subseteq\AA^d & p\in Z_i,\\
			\emptyset & p\notin Z_i.
		\end{cases}
	\]
	Thus $Z_i\cap U_p=\phi_p^{-1}(W_i)$ for every $i$.
	The $W_i$ define a sequence of blowups
	\[
		Y^{(0)}\coloneq\AA^d,\qquad
		Y^{(i)}\coloneq\Bl_{W_i^{(i-1)}}Y^{(i-1)},\qquad
		\tilde Y_p\coloneq Y^{(n)},
	\]
	where the stages with $W_i=\emptyset$ are trivial.
	Let $F_i$ denote the pullback to $\tilde Y_p$ of the exceptional divisor of the $i$-th blowup, with the convention that $F_i=0$ if $W_i=\emptyset$.

	We call
	\[
		\pi_{Y_p}\colon\tilde Y_p\longrightarrow\AA^d
		\quad\text{together with}\quad
		\phi_p\colon U_p\longrightarrow\AA^d
	\]
	a \emph{coordinate model} of $\tilde X\to X$ at $p$, with \emph{exceptional divisors} $F_1,\dotsc,F_n$.
\end{defi}

Note that in \cref{def:coordmodel}, the nonempty $W_i$ satisfy the assumptions of \cref{set:coordinate}.

The next proposition shows that $\tilde X\to X$ is locally obtained from its coordinate models by \'etale base change.

\begin{prop}\label{prop:localmodel}
	In \cref{set:main}, let $p\in\bigcup_iZ_i$ be a closed point, and let $\pi_{Y_p}\colon\tilde Y_p\to\AA^d$ together with $\phi_p\colon U_p\to\AA^d$ be a coordinate model of $\pi\colon\tilde X\to X$ at $p$, with exceptional divisors $F_1,\dotsc,F_n$.
	Let $\tilde U_p\coloneq\pi^{-1}(U_p)$, the successive blowup of $U_p$ along $Z_1\cap U_p,\dotsc,Z_n\cap U_p$.
	\begin{enumerate}
		\item\label{item:fiberproduct}
			There is a cartesian diagram
			\[
			\begin{tikzcd}
				\tilde U_p\arrow[r,"\tilde\phi_p"]\arrow[d,"\pi"']
					& \tilde Y_p\arrow[d,"\pi_{Y_p}"]\\
				U_p\arrow[r,"\phi_p"']
					& \AA^d
			\end{tikzcd}
			\]
			In particular, $\tilde\phi_p$ induces an isomorphism $\pi^{-1}(p)\isomarrow\pi_{Y_p}^{-1}(0)$.
		\item\label{item:exceptionaldivisors}
			$\restrdiv{E_i}{\tilde U_p}=\tilde\phi_p^{\,*}F_i$ for $1\le i\le n$.
		\item\label{item:NEcompare}
			The diagram in \cref{item:fiberproduct} gives a commutative diagram
			\[
			\begin{tikzcd}
				& N_1(\pi^{-1}(p))\arrow[r,"\sim"]\arrow[d]
					& N_1(\pi_{Y_p}^{-1}(0))\arrow[d]\\
				N_1(\tilde X/X)
					& N_1(\tilde U_p/U_p)\arrow[l,hook']\arrow[r,"(\tilde\phi_p)_*"',"\sim"]
					& N_1(\tilde Y_p/\AA^d)
			\end{tikzcd}
			\]
			in which $(\tilde\phi_p)_*$ is an isomorphism.
		\item\label{item:elementaryequiv}
			Let $C\subseteq\pi^{-1}(p)$ be a curve, and let $C_Y\coloneq\tilde\phi_p(C)\subseteq\pi_{Y_p}^{-1}(0)$.
			Then $C$ is an elementary curve of $\pi\colon\tilde X\to X$ if and only if $C_Y$ is an elementary curve of $\pi_{Y_p}\colon\tilde Y_p\to\AA^d$.
			Here elementary curves on $\tilde Y_p$ are defined with respect to $(F_1,\dotsc,F_n)$.
	\end{enumerate}
\end{prop}

\begin{proof}
	Let $U^{(i)}$ be the $i$-th blowup of $U_p$, and let $W^{(i)}_h$ ($i<h$) be the strict transform of $W_h$ on $Y^{(i)}$.
	We prove by induction on $i$ that there is a cartesian diagram
	\[
	\begin{tikzcd}
		U^{(i)}\arrow[r,"\phi_i"]\arrow[d]
			& Y^{(i)}\arrow[d]\\
		U_p\arrow[r,"\phi_p"']
			& \AA^d
	\end{tikzcd}
		\]
	and that $Z_h^{(i)}\cap U^{(i)}=\phi_i^{-1}(W_h^{(i)})$ for every $h>i$.

	For $i=0$, the statements are the choice of $U_p$ and $\phi_p$.
	Let $i\ge1$.
	Since blowups commute with flat base change, we obtain a cartesian diagram
	\[
	\begin{tikzcd}
		U^{(i)}\arrow[d]\arrow[r,"\phi_i"]
			& Y^{(i)}\arrow[d]\\
		U^{(i-1)}\arrow[r,"\phi_{i-1}"']
			& Y^{(i-1)}
	\end{tikzcd}
	\]
	from the induction hypothesis.
	Moreover, since strict transforms commute with flat base change, the equality
	\[
		Z_h^{(i)}\cap U^{(i)}=\phi_i^{-1}(W_h^{(i)})\qquad(h>i)
	\]
	follows as well.

	Letting $\tilde\phi_p\coloneq\phi_n$, we now have a cartesian diagram
	\[
	\begin{tikzcd}
		\tilde U_p\arrow[r,"\tilde\phi_p"]\arrow[d]
			& \tilde Y_p\arrow[d,"\pi_{Y_p}"]\\
		U_p\arrow[r,"\phi_p"']
			& \AA^d
	\end{tikzcd}
	\]
	and, in particular, an isomorphism $\pi^{-1}(p)\isomarrow\pi_{Y_p}^{-1}(0)$.
	This proves \cref{item:fiberproduct}.

	For \cref{item:exceptionaldivisors}, first assume $W_i\ne\emptyset$.
	Let $F_i^{(i)}\subseteq Y^{(i)}$ be the exceptional divisor of the $i$-th blowup.
	The base-change diagram at the $i$-th stage identifies the two exceptional divisors, so $\restrdiv{E_i^{(i)}}{U^{(i)}}=\phi_i^*F^{(i)}_i$ on $U^{(i)}$.
	Pulling back to $\tilde U_p$ gives $\restrdiv{E_i}{\tilde U_p}=\tilde\phi_p^*F_i$.
	If $W_i=\emptyset$, then $Z_i\cap U_p=\emptyset$, and hence $\restrdiv{E_i}{\tilde U_p}=0$.
	Thus the desired equality follows from $F_i=0$.

	The diagram of \cref{item:NEcompare} follows from the diagram in \cref{item:fiberproduct}.
	By \cref{item:exceptionaldivisors}, the pullback
	\[
		\tilde\phi_p^*\colon N^1(\tilde Y_p/\AA^d)\longrightarrow N^1(\tilde U_p/U_p)
	\]
	is an isomorphism: it sends the basis given by nonzero $[F_i]$ to the corresponding basis given by $\restrdiv{E_i}{\tilde U_p}$.
	Hence the map $(\tilde\phi_p)_*\colon N_1(\tilde U_p/U_p)\to N_1(\tilde Y_p/\AA^d)$ is an isomorphism.

	Finally, for every $i$, \cref{item:exceptionaldivisors} and the isomorphism $C\isomarrow C_Y$ give $E_i\cdot C=F_i\cdot C_Y$.
	Thus $C$ and $C_Y$ have the same intersection vector.
	By \cref{def:elementary}, $C$ is elementary if and only if $C_Y$ is elementary, proving \cref{item:elementaryequiv}.
\end{proof}

\section{Elementary generators of the relative cone of curves}\label{sec:cone}

This section proves \cref{introthm:cone}, which states that $\NEb(\tilde X/X)$ is generated by finitely many elementary curve classes, and bounds the number of extremal rays (\cref{ss:conegeneration}).
We then give examples of extremal rays, including a construction realizing every vector allowed by the definition of an elementary curve (\cref{sec:rays}).

\subsection{Generation by elementary curves}\label{ss:conegeneration}

We prove that the relative cone of curves is generated by finitely many elementary curve classes.
We first treat coordinate blowups, and then pass to the general setting using the coordinate models of \cref{sec:model}.
We begin by recalling the relevant facts from toric geometry.

Let $N$ be a lattice of rank $d$, and let $\Sigma$ be a smooth fan in $N_\RR$ whose support $\sigma_0\coloneq\abs{\Sigma}$ is a full-dimensional strongly convex cone.
Since $\Sigma$ subdivides $\sigma_0$, the identity map on $N$ induces a proper toric morphism
\[
	\pi\colon X_\Sigma\longrightarrow X_{\sigma_0}.
\]

A \emph{wall} of $\Sigma$ is a cone $\tau\in\Sigma(d-1)$ contained in exactly two maximal cones, say $\sigma_+$ and $\sigma_-$.
Since $\Sigma$ is smooth, there are rays $\rho_+,\rho_-\notin\tau(1)$ such that $\sigma_\pm=\tau+\rho_\pm$.
Then $\pi$ maps the orbit closure $V(\tau)\subseteq X_\Sigma$ to the torus-fixed point $V(\sigma_0)\subseteq X_{\sigma_0}$, since $\tau$ is not contained in a proper face of $\sigma_0$.
We call this proper curve $V(\tau)$ the \emph{wall curve} associated with $\tau$.

For each ray $\rho\in\Sigma(1)$, let $u_\rho\in N$ be its primitive generator and let $D_\rho\subseteq X_\Sigma$ be the corresponding torus-invariant prime divisor.
Set $b_\rho\coloneq D_\rho\cdot V(\tau)$.
Since $\Sigma$ is smooth,
\[
	b_{\rho_+}=b_{\rho_-}=1,\qquad
	b_\rho=0\quad\text{if }\rho\notin\tau(1)\cup\{\rho_+,\rho_-\}.
\]
Moreover, the class $[V(\tau)]\in N_1(X_\Sigma/X_{\sigma_0})$ is represented by the \emph{wall relation}
\begin{equation}\label{eq:wallrelation}
	u_{\rho_+}+u_{\rho_-}+\sum_{\rho\in\tau(1)}b_\rho u_\rho=0
	\quad\text{in }N.
\end{equation}

A \emph{primitive collection} of $\Sigma$ is a subset $P\subseteq\Sigma(1)$ that is not contained in $\sigma(1)$ for any $\sigma\in\Sigma$, while every proper subset of $P$ is contained in $\sigma(1)$ for some $\sigma\in\Sigma$.
Since $\sigma_0$ is convex, the lattice point
\[
	s_P\coloneq\sum_{\rho\in P}u_\rho
\]
belongs to $\sigma_0$.
Let $\gamma(P)$ be the unique cone of $\Sigma$ whose relative interior contains $s_P$.
Since $\Sigma$ is smooth, there are unique positive integers $c_\rho$ such that
\[
	s_P=\sum_{\rho\in\gamma(P)(1)}c_\rho u_\rho.
\]
Set $c_\rho=0$ for $\rho\notin\gamma(P)(1)$.
The resulting lattice relation defines the \emph{primitive relation} $r(P)\in N_1(X_\Sigma/X_{\sigma_0})$, characterized by
\[
	D_\rho\cdot r(P)
	=\mathbf 1_{\rho\in P}-c_\rho
	=\begin{cases}
		1 & \rho\in P,\;\rho\notin\gamma(P)(1),\\
		1-c_\rho & \rho\in P\cap\gamma(P)(1),\\
		-c_\rho & \rho\in\gamma(P)(1),\;\rho\notin P,\\
		0 & \text{otherwise}.
	\end{cases}
\]

We use the following relative toric cone theorem.

\begin{thm}[Cox--von Renesse {\cite{CvR}}]\label{thm:toriccone}
	Let $\Sigma$ be a smooth fan in $N_\RR$ whose support $\sigma_0$ is a full-dimensional strongly convex cone.
	\begin{enumerate}
		\item\label{item:wallcurves}
			The relative cone of curves is generated by the classes of the wall curves:
			\[
				\NEb(X_\Sigma/X_{\sigma_0})
				=\sum_{\tau\text{ a wall of }\Sigma}\RRp[V(\tau)].
			\]
		\item\label{item:primitive}
			If $X_\Sigma$ is quasi-projective, then the relative cone of curves is generated by the primitive relations:
			\[
				\NEb(X_\Sigma/X_{\sigma_0})
				=\sum_{P\text{ a primitive collection of }\Sigma}\RRp\,r(P).
			\]
	\end{enumerate}
\end{thm}

\begin{proof}
	Part \cref{item:wallcurves} is the relative toric cone theorem \cite[Theorem~1.6 and Proposition~1.11]{CvR}.
	Part \cref{item:primitive} follows from \cite[Theorem~1.4 and Proposition~1.10]{CvR}.
\end{proof}

To apply this theorem to successive blowups, we record how wall curves and their intersections with the exceptional divisor behave under a star subdivision.

\begin{lem}\label{lem:wallrec}
	Let $\Sigma$ be a smooth full-dimensional fan in $N_\RR\isom\RR^d$, and let $\sigma\in\Sigma$ be a cone of dimension at least two.
	Let $\Sigma^*$ be the star subdivision of $\Sigma$ along $\sigma$.
	The induced toric morphism
	\[
		\pi\colon X_{\Sigma^*}\longrightarrow X_\Sigma
	\]
	is the blowup along $V_\Sigma(\sigma)$.
	Let $F\subseteq X_{\Sigma^*}$ be its exceptional divisor.

	Let $\tau$ be a wall of $\Sigma^*$ and set $C\coloneq V_{\Sigma^*}(\tau)$.
	Then exactly one of the following occurs.
	\begin{enumerate}
		\item\label{item:curveimage}
			The image $C'\coloneq\pi(C)$ is a wall curve of $X_\Sigma$, and the induced morphism $C\to C'$ is an isomorphism.
			Moreover, at least one of the following holds:
			\begin{itemize}
				\item $F\cdot C=0$, or
				\item $F\cdot C=D_\rho\cdot C'$ for some ray $\rho\in\sigma(1)$.
			\end{itemize}
		\item\label{item:vertical}
			The morphism $\pi$ contracts $C$ to a point, and $F\cdot C=-1$.
	\end{enumerate}
\end{lem}

\begin{proof}
	Let $r\coloneq\dim\sigma$, and let $u_1,\dotsc,u_r$ be the primitive ray generators of $\sigma$.
	Set $v\coloneq u_1+\dotsb+u_r$ and $\rho_v\coloneq\Cone(v)$.
	The subdivision replaces each maximal cone $\mu$ containing $\sigma$ by the cones obtained by replacing one of $u_1,\dotsc,u_r$ with $v$.
	The exceptional divisor is $F=V_{\Sigma^*}(\rho_v)$.

	Let $\bar\tau$ be the smallest cone of $\Sigma$ containing $\tau$.
	Since $\dim\tau=d-1$, we have $\dim\bar\tau\in\{d-1,d\}$.
	The restriction $O(\tau)\to O(\bar\tau)$ of $\pi$ to the dense orbits is induced by the quotient map
	\[
		N/(N\cap\operatorname{span}_\RR(\tau))
		\longrightarrow N/(N\cap\operatorname{span}_\RR(\bar\tau)).
	\]
	Taking closures and using the properness of $\pi$, we obtain $\pi(C)=V_\Sigma(\bar\tau)$.
	If $\dim\bar\tau=d-1$, then $\tau$ and $\bar\tau$ span the same subspace, so this quotient map is an isomorphism.
	Set $C'\coloneq V_\Sigma(\bar\tau)$.
	Since $C\to C'$ is a proper birational morphism of smooth curves, it is an isomorphism.
	In particular, $C'$ is a wall curve.
	If $\dim\bar\tau=d$, then $C$ is contracted to the fixed point $V_\Sigma(\bar\tau)$.

	We now compute $F\cdot C$.
	First assume that $\rho_v\notin\tau(1)$.
	Then $\tau\in\Sigma$, $\bar\tau=\tau$, and $\sigma\not\subseteq\tau$.
	If neither maximal cone adjacent to $\tau$ in $\Sigma$ contains $\sigma$, then $F\cap C=\emptyset$, so $F\cdot C=0$.
	Otherwise, exactly one adjacent cone $\mu$ contains $\sigma$, since the two adjacent cones intersect along $\tau$.
	Let $\rho$ be the unique ray of $\mu$ outside $\tau$.
	Then $\rho\in\sigma(1)$, and the cone of $\Sigma^*$ adjacent to $\tau$ inside $\mu$ is obtained by replacing $\rho$ with $\rho_v$.
	The toric intersection formula gives
	\[
		F\cdot C=1=D_\rho\cdot C'.
	\]
	This proves \cref{item:curveimage} in this case.

	Now assume that $\rho_v\in\tau(1)$.
	Since $v\in\relint(\sigma)$ and $\bar\tau$ contains $v$, we have $\sigma\subseteq\bar\tau$.
	Assume first that $\dim\bar\tau=d-1$.
	The description of the subdivision gives
	\[
		\tau(1)=\{\rho_v\}\cup(\bar\tau(1)\setminus\{\rho_j\})
	\]
	for some $\rho_j=\Cone(u_j)\in\sigma(1)$.
	Substituting
	\[
		u_j=v-\sum_{1\le l\le r,\;l\ne j}u_l
	\]
	into the normalized wall relation of $\bar\tau$ gives the normalized wall relation of $\tau$.
	Indeed, the two rays outside the wall are unchanged, and their coefficients remain $1$.
	The coefficient of $v$ in the new relation equals the coefficient of $u_j$ in the original relation.
	Hence \eqref{eq:wallrelation} gives
	\[
		F\cdot C=D_{\rho_j}\cdot C',
	\]
	proving \cref{item:curveimage}.

	It remains to consider the case $\dim\bar\tau=d$.
	The two maximal cones adjacent to $\tau$ are obtained from $\bar\tau$ by replacing $u_{j_1}$ and $u_{j_2}$, respectively, with $v$, for distinct $j_1,j_2\in\{1,\dotsc,r\}$.
	Thus the two rays outside $\tau$ are generated by $u_{j_1}$ and $u_{j_2}$, and the normalized wall relation is
	\[
		u_{j_1}+u_{j_2}-v
		+\sum_{1\le l\le r,\;l\ne j_1,j_2}u_l=0.
	\]
	By \eqref{eq:wallrelation}, we obtain $F\cdot C=-1$.
	This proves \cref{item:vertical}.
\end{proof}

We now work in \cref{set:coordinate}.
This is a special case of \cref{set:main}, so the notation and terminology introduced in \cref{ss:tower} apply to the blowup
\[
	\pi_Y\colon\tilde Y\longrightarrow\AA^d.
\]
Thus, for a $\pi_Y$-contracted curve $C\subseteq\tilde Y$, we write $C_i=F_i\cdot C$ and use the terms \emph{source} and \emph{elementary curve} in the sense of \cref{ss:tower}.

We first apply \cref{lem:wallrec} to describe the intersection vectors of wall curves.
We then obtain numerical bounds when the curve class is a positive multiple of a primitive relation.

\begin{lem}\label{lem:wallcurves}
	In \cref{set:coordinate}, let $C\subseteq\tilde Y$ be a wall curve with source $i_0$.
	\begin{enumerate}
		\item\label{item:source}
			$C_{i_0}=-1$.
		\item\label{item:wallentry}
			For every $i>i_0$ with $C_i\ne0$, there exists $j\in J_i$ such that
			\[
				C_i=-\sum_{h<i,\;h\in I_j}C_h.
			\]
	\end{enumerate}
\end{lem}

\begin{proof}
	For $0\le i\le n$, let $C^{(i)}\subseteq Y^{(i)}$ be the image of $C$.
	For $1\le h\le i\le n$, let $F_h^{(i)}$ denote the pullback to $Y^{(i)}$ of the exceptional divisor of the $h$-th blowup, and, when $i\ge i_0$, set $C_h^{(i)}\coloneq F_h^{(i)}\cdot C^{(i)}$.

	We first prove \cref{item:source} by induction on $n$.
	The fan $\tilde\Sigma=\Sigma^{(n)}$ is obtained from $\Sigma^{(n-1)}$ by the star subdivision along $\sigma_n$.
	If $i_0<n$, then $C$ is not contracted by $\tilde Y\to Y^{(n-1)}$.
	Hence, by \cref{lem:wallrec}, $C^{(n-1)}$ is a wall curve on $Y^{(n-1)}$, and the induced morphism $C\to C^{(n-1)}$ is an isomorphism.
	Its source is again $i_0$, and the induction hypothesis gives $C_{i_0}^{(n-1)}=-1$.
	The projection formula then yields
	\[
		C_{i_0}=C_{i_0}^{(n-1)}=-1.
	\]
	If $i_0=n$, then $C$ is contracted by $\tilde Y\to Y^{(n-1)}$, so
	\[
		C_n=D_{v_n}\cdot C=-1
	\]
	by \cref{lem:wallrec}~\cref{item:vertical}.
	This also establishes the initial case $n=1$.

	We now prove \cref{item:wallentry}.
	Fix $i>i_0$ with $C_i\ne0$.
	Repeated application of \cref{lem:wallrec} shows that
	$C^{(i)}$ and $C^{(i-1)}$ are wall curves on $Y^{(i)}$ and
	$Y^{(i-1)}$, respectively, and that the induced morphisms $C\to C^{(i)}\to C^{(i-1)}$ are isomorphisms.
	The projection formula gives $C_h^{(i-1)}=C_h$ for every $h<i$.
	Apply \cref{lem:wallrec} to the morphism $Y^{(i)}\to Y^{(i-1)}$.
	Since $C_i\ne0$, there exists a ray $\rho\in\sigma_i(1)$ such that
	\[
		C_i
		=F_i^{(i)}\cdot C^{(i)}
		=D_\rho\cdot C^{(i-1)}.
	\]
	As $\sigma_i(1)=\sset{\rho_{e_j}}{j\in J_i}$, write $\rho=\rho_{e_j}$ for some $j\in J_i$.
	Applying \cref{lem:D_e_j} to the truncated sequence of blowups $Y^{(i-1)}\to\AA^d$, we obtain
	\[
		C_i=D_{e_j}\cdot C^{(i-1)}
		=-\sum_{h<i,\;h\in I_j}C_h.
	\]
	This proves \cref{item:wallentry}.
\end{proof}

\begin{lem}\label{lem:vennlocal}
	In \cref{set:coordinate}, let $C$ be a wall curve on $\tilde Y$, and let $P$ be a primitive collection such that $r(P)\in\RRp[C]$.
	Then
	\begin{enumerate}
		\item\label{item:vlclass} $[C]=r(P)$,
		\item\label{item:vlmem}
			for every ray $\rho\in\tilde\Sigma(1)$, the inequality $D_\rho\cdot C\le1$ holds, and if $D_\rho\cdot C=1$, then $\rho\in P$, and
		\item\label{item:vlnorm} in particular, the following inequalities hold:
			\[
				C_i\le1\quad(1\le i\le n),\qquad
				\sum_{i\in I_j}C_i\ge-1\quad(1\le j\le d).
			\]
	\end{enumerate}
\end{lem}

\begin{proof}
	Write $r(P)=\lambda[C]$ for some $\lambda>0$.
	Let $\tau$ be the wall such that $C=V(\tau)$, and let $\rho_+$ be one of the two rays adjacent to $\tau$ but not contained in it.
	Then $D_{\rho_+}\cdot C=1$.
	Consequently,
	\[
		0<\lambda
		=\lambda(D_{\rho_+}\cdot C)
		=D_{\rho_+}\cdot r(P)
		=\mathbf 1_{\rho_+\in P}-c_{\rho_+}
		\le1.
	\]
	Since $c_{\rho_+}$ is an integer, this gives $\lambda=1$ and proves \cref{item:vlclass}.

	It follows that, for every ray $\rho\in\tilde\Sigma(1)$,
	\[
		D_\rho\cdot C
		=\mathbf 1_{\rho\in P}-c_\rho
		\le1.
	\]
	Taking $\rho=\rho_{v_i}$ gives $C_i=D_{v_i}\cdot C\le1$.
	Similarly, taking $\rho=\rho_{e_j}$ and applying \cref{lem:D_e_j}, we obtain $-\sum_{i\in I_j}C_i=D_{e_j}\cdot C\le1$.
	This proves \cref{item:vlnorm}.

	Finally, assume that $D_\rho\cdot C=1$.
	Then
	\[
		1=\mathbf 1_{\rho\in P}-c_\rho
		\le\mathbf 1_{\rho\in P}
		\le1.
	\]
	Hence $\mathbf 1_{\rho\in P}=1$ and $c_\rho=0$, so $\rho\in P$. This proves \cref{item:vlmem}.
\end{proof}

We now prove the coordinate-blowup case of \cref{introthm:cone} by applying \cref{thm:toriccone} and the preceding lemmas to each extremal ray.

\begin{prop}\label{prop:coordinate}
	In \cref{set:coordinate}, the relative cone of curves $\NEb(\tilde Y/\AA^d)$ is generated by the classes of finitely many elementary curves, all contained in the fiber $\pi_Y^{-1}(0)$ over the origin.
\end{prop}

\begin{proof}
	By \cref{lem:validstar}, we have $\tilde Y=X_{\tilde\Sigma}$ and $F_i=D_{v_i}$.
	The fan $\tilde\Sigma$ is smooth because $\tilde Y$ is smooth by \cref{lem:inherit}.
	Moreover, $\abs{\tilde\Sigma}=\Cone(e_1,\dotsc,e_d)$, which is a full-dimensional strongly convex cone, since star subdivisions do not change the support.
	Finally, $\pi_Y\colon\tilde Y\to\AA^d$ is projective, being a composite of blowups.
	In particular, $\tilde Y$ is quasi-projective.

	We may therefore apply \cref{thm:toriccone}.
	The first part of the theorem shows that $\NEb(\tilde Y/\AA^d)$ is generated by finitely many wall classes.
	Moreover, since $\pi_Y$ is projective, this cone is strongly convex.
	It is therefore generated by its finitely many extremal rays.

	Let $R\subseteq\NEb(\tilde Y/\AA^d)$ be an extremal ray.
	Since the cone is generated both by wall classes and by primitive relations, choose a wall curve $C$ and a primitive collection $P$ of $\tilde\Sigma$ such that $[C]\in R$ and $r(P)\in R$.

	Because $C$ is torus-invariant and $\pi_Y$ is torus-equivariant, the point $\pi_Y(C)$ is fixed by the torus and hence is the origin.
	Thus $C\subseteq\pi_Y^{-1}(0)$.

	It remains to show that $C$ is elementary.
	Let $i_0$ be its source.
	By \cref{lem:wallcurves,lem:vennlocal}, we have $C_{i_0}=-1$ and $C_i\le1$ for every $i$.
	Fix $i>i_0$.
	If $C_i\ge0$, then the lower bound in \cref{def:elementary} is immediate.
	If $C_i<0$, then \cref{lem:wallcurves}~\cref{item:wallentry} gives a coordinate $j\in J_i$ such that
	\[
		C_i=-\sum_{h<i,\;h\in I_j}C_h
		\ge-\sum_{h<i,\;C_h>0}C_h
		=-\bigcard{\sset{h<i}{C_h=1}}.
	\]
	Hence $C$ satisfies all the defining conditions of an elementary curve.

	We have therefore found on each extremal ray the class of an elementary curve contained in $\pi_Y^{-1}(0)$.
	Hence the result follows.
\end{proof}

We now return to \cref{set:main} and use the coordinate models of \cref{sec:model} to prove \cref{introthm:cone}.

\begin{lem}\label{lem:NEinj}
	In \cref{set:main}, let $p\in\bigcup_iZ_i$ be a closed point, and let $\pi_{Y_p}\colon\tilde Y_p\to\AA^d$ be a coordinate model of $\pi\colon\tilde X\to X$ at $p$.
	Then the injective linear map
	\[
	\begin{tikzcd}
		N_1(\tilde Y_p/\AA^d)\arrow[r,hook]
			& N_1(\tilde X/X)
	\end{tikzcd}
	\]
	induced by \cref{prop:localmodel}~\cref{item:NEcompare} restricts to an injection of cones, denoted by $\iota_p$, which fits into the following commutative diagram:
	\[
	\begin{tikzcd}
		\NEb(\pi^{-1}(p))\arrow[r,"\sim"]\arrow[d]
			& \NEb(\pi_{Y_p}^{-1}(0))\arrow[d,two heads]\\
		\NEb(\tilde X/X)
			& \NEb(\tilde Y_p/\AA^d)\arrow[l,hook',"\iota_p"]
	\end{tikzcd}
	\]
	The upper horizontal arrow is induced by the fiber isomorphism $\pi^{-1}(p)\isomarrow\pi_{Y_p}^{-1}(0)$, and the vertical arrows are the pushforwards of the inclusions of the fibers.
\end{lem}

\begin{proof}
	The diagram in \cref{prop:localmodel}~\cref{item:NEcompare} restricts to the following commutative diagram of cones:
	\[
	\begin{tikzcd}
		\NEb(\pi^{-1}(p))\arrow[r,equal]\arrow[d]
			& \NEb(\pi^{-1}(p))\arrow[r,"\sim"]\arrow[d]
			& \NEb(\pi_{Y_p}^{-1}(0))\arrow[d]\\
		\NEb(\tilde X/X)
			& \NEb(\tilde U_p/U_p)\arrow[r,hook,"(\tilde\phi_p)_*"']\arrow[l,hook']
			& \NEb(\tilde Y_p/\AA^d)
	\end{tikzcd}
	\]
	After omitting the empty stages, \cref{prop:coordinate} shows that $\NEb(\tilde Y_p/\AA^d)$ is generated by classes of curves contained in $\pi_{Y_p}^{-1}(0)$.
	Hence the map $\NEb(\pi_{Y_p}^{-1}(0))\to\NEb(\tilde Y_p/\AA^d)$ is surjective.
	The commutative diagram then shows that the injection
	\[
		(\tilde\phi_p)_*\colon\NEb(\tilde U_p/U_p)\longrightarrow\NEb(\tilde Y_p/\AA^d)
	\]
	is also surjective, and hence an isomorphism.
	Thus the composite map
	\[
	\begin{tikzcd}[column sep=large]
		\iota_p\colon\NEb(\tilde Y_p/\AA^d)\arrow[r,"(\tilde\phi_p)_*^{-1}","\sim"']
			& \NEb(\tilde U_p/U_p)\arrow[r,hook]
			& \NEb(\tilde X/X)
	\end{tikzcd}
	\]
	is the required injection.
\end{proof}

\begin{thm}\label{thm:cone}
	In \cref{set:main}, the relative cone of curves $\NEb(\tilde X/X)$ is generated by the classes of finitely many elementary curves.
\end{thm}

\begin{proof}
	Let $K\subseteq N_1(\tilde X/X)$ be the cone generated by the classes of all elementary curves.
	By definition, elementary curves have only finitely many possible integral intersection vectors, each of which determines a numerical class in $N_1(\tilde X/X)$.
	Hence $K$ is finitely generated and therefore closed.

	For each closed point $p\in\bigcup_iZ_i$, choose a coordinate model $\pi_{Y_p}\colon\tilde Y_p\to\AA^d$.
	Consider the diagram of \cref{lem:NEinj}:
	\[
	\begin{tikzcd}
		\NEb(\pi^{-1}(p))\arrow[r,"\sim"]\arrow[d]
			& \NEb(\pi_{Y_p}^{-1}(0))\arrow[d,two heads]\\
		\NEb(\tilde X/X)
			& \NEb(\tilde Y_p/\AA^d)\arrow[l,hook',"\iota_p"]
	\end{tikzcd}
	\]
	After omitting the empty stages, \cref{prop:coordinate} shows that the cone $\NEb(\tilde Y_p/\AA^d)$ is generated by the classes of elementary curves contained in $\pi_{Y_p}^{-1}(0)$.
	It follows that the image of $\iota_p$ is contained in $K$.
	By commutativity, the image of $\NEb(\pi^{-1}(p))$ in $\NEb(\tilde X/X)$ is also contained in $K$.

	Every curve contracted by $\pi$ is contained in $\pi^{-1}(p)$ for some closed point $p\in\bigcup_iZ_i$.
	Thus every effective relative curve class belongs to $K$.
	Since $K$ is closed, it follows that $\NEb(\tilde X/X)\subseteq K$, and hence $\NEb(\tilde X/X)=K$.
\end{proof}

We next bound the number of extremal rays in terms of the number of centers.
For $n\ge0$, let $B_n$ denote the \emph{$n$-th Bell number}, that is, the number of partitions of an $n$-element set, with $B_0=1$.
These numbers satisfy the recurrence
\[
	B_{n+1}=\sum_{i=0}^n\binom{n}{i}B_i\qquad(n\ge0).
\]
Indeed, one may first choose the $i$ elements outside the block containing a fixed element and then partition those $i$ elements.

\begin{prop}\label{prop:raycount}
	In \cref{set:main}, let $r(\tilde X/X)$ be the number of extremal rays of $\NEb(\tilde X/X)$.
	Then
	\[
		n\le r(\tilde X/X)\le 1+\sum_{i=1}^{n-1}B_i.
	\]
	In particular, $r(\tilde X/X)\le2(n-1)!$.
\end{prop}

\begin{proof}
	Since $\NEb(\tilde X/X)$ has dimension $n$, it has at least $n$ extremal rays.

	For $n\ge1$, let $\mathcal S_n$ be the set of integer vectors satisfying the numerical conditions in \cref{def:elementary}, that is,
	\[
		\mathcal S_n=\left\{\alpha\in\ZZ^n\;\middle|\;
		\begin{gathered}
			\text{for some }i_0\in\{1,\dotsc,n\},\quad
			\alpha_i=0\;(i<i_0),\quad \alpha_{i_0}=-1,\\
			-\bigcard{\sset{h<i}{\alpha_h=1}}\le\alpha_i\le1\quad (i>i_0)
		\end{gathered}\right\},
	\]
	and put $\mathcal S_0=\emptyset$.
	By \cref{thm:cone}, every extremal ray has a unique generator in $\mathcal S_n$.
	Such a generator is either $\ell_n$ or a vector $(a,b)$ with $a\in\mathcal S_{n-1}$.
	If two generators $(a,b)$ and $(a,c)$ have the same first $n-1$ entries and $b>c$, then
	\[
		(a,c)=(a,b)+(b-c)\ell_n.
	\]
	Since $\ell_n\in\NEb(\tilde X/X)$ by \cref{lem:fiberclass}, this contradicts extremality.
	Thus
	\[
		r(\tilde X/X)\le1+\card{\mathcal S_{n-1}}.
	\]

	We now show that $\card{\mathcal S_n}=\sum_{i=1}^nB_i$ for $n\ge0$.
	For $n\ge1$ and $i\ge0$, let $a_{n,i}$ be the number of vectors $\alpha\in\mathcal S_n$ such that $\alpha_1=-1$ and $\alpha_h=1$ for exactly $i$ indices $h$.
	Appending an entry in $\{-i,\dotsc,0\}$ to such a vector gives $i+1$ vectors counted by $a_{n+1,i}$.
	The remaining vectors counted by $a_{n+1,i}$ are obtained by appending $1$ to a vector counted by $a_{n,i-1}$.
	Hence
	\[
		a_{n+1,i}=(i+1)a_{n,i}+a_{n,i-1},
	\]
	where $a_{n,i}=0$ for $i<0$, and the initial values are $a_{1,0}=1$ and $a_{1,i}=0$ for $i>0$.
	The number of partitions of an $n$-element set into $i+1$ blocks satisfies the same recurrence and initial values.
	Thus $\sum_i a_{n,i}=B_n$.

	For a vector in $\mathcal S_n$ whose first nonzero entry is in position $i_0$, deleting the initial zero entries gives a vector in $\mathcal S_{n-i_0+1}$ with first entry $-1$.
	There are therefore $B_{n-i_0+1}$ such vectors, and summing over $i_0$ gives $\card{\mathcal S_n}=\sum_{i=1}^nB_i$.
	Applying this formula to $\mathcal S_{n-1}$ gives the required upper bound.

	Finally, when a new element is added to an $(i-1)$-element set, each partition extends in at most $i$ ways.
	Thus $B_i\le iB_{i-1}$ and hence $B_i\le i!$.
	For $n\ge2$, it follows that
	\[
		1+\sum_{i=1}^{n-1}B_i
		\le\sum_{i=0}^{n-1}i!
		\le(n-1)!+(n-1)(n-2)!=2(n-1)!.
	\]
	The case $n=1$ is immediate.
\end{proof}

\begin{rem}
	The asymptotic expansion of the Bell numbers gives
	\[
		\log B_n=n\log n-n\log\log n-n+o(n)
		\qquad(n\to\infty),
	\]
	see \cite[(19)]{Jakimczuk}.
	Since $B_{n-1}\le1+\sum_{i=1}^{n-1}B_i\le B_n$, the same expansion holds for $\log(1+\sum_{i=1}^{n-1}B_i)$.
	Equivalently, the upper bound in \cref{prop:raycount} satisfies
	\[
		1+\sum_{i=1}^{n-1}B_i=\left(\frac{n}{e\log n}\right)^n e^{o(n)}.
	\]

	\Cref{tab:raybounds} compares the two upper bounds in \cref{prop:raycount} for $1\le n\le5$.

	\begin{table}[htb]
		\centering
		\renewcommand{\arraystretch}{1.2}
		\begin{tabular}{r@{\qquad}r@{\qquad}r}
			\hline
			\noalign{\smallskip}
			$n$ & $1+\sum_{i=1}^{n-1}B_i$ & $2(n-1)!$ \\[4pt]
			\hline
			1 & $1$ & $2$ \\
			2 & $2$ & $2$ \\
			3 & $4$ & $4$ \\
			4 & $9$ & $12$ \\
			5 & $24$ & $48$ \\
			\hline
		\end{tabular}
		\vspace{4pt}
		\caption{Comparison of the upper bounds for the number of extremal rays.}
		\label{tab:raybounds}
	\end{table}
\end{rem}

\subsection{Realization of extremal rays}\label{sec:rays}

We show that every nonzero integer vector satisfying the numerical conditions in \cref{def:elementary} occurs as the intersection vector of an elementary curve spanning an extremal ray.

For the coordinate construction below, we use the notation of \cref{set:main}, with $X=\AA^d$, $Z_i=W_i$, $\tilde X=\tilde Y$, $\pi=\pi_Y$, and $E_i=F_i$ as in \cref{set:coordinate}.

\begin{prop}\label{prop:realization}
	Let $n\ge1$ be an integer, and let $\alpha=(\alpha_1,\dotsc,\alpha_n)\in\ZZ^n$ be a vector satisfying the numerical conditions in \cref{def:elementary}.
	Thus, setting $i_0\coloneq\min\sset{i}{\alpha_i\ne0}$, we have
	\[
		\alpha_{i_0}=-1,\qquad
		-\bigcard{\sset{h<i}{\alpha_h=1}}\le\alpha_i\le1\quad(i>i_0).
	\]
	Then there exist an integer $d\ge2$, coordinate subspaces $Z_1,\dotsc,Z_n\subseteq X=\AA^d$ satisfying the assumptions of \cref{set:main}, and an elementary curve $C\subseteq\tilde X$ in the successive blowup $\pi\colon\tilde X\to X$ along these centers such that
	\begin{itemize}
		\item $(C_1,\dotsc,C_n)=\alpha$, and
		\item $\RRp[C]$ is an extremal ray of $\NEb(\tilde X/X)$.
	\end{itemize}
\end{prop}

\begin{proof}
	We first explain how to insert a zero entry into a realization.
	Given a realization $\pi\colon\tilde X\to X=\AA^d$ with centers $Z_1,\dotsc,Z_n$ and curve $C$, consider the centers $Z_1\times\AA^2,\dotsc,Z_n\times\AA^2$ in $X\times\AA^2$ and insert $X\times\{0\}$ at any position in this sequence.
	The resulting successive blowup is the product morphism
	\[
		\tilde X\times\Bl_0\AA^2\longrightarrow X\times\AA^2.
	\]
	For any closed point $q\in\Bl_0\AA^2$, the intersection vector of $C\times\{q\}$ is obtained from $(C_1,\dotsc,C_n)$ by inserting a zero at the same position.
	Its class still spans an extremal ray, since the relative cone is the direct sum of the relative cones of the two factors.

	By omitting zero entries and inserting them after the construction, we may therefore assume that all entries of $\alpha$ are nonzero.
	In particular, $\alpha_1=-1$.
	If $\alpha$ has no positive entry, then the numerical conditions give $\alpha=(-1)$, which is realized by the exceptional curve of $\Bl_0\AA^2\to\AA^2$.

	We may now assume that $\alpha$ has a positive entry.
	Set
	\[
		I^+(\alpha)\coloneq\sset{i\in\{1,\dotsc,n\}}{\alpha_i=1},\qquad
		I^-(\alpha)\coloneq\sset{h\in\{1,\dotsc,n\}}{\alpha_h<0}.
	\]
	By the numerical conditions, for each $h\in I^-(\alpha)\setminus\{1\}$ we may choose a subset
	\[
		A_h\subseteq\sset{i\in I^+(\alpha)}{i<h},\qquad \card{A_h}=-\alpha_h.
	\]
	For each $i\in I^+(\alpha)$, set
	\[
		B_i\coloneq\sset{h\in I^-(\alpha)\setminus\{1\}}{i\in A_h}.
	\]
	Thus, for $i\in I^+(\alpha)$ and $h\in I^-(\alpha)\setminus\{1\}$, we have $i\in A_h$ if and only if $h\in B_i$, and in this case $i<h$.
	Set $n_0\coloneq\card{I^-(\alpha)}$.
	Take $X=\AA^{n+n_0}$ with coordinates
	\[
		x_1,y,\qquad
		x_i\quad(i\in I^+(\alpha)),\qquad
		z_h,w_h\quad(h\in I^-(\alpha)\setminus\{1\}).
	\]
	Let $Z_1,\dotsc,Z_n$ be the coordinate subspaces cut out by the coordinates in $J_1,\dotsc,J_n$, respectively, where
	\[
	\begin{aligned}
		J_1&=\{x_1\}\cup\sset{x_i}{i\in I^+(\alpha)},\\
		J_i&=\{x_i,y\}\cup\bigcup_{h\in B_i}\{z_h,w_h\}\qquad(i\in I^+(\alpha)),\\
		J_h&=\{z_h,w_h\}\qquad(h\in I^-(\alpha)\setminus\{1\}).
	\end{aligned}
	\]
	Every center has codimension at least two.
	The coordinate $x_1$ belongs only to $J_1$, and, for each $i\in I^+(\alpha)$, the coordinate $x_i$ belongs to no later coordinate set.
	Likewise, for each $h\in I^-(\alpha)\setminus\{1\}$, neither $z_h$ nor $w_h$ belongs to any later coordinate set, since $A_h\subseteq\{1,\dotsc,h-1\}$.
	Thus $J_i\not\subseteq J_h$ for $i,h\in\{1,\dotsc,n\}$ with $i<h$.
	Since the centers are coordinate subspaces, they satisfy all the assumptions of \cref{set:main}.

	We first construct a wall curve with intersection vector $\alpha$.
	Let $N=\ZZ^{n+n_0}$ be the lattice with basis vectors indexed by the coordinates above, and let $v_1,\dotsc,v_n\in N$ be the primitive ray generators introduced by the successive blowups, as in \cref{set:coordinate}.
	Then
	\[
	\begin{aligned}
		v_1&=e_{x_1}+\sum_{i\in I^+(\alpha)}e_{x_i},\\
		v_i&=e_{x_i}+e_y+\sum_{h\in B_i}v_h\quad(i\in I^+(\alpha)),\\
		v_h&=e_{z_h}+e_{w_h}\quad(h\in I^-(\alpha)\setminus\{1\}).
	\end{aligned}
	\]
	Hence
	\begin{equation}\label{eq:realizationrelation}
		e_{x_1}+\sum_{i\in I^+(\alpha)}v_i
		=v_1+(n-n_0)e_y-\sum_{h\in I^-(\alpha)\setminus\{1\}}\alpha_hv_h.
	\end{equation}
	Fix $i_1\in I^+(\alpha)$ and set
	\[
		\tau\coloneq\Cone\bigl(v_1,e_y,\;v_i\ (i\in I^+(\alpha)\setminus\{i_1\}),\;
		v_h,e_{w_h}\ (h\in I^-(\alpha)\setminus\{1\})\bigr).
	\]

	\begin{claim}\label{claim:wallcones}
		The cones $\Cone(\tau,e_{x_1})$ and $\Cone(\tau,v_{i_1})$ are maximal cones of the final fan $\tilde\Sigma$.
	\end{claim}

	\begin{proof}[Proof of \cref{claim:wallcones}]
		Start with the unique maximal cone of the fan of $X$ in $N_\RR$, namely
		\[
			\sigma_0\coloneq\Cone(
				e_{x_1},e_y,\;
				e_{x_i}\; (i\in I^+(\alpha)),\,
				e_{z_h},e_{w_h}\; (h\in I^-(\alpha)\setminus\{1\})).
		\]
		We choose a maximal cone at each subdivision as follows.
		\begin{itemize}
			\item At the first subdivision, replace $e_{x_{i_1}}$ with $v_1$ for the first cone, and replace $e_{x_1}$ with $v_1$ for the second.
			\item At each stage $i\in I^+(\alpha)\setminus\{i_1\}$, replace $e_{x_i}$ with $v_i$ for both cones.
			At stage $i_1$, make this replacement only for the second cone.
			In each of these cases, $e_{x_i}$ and $e_y$ are still present.
			For every $h\in B_i$, the generators $e_{z_h}$ and $e_{w_h}$ are also still present because $h>i$.
			Thus the current cone contains $\sigma_i$.
			For the first cone at stage $i_1$, the generator $e_{x_{i_1}}$ was removed at the first subdivision, so the current cone does not contain $\sigma_{i_1}$ and is unchanged.
			\item At each stage $h\in I^-(\alpha)\setminus\{1\}$, replace $e_{z_h}$ with $v_h$, which is possible since $e_{z_h}$ and $e_{w_h}$ have both been retained until this stage.
		\end{itemize}
		The resulting cones are precisely $\Cone(\tau,e_{x_1})$ and $\Cone(\tau,v_{i_1})$, proving the claim.
	\end{proof}

	These two cones share the wall $\tau$.
	The coefficients of $e_{x_1}$ and $v_{i_1}$ in \eqref{eq:realizationrelation} are both $1$, so this is the normalized wall relation.
	Since $E_i=D_{v_i}$ for $i\in\{1,\dotsc,n\}$ by \cref{lem:validstar}, the corresponding wall curve $C=V(\tau)$ has intersection vector $\alpha$ and is therefore elementary.

	It remains to prove that $\RRp[C]$ is extremal.
	For each $h\in I^-(\alpha)\setminus\{1\}$, the invariant divisors $D_{e_{z_h}}$ and $D_{e_{w_h}}$ are disjoint, since no cone of $\tilde\Sigma$ contains both coordinate rays after the subdivision along $\sigma_h=\Cone(e_{z_h},e_{w_h})$.
	Moreover, \cref{lem:D_e_j} gives
	\[
		D_{e_{z_h}}\lineq D_{e_{w_h}}\lineq-E_h-\sum_{i\in A_h}E_i.
	\]
	Thus the linear form
	\[
		f_h\colon N_1(\tilde X/X)\longrightarrow\RR,\qquad
		\beta\longmapsto-\beta_h-\sum_{i\in A_h}\beta_i
	\]
	is nonnegative on $\NEb(\tilde X/X)$.
	Indeed, any contracted curve is not contained in at least one of these two effective divisors, and their degrees on it coincide.
	By the choice of $A_h$, we have $f_h(\alpha)=0$.

	Consider the face
	\[
		F\coloneq\NEb(\tilde X/X)\cap
		\bigcap_{h\in I^-(\alpha)\setminus\{1\}}\Ker f_h.
	\]
	This face contains $\alpha$.
	We show that the source of any nonzero class $\beta\in F$, namely the index of its first nonzero entry, belongs to $\{1\}\cup I^+(\alpha)$.
	Indeed, the defining equalities of $F$ give
	\[
		\beta_h=-\sum_{i\in A_h}\beta_i
		\qquad(h\in I^-(\alpha)\setminus\{1\}).
	\]
	If $\beta$ had source $h\in I^-(\alpha)\setminus\{1\}$, then every entry on the right would vanish because $i<h$ for all $i\in A_h$.
	This would give $\beta_h=0$, a contradiction.

	We now construct a linear form that is nonnegative on $F$ and vanishes precisely on $\RRp[C]$.
	Choose positive integers $a_i$ for $i\in I^+(\alpha)$, in decreasing order of $i$, so that
	\[
		a_i>\sum_{i'\in I^+(\alpha),\;i'>i}a_{i'}.
	\]
	Consider the linear form
	\[
		f\colon N_1(\tilde X/X)\longrightarrow\RR,\qquad
		\beta\longmapsto-\sum_{i\in I^+(\alpha)}a_i(\beta_1+\beta_i).
	\]
	Let $\beta\in F$ be an elementary curve class with source $i_0$.
	If $i_0=1$, then $\beta_1=-1$ and $\beta_i\le1$ for every $i\in I^+(\alpha)$, so
	\[
		f(\beta)=\sum_{i\in I^+(\alpha)}a_i(1-\beta_i)\ge0.
	\]
	Equality holds if and only if $\beta_i=1$ for every $i\in I^+(\alpha)$.
	The defining equalities of $F$ show that this is equivalent to $\beta=\alpha=[C]$.
	If $i_0\in I^+(\alpha)$, then
	\[
		f(\beta)=a_{i_0}-\sum_{i\in I^+(\alpha),\;i>i_0}a_i\beta_i
		\ge a_{i_0}-\sum_{i\in I^+(\alpha),\;i>i_0}a_i>0.
	\]
	Since $F$ is generated by its elementary curve classes by \cref{prop:coordinate}, it follows that $f$ is nonnegative on $F$ and
	\[
		F\cap\Ker f=\RRp[C].
	\]
	Thus $\RRp[C]$ is an extremal ray of $F$ and therefore of $\NEb(\tilde X/X)$, completing the proof.
\end{proof}

For two centers, the relative cone of curves depends only on how the centers meet.
In \cref{set:main} with $n=2$, let $Z$ be an irreducible component of $Z_1\cap Z_2$, let $p\in Z$ be a closed point, and let $J_1,J_2$ be as in \cref{def:snc} at $p$.
Near $p$, the component $Z$ is defined by the $x_j$ with $j\in J_1\cup J_2$.
Hence
\[
	\codim_XZ=\card{J_1\cup J_2}=c_1+c_2-\card{J_1\cap J_2}.
\]
We say that $Z$ has \emph{excess dimension} if $\codim_XZ<c_1+c_2$, that is, if $J_1\cap J_2\ne\emptyset$.
We say that $Z_1$ and $Z_2$ \emph{meet transversely} if no irreducible component of $Z_1\cap Z_2$ has excess dimension, that is, if $Z_1\cap Z_2$ is either empty or of pure codimension $c_1+c_2$ in $X$.
The following corollary of \cref{thm:cone} recovers \cite[Proposition~3.6]{MasamuraYoshida}.

\begin{cor}\label{prop:twocenters}
	In \cref{set:main}, assume that $n=2$.
	Then the following hold.
	\begin{enumerate}
		\item\label{item:twotransverse}
			If $Z_1$ and $Z_2$ meet transversely, then
			\[
				\NEb(\tilde X/X)=\RRp[(-1,0),(0,-1)].
			\]
		\item\label{item:twoexcess}
			If some irreducible component of $Z_1\cap Z_2$ has excess dimension, then
			\[
				\NEb(\tilde X/X)=\RRp[(-1,1),(0,-1)].
			\]
	\end{enumerate}
\end{cor}

\begin{proof}
	By \cref{def:elementary}, the intersection vector of an elementary curve is one of $(-1,0)$, $(-1,1)$, and $(0,-1)$.
	The vectors $(-1,0)$ and $(0,-1)$ are represented by elementary curves by \cref{lem:fiberclass}.
	Since $(-1,0)=(-1,1)+(0,-1)$, \cref{thm:cone} shows that
	\[
		\NEb(\tilde X/X)=\RRp[(-1,1),(0,-1)]
	\]
	if some elementary curve has intersection vector $(-1,1)$, and
	\[
		\NEb(\tilde X/X)=\RRp[(-1,0),(0,-1)]
	\]
	otherwise.
	It therefore suffices to show that such an elementary curve exists if and only if some irreducible component of $Z_1\cap Z_2$ has excess dimension.

	Assume first that an irreducible component $Z$ of $Z_1\cap Z_2$ has excess dimension, and let $p\in Z$ be a closed point.
	By \cref{prop:localmodel}, the fiber $\pi^{-1}(p)$ is isomorphic to the fiber over the origin of a coordinate model at $p$, and this isomorphism preserves intersection vectors.
	Replacing $\tilde X\to X$ by this coordinate model, we may therefore assume that $X=\AA^d$, $p=0$, and $Z_i=V(x_j\mid j\in J_i)$ for $i=1,2$.
	Then $J_1\cap J_2\ne\emptyset$, and $J_1\not\subseteq J_2$ since $Z_2\not\subseteq Z_1$.
	The coordinates $x_j$ for $j\in J_1$ give homogeneous coordinates $y_j$ for $j\in J_1$ on the fiber $\pi_1^{-1}(0)\isom\PP^{c_1-1}$.
	By the chart description in the proof of \cref{lem:inherit},
	\[
		Z_2^{(1)}\cap\pi_1^{-1}(0)=V(y_j\mid j\in J_1\cap J_2)\subseteq\pi_1^{-1}(0).
	\]
	This is a nonempty proper linear subspace, since $J_1\setminus J_2$ and $J_1\cap J_2$ are nonempty.
	Let $C^{(1)}\subseteq\pi_1^{-1}(0)$ be a line that meets this subspace in one reduced point, and let $C\subseteq\tilde X$ be its strict transform.
	Since $C^{(1)}$ meets the center $Z_2^{(1)}$ in one reduced point, the morphism $C\to C^{(1)}$ is an isomorphism and $C_2=1$.
	Moreover, $C_1=E_1^{(1)}\cdot C^{(1)}=-1$.
	Thus $C$ is an elementary curve with intersection vector $(-1,1)$.

	Conversely, let $p\in X$ be a closed point such that either $p\notin Z_1\cap Z_2$ or the irreducible component of $Z_1\cap Z_2$ containing $p$ does not have excess dimension.
	We show that $C_2\le0$ for every curve $C\subseteq\pi^{-1}(p)$.
	We may assume that $p\in Z_1\cup Z_2$, since otherwise $\pi^{-1}(p)$ is a point.
	As in the first part, replacing $\tilde X\to X$ by a coordinate model at $p$, we may assume that $X=\AA^d$, $p=0$, $Z_i=V(x_j\mid j\in J_i)$ if $p\in Z_i$, and $Z_i=\emptyset$ otherwise.
	We first show that $Z_2^{(1)}=\pi_1^{-1}(Z_2)$.
	This is clear if $Z_1$ or $Z_2$ is empty.
	Otherwise, $p\in Z_1\cap Z_2$, so $J_1\cap J_2=\emptyset$ by the assumption on $p$, and the equality follows from the chart description in the proof of \cref{lem:inherit}.
	Hence the invertible sheaf $\Ob{\tilde X}(-E_2)=\Ic_{Z_2}\cdot\Ob{\tilde X}$ is generated by global sections.
	Thus $-E_2\cdot C\ge0$ for every curve $C\subseteq\pi^{-1}(p)$, which proves $C_2\le0$.

	In particular, if no irreducible component of $Z_1\cap Z_2$ has excess dimension, then no $\pi$-contracted curve has intersection vector $(-1,1)$.
	Hence the result follows.
\end{proof}

For three centers, the numerical conditions give the following list.

\begin{prop}\label{prop:eight}
	In \cref{set:main}, assume that $n=3$.
	\begin{enumerate}
		\item\label{item:eightlist}
			The intersection vector of every elementary curve is one of the following eight vectors, grouped here by source:
			\begin{align*}
				\text{source }1:\quad&(-1,0,0),\quad(-1,1,0),\quad(-1,0,1),\quad(-1,1,1),\quad(-1,1,-1),\\
				\text{source }2:\quad&(0,-1,0),\quad(0,-1,1),\\
				\text{source }3:\quad&(0,0,-1).
			\end{align*}
		\item\label{item:eightgen}
			The cone $\NEb(\tilde X/X)$ is generated by those vectors in \cref{item:eightlist} that occur as intersection vectors of elementary curves on $\tilde X$.
	\end{enumerate}
\end{prop}

\begin{proof}
	For source $1$, \cref{def:elementary} gives $C_2\in\{0,1\}$ and $-C_2\le C_3\le1$. For source $2$, it gives $C_3\in\{0,1\}$. For source $3$, the only possible vector is $(0,0,-1)$.
	This proves \cref{item:eightlist}, and \cref{item:eightgen} follows from \cref{thm:cone}.
\end{proof}

By \cref{prop:realization}, each vector in \cref{prop:eight} spans an extremal ray for a suitable choice of $X$ and centers $Z_1,Z_2,Z_3$, which may depend on the vector.
For a fixed configuration, at most four of these eight vectors span extremal rays by \cref{prop:raycount}.
The following example attains this bound.

\begin{exa}\label{exa:fourrays}
	Let $X=\AA^3$ with coordinates $x,y,z$, and set
	\[
		Z_1=V(x,y),\qquad Z_2=V(x,z),\qquad Z_3=V(y,z-1).
	\]
	These centers satisfy the assumptions of \cref{set:main}.
	Their nonempty pairwise intersections are $Z_1\cap Z_2=\{p_1\}$ and $Z_1\cap Z_3=\{p_2\}$, where $p_1=(0,0,0)$ and $p_2=(0,0,1)$.
	In particular, $Z_2\cap Z_3=\emptyset$.
	The cone $\NEb(\tilde X/X)$ has exactly four extremal rays, with generators
	\[
		\alpha_1=(-1,1,0),\qquad \alpha_2=(-1,0,1),\qquad \ell_2=(0,-1,0),\qquad \ell_3=(0,0,-1).
	\]

	Indeed, the fiber $C_1^{(1)}\coloneq\pi_1^{-1}(p_1)$ is a line meeting $Z_2^{(1)}$ in one reduced point.
	Its strict transform $C_1$ in $\tilde X$ has class $\alpha_1$.
	The same construction over $p_2$ gives the class $\alpha_2$, and $\ell_2,\ell_3$ are effective by \cref{lem:fiberclass}.
	For any $\pi$-contracted curve $C$, the intersection numbers $C_2$ and $C_3$ cannot both be nonzero, since its image in $X$ lies outside at least one of $Z_2,Z_3$.
	Thus \cref{prop:eight} shows that the cone is generated by the four displayed vectors together with $\ell_1$.
	
	Since $\ell_1=\alpha_1+\ell_2$, the four displayed vectors generate the cone.
	None of these four vectors is a nonnegative linear combination of the other three, so all four rays are extremal.
\end{exa}

The following example shows that elementary curves need not generate the relative cone of curves without the simple normal crossings hypothesis.

\begin{exa}\label{exa:sharp}
	Let $X=\AA^3$ and $Z_1=\{0\}$.
	Choose five distinct lines $Z_2,\dotsc,Z_6$ through the origin such that no three of the corresponding points
	\[
		p_2,\dotsc,p_6\in E_1^{(1)}\isom\PP^2
	\]
	are collinear.
	Let
	\[
		\pi\colon\tilde X=X^{(6)}\longrightarrow X
	\]
	be the successive blowup along $Z_1,\dotsc,Z_6$.

	The first center has codimension three, and the remaining centers have codimension two.
	Every center $Z_i^{(i-1)}$ is smooth, and $Z_h\not\supseteq Z_i$ for $h<i$.
	Thus all the assumptions of \cref{set:main} other than simple normal crossings are satisfied.
	The arrangement does not have simple normal crossings at the origin.

	After the first blowup, the strict transforms $Z_2^{(1)},\dotsc,Z_6^{(1)}$ are pairwise disjoint and meet $E_1^{(1)}$ at $p_2,\dotsc,p_6$, respectively.
	The subsequent blowups therefore induce on $E_1^{(1)}$ the blowups at these five points.
	Hence
	\[
		E_1\isom\Bl_{p_2,\dotsc,p_6}\PP^2.
	\]

	There is a unique conic through $p_2,\dotsc,p_6$, which is irreducible and smooth.
	Let $C\subseteq E_1$ be its strict transform.
	Then $C$ is $\pi$-contracted and has intersection vector
	\[
		(E_1\cdot C,\dotsc,E_6\cdot C)=(-2,1,1,1,1,1).
	\]

	Let $C'$ be an elementary curve.
	If its source is $1$, then $C'$ is the strict transform of a line in $E_1^{(1)}$.
	Since no three of the points $p_i$ are collinear, its intersection vector is
	\[
		(-1,C'_2,\dotsc,C'_6),
		\qquad
		C'_i\in\{0,1\},
		\qquad
		\sum_{i=2}^6C'_i\le2.
	\]
	If its source is $i_0>1$, then $C'$ is the strict transform of a fiber of the $i_0$-th blowup.
	Since the later centers are disjoint from this fiber, we have $[C']=\ell_{i_0}$.

	It follows that
	\[
		D\coloneq2E_1+E_2+\dotsb+E_6
	\]
	has nonpositive degree on every elementary curve class.
	On the other hand, $D\cdot C=1$.
	Therefore $[C]$ does not belong to the cone generated by the elementary classes, and these classes do not generate $\NEb(\tilde X/X)$.
\end{exa}

\section{Contractions}\label{sec:contractions}

In this section, we study the contractions of the faces of $\NEb(\tilde X/X)$: their existence (\cref{sec:existence}), their targets (\cref{sec:target}), the flips of the small ones (\cref{sec:flips}), and the exceptional loci of extremal ray contractions (\cref{sec:types}).

\subsection{Relative base-point-freeness and contractions of faces}\label{sec:existence}

We prove \cref{introthm:contract}: every $\pi$-nef Cartier divisor $D$ on $\tilde X$ is $\pi$-base-point-free.
It follows that every face of $\NEb(\tilde X/X)$ admits a contraction over $X$.

The key local input is the following base-point-free statement for toric morphisms.

\begin{prop}[Fujino--Sato {\cite{FujinoSato}}]\label{prop:toricbpf}
	Let $\Sigma$ be a smooth fan in $N_\RR$ whose support $\sigma_0\coloneq\abs\Sigma$ is a full-dimensional strongly convex cone, and let $\pi\colon X_\Sigma\to X_{\sigma_0}$ be the induced proper toric morphism, as in \cref{ss:conegeneration}.
	Let $D$ be a Cartier divisor on $X_\Sigma$.
	If $D$ is $\pi$-nef, then $D$ is $\pi$-base-point-free.
\end{prop}

\begin{proof}
	Replacing $D$ by a linearly equivalent divisor, we may assume that $D$ is torus-invariant.
	The assertion then follows from \cite[Proposition~4.6]{FujinoSato}.
\end{proof}

Relative nefness passes to the coordinate models.

\begin{lem}\label{lem:transport}
	In \cref{set:main}, let $D=-\sum_ia_iE_i$ be a $\pi$-nef $\RR$-divisor on $\tilde X$.
	Let $p\in\bigcup_iZ_i$ be a closed point, and let $\pi_{Y_p}\colon\tilde Y_p\to\AA^d$ be a coordinate model at $p$ with exceptional divisors $F_1,\dotsc,F_n$.
	Then the divisor $D_Y\coloneq-\sum_ia_iF_i$ on $\tilde Y_p$ is $\pi_{Y_p}$-nef.
\end{lem}

\begin{proof}
	By \cref{prop:coordinate}, the cone $\NEb(\tilde Y_p/\AA^d)$ is generated by classes of curves in the central fiber $\pi_{Y_p}^{-1}(0)$, so it suffices to prove $D_Y\cdot C_Y\ge0$ for every such curve $C_Y$.
	Under the isomorphism $\pi^{-1}(p)\isomarrow\pi_{Y_p}^{-1}(0)$ of \cref{prop:localmodel}, the curve $C_Y$ corresponds to a $\pi$-contracted curve $C\subseteq\tilde X$ with the same intersection vector, so
	\[
		D_Y\cdot C_Y=D\cdot C\ge0,
	\]
	since $D$ is $\pi$-nef.
\end{proof}

We now globalize the toric statement.

\begin{thm}\label{thm:bpf}
	In \cref{set:main}, let $D$ be a $\pi$-nef Cartier divisor on $\tilde X$.
	Then $D$ is $\pi$-base-point-free.
\end{thm}

\begin{proof}
	By \cref{lem:divisordecomposition}, we may write $D=\pi^*D_0-\sum_ia_iE_i$, where $D_0$ is a Cartier divisor on $X$ and $a_i\in\ZZ$.
	Subtracting $\pi^*D_0$ preserves relative nefness and relative base-point-freeness.
	Thus we may assume that $D=-\sum_ia_iE_i$.

	Relative base-point-freeness is local on the base $X$.
	Cover $X$ by the open set $X_0\coloneq X\setminus\bigcup_iZ_i$ together with the coordinate neighborhoods $U_p$ for the points $p\in\bigcup_iZ_i$.

	Over $X_0$, the morphism $\pi$ is an isomorphism and every $E_i$ is disjoint from $\tilde X_0\coloneq\pi^{-1}(X_0)$, so $\restrdiv D{\tilde X_0}=0$.
	In particular, the restriction is base-point-free over $X_0$.

	Over a coordinate neighborhood $U_p$, \cref{prop:localmodel} gives $\restrdiv D{\tilde U_p}=\tilde\phi_p^*D_Y$, where $D_Y\coloneq-\sum_ia_iF_i$ and the $F_i$ are the exceptional divisors of $\pi_{Y_p}$.
	By \cref{lem:transport}, the divisor $D_Y$ is $\pi_{Y_p}$-nef, so it is $\pi_{Y_p}$-base-point-free by \cref{prop:toricbpf}.
	Relative base-point-freeness is preserved under base change, so $\restrdiv D{\tilde U_p}$ is base-point-free over $U_p$.
	Thus $D$ is $\pi$-base-point-free over each member of the cover, hence over $X$.
\end{proof}

\begin{cor}\label{cor:contract}
	In \cref{set:main}, let $F\subseteq\NEb(\tilde X/X)$ be a face.
	Then the contraction
	\[
		g_F\colon\tilde X\longrightarrow\tilde X_F
	\]
	of $F$ over $X$ exists. Moreover, the sequence
	\[
	\begin{tikzcd}
		0\arrow[r]
			& N^1(\tilde X_F/X)\arrow[r,"g_F^*"]
			& N^1(\tilde X/X)\arrow[r]
			& N^1(\tilde X/\tilde X_F)\arrow[r]
			& 0
	\end{tikzcd}
	\]
	is exact. In particular, $\rho(\tilde X/X)=\rho(\tilde X_F/X)+\dim F$.
\end{cor}

\begin{proof}
	By \cref{thm:cone}, the face $F$ has an integral supporting divisor $D_F$.
	The divisor $D_F$ is $\pi$-base-point-free by \cref{thm:bpf}.
	The Stein factorization of the morphism defined by $D_F$ is therefore the contraction $g_F$ of $F$.

	We prove the exactness.
	The pullback $g_F^*$ is injective, while the last arrow is surjective with kernel $F^\perp$.
	Since $F$ is rational polyhedral, it is enough to show that every integral class in $F^\perp$ lies in the image of $g_F^*$.
	Let $D\in\bigoplus_i\ZZ E_i$ be a divisor whose class lies in $F^\perp$.
	For $m\gg0$, the divisor $mD_F+D$ is again a supporting divisor of $F$, and hence is $\pi$-base-point-free by \cref{thm:bpf}.
	The morphisms defined by $D_F$ and $mD_F+D$ both have $g_F$ as their Stein factorization, so both divisors descend to $\tilde X_F$.
	Hence $D=(mD_F+D)-mD_F$ also descends, which proves the exactness.
	The equality of relative Picard numbers follows from $\codim F^\perp=\dim F$.
\end{proof}

Finally, we compare the contraction of a face with the toric contraction of the corresponding face of a local coordinate model.

\begin{lem}\label{lem:localcontract}
	In \cref{set:main}, let $F\subseteq\NEb(\tilde X/X)$ be a face with contraction $g_F\colon\tilde X\to\tilde X_F$ (\cref{cor:contract}).
	Let $p\in\bigcup_iZ_i$ be a closed point, let $\pi_{Y_p}\colon\tilde Y_p\to\AA^d$ together with $\phi_p\colon U_p\to\AA^d$ be a coordinate model at $p$, and let
	\[
		\iota_p\colon\NEb(\tilde Y_p/\AA^d)\longhookrightarrow\NEb(\tilde X/X)
	\]
	be the injection of \cref{lem:NEinj}.
	The contraction $g_{F_p}\colon\tilde Y_p\to\tilde Y_{p,F}$ of the face $F_p\coloneq\iota_p^{-1}(F)$ fits into the following commutative diagram, in which all four squares are cartesian:
	\[
	\begin{tikzcd}
		\tilde X\arrow[d,"g_F"']
			& \tilde U_p\arrow[r,"\tilde\phi_p"]\arrow[d]\arrow[l,hook']
			& \tilde Y_p\arrow[d,"g_{F_p}"]\\
		\tilde X_F\arrow[d]
			& \tilde U_{p,F}\arrow[r,"\tilde\phi_{p,F}"]\arrow[d]\arrow[l,hook']
			& \tilde Y_{p,F}\arrow[d]\\
		X
			& U_p\arrow[r,"\phi_p"']\arrow[l,hook']
			& \AA^d
	\end{tikzcd}
	\]
	Here all horizontal arrows are \'etale.
	Moreover,
	\[
		\Exc(g_F)\cap\tilde U_p=\tilde\phi_p^{-1}(\Exc(g_{F_p})).
	\]
	In particular, if $R\coloneq F$ is an extremal ray and $p\in\pi(\Exc(g_R))$, then $R_p\coloneq F_p$ is an extremal ray with $\iota_p(R_p)=R$.
\end{lem}

\begin{proof}
	After omitting the empty stages, the coordinate model satisfies \cref{set:coordinate}, so the contraction $g_{F_p}$ exists by \cref{cor:contract}.
	Set
	\[
		\tilde U_{p,F}\coloneq\tilde Y_{p,F}\times_{\AA^d}U_p.
	\]
	Since $\pi_{Y_p}$ factors through $g_{F_p}$, the cartesian square of \cref{prop:localmodel}~\cref{item:fiberproduct} splits into the two right-hand squares, both cartesian, where $g'_p\colon\tilde U_p\to\tilde U_{p,F}$ is the base change of $g_{F_p}$.
	The morphism $g'_p$ is a contraction by flat base change, and both horizontal arrows in these squares are \'etale.

	We show that $g'_p$ is the restriction of $g_F$ over $U_p$.
	Let $C\subseteq\tilde U_p$ be a curve contracted over $U_p$, and let $C_Y\coloneq\tilde\phi_p(C)$ be its image in $\tilde Y_p$.
	Since $\tilde\phi_p$ identifies the fibers over $\pi(C)$ and $\phi_p(\pi(C))$, it maps $C$ isomorphically onto $C_Y$.
	Since $g'_p$ is the base change of $g_{F_p}$, it contracts $C$ if and only if $[C_Y]\in F_p$.
	By \cref{lem:NEinj}, we have $\iota_p[C_Y]=[C]$, so this condition is equivalent to $[C]\in F$.
	Thus $g'_p$ and the restriction of $g_F$ contract exactly the same curves, and we obtain an isomorphism
	\[
		\tilde U_{p,F}\isom\tilde X_F\times_XU_p
	\]
	over $U_p$, compatible with the morphisms from $\tilde U_p$.
	This gives the two left-hand cartesian squares, whose horizontal arrows are open immersions, and completes the diagram.
	The equality of exceptional loci holds set-theoretically because the isomorphism locus of a proper birational morphism to a normal variety commutes with \'etale base change.
	It holds as an equality of closed subschemes because an \'etale inverse image of a reduced scheme is reduced.

	Finally, assume that $R\coloneq F$ is an extremal ray and $p\in\pi(\Exc(g_R))$.
	The equality of exceptional loci gives $\Exc(g_{R_p})\ne\emptyset$, so $R_p\ne0$.
	Since $\iota_p$ is injective and $R_p=\iota_p^{-1}(R)$ is a face, $R_p$ is an extremal ray with $\iota_p(R_p)=R$.
\end{proof}

\subsection{Targets of the contractions}\label{sec:target}

In this subsection, we identify the target of the contraction of any face as the normalized blowup of an explicit intersection of powers of the center ideals.

In \cref{set:main}, for a vector $a=(a_1,\dotsc,a_n)\in\ZZ^n_{\ge0}$, set
\[
	\Ic^{(a)}\coloneq\Ic_{Z_1}^{a_1}\cap\dotsb\cap\Ic_{Z_n}^{a_n}.
\]
By \cref{thm:cone}, every face of $\NEb(\tilde X/X)$ has an integral supporting divisor $D=-\sum_i a_iE_i$.
The coefficients $a_i$ are nonnegative by \cref{lem:fiberclass}, so the ideal $\Ic^{(a)}$ is defined for every such divisor.

The next lemma computes the graded pieces that occur in the section algebra of $D$.

\begin{lem}\label{lem:valuation}
	In \cref{set:main}, let $a\in\ZZ^n_{\ge0}$.
	Then
	\[
		\pi_*\Ob{\tilde X}\biggl(-\sum_ia_iE_i\biggr)=\Ic^{(a)}.
	\]
\end{lem}

\begin{proof}
	Put $D\coloneq-\sum_ia_iE_i$.
	For each $i$, let $v_{E_i}$ denote the divisorial valuation of $k(X)$ associated with the prime divisor $E_i$ on $\tilde X$.
	Since $D\le0$, we have $\Ob{\tilde X}(D)\subseteq\Ob{\tilde X}$.
	Since $\pi_*\Ob{\tilde X}=\Ob X$, we may therefore regard $\pi_*\Ob{\tilde X}(D)$ as an ideal sheaf on $X$.
	It is enough to compare stalks at closed points.
	Let $p\in X$ be a closed point.
	The definition of $\Ob{\tilde X}(D)$ gives
	\[
		(\pi_*\Ob{\tilde X}(D))_p
		=\sset{f\in\Ob{X,p}}{v_{E_i}(f)\ge a_i\text{ for every $i$ such that }p\in Z_i}.
	\]

	Since $Z_h\not\supseteq Z_i$ for $h<i$ (\cref{set:main}), the generic point $\eta_i$ of $Z_i$ lies on none of the earlier centers.
	Thus, near $\eta_i$, the $i$-th blowup is identified with the blowup of $X$ along $Z_i$.
	The corresponding exceptional valuation is therefore the $\Ic_{Z_i}$-adic order.
	The later blowups preserve the prime strict transform of $E_i^{(i)}$ and hence leave this valuation unchanged.

	If $p\notin Z_i$, then the $i$-th valuation condition is absent and $\Ic_{Z_i,p}^{a_i}=\Ob{X,p}$.
	If $p\in Z_i$, choose a regular system of parameters $x_1,\dotsc,x_d$ of $\Ob{X,p}$ such that $\Ic_{Z_i,p}=(x_j\mid j\in J_i)$.
	Then $v_{E_i}$ is the monomial valuation that assigns weight one to the parameters $x_j$ with $j\in J_i$ and weight zero to the remaining parameters.
	Consequently,
	\[
		\sset{f\in\Ob{X,p}}{v_{E_i}(f)\ge a_i}=\Ic_{Z_i,p}^{a_i}.
	\]
	Intersecting these valuation ideals over all $i$ gives
	\[
		\pi_*\Ob{\tilde X}(D)
		=\bigcap_i\Ic_{Z_i}^{a_i}=\Ic^{(a)}. \qedhere
	\]
\end{proof}

We can now identify the target.

\begin{thm}\label{thm:targetn}
	In \cref{set:main}, let $F$ be a face of $\NEb(\tilde X/X)$, and write
	$g_F\colon\tilde X\to\tilde X_F$ for its contraction.
	Let $a=(a_1,\dotsc,a_n)\in\ZZ^n_{\ge0}$ be such that $-\sum_i a_iE_i$ is a supporting divisor of $F$.
	Then
	\[
		\tilde X_F
		\isom\overline{\Bl}_{\Ic^{(a)}}X,
	\]
	where $\overline{\Bl}_{\Ic^{(a)}}X$ denotes the normalization of $\Bl_{\Ic^{(a)}}X$.
	In particular, this normalized blowup is independent of the choice of $a$, up to isomorphism over $X$.
\end{thm}

\begin{proof}
	Set $D\coloneq-\sum_i a_iE_i$ and $\mathfrak a\coloneq\Ic^{(a)}$.
	The divisor $D$ is $\pi$-nef because it is a supporting divisor of $F$.
	By \cref{lem:valuation,thm:bpf}, the evaluation map
	\[
		\pi^*\mathfrak a
		=\pi^*\pi_*\Ob{\tilde X}(D)
		\longrightarrow\Ob{\tilde X}(D)
	\]
	is surjective.
	Since $D\le0$, the invertible sheaf $\Ob{\tilde X}(D)$ is an ideal subsheaf of $\Ob{\tilde X}$, and the image of the evaluation map is $\mathfrak a\Ob{\tilde X}$.
	Hence $\mathfrak a\Ob{\tilde X}=\Ob{\tilde X}(D)$.

	By the universal property of the blowup, $\pi$ factors through $\Bl_{\mathfrak a}X$.
	Since $\tilde X$ is normal, the morphism $\tilde X\to\Bl_{\mathfrak a}X$ factors uniquely through the normalization $\bar X\coloneq\overline{\Bl}_{\mathfrak a}X$.
	We therefore obtain a factorization
	\[
	\begin{tikzcd}
		\pi\colon\tilde X\arrow[r,"g"]
			& \bar X=\overline{\Bl}_{\mathfrak a}X\arrow[r,"h"]
			& X,
	\end{tikzcd}
	\]
	where $g$ is projective and birational.
	Since $\bar X$ is normal, $g_*\Ob{\tilde X}=\Ob{\bar X}$.

	The tautological ideal $\mathfrak a\Ob{\Bl_{\mathfrak a}X}$ is invertible and ample over $X$.
	Its pullback $L\coloneq\mathfrak a\Ob{\bar X}$ is therefore invertible and $h$-ample.
	Moreover,
	\[
		g^*L=\mathfrak a\Ob{\tilde X}=\Ob{\tilde X}(D).
	\]
	Let $C\subseteq\tilde X$ be a $\pi$-contracted curve.
	Since $L$ is $h$-ample, the curve $C$ is contracted by $g$ if and only if $D\cdot C=(g^*L)\cdot C=0$.
	Since $D$ is a supporting divisor of $F$, this is equivalent to $[C]\in F$.
	Thus $g$ is a contraction of $F$, and $\tilde X_F\isom\overline{\Bl}_{\mathfrak a}X$.
\end{proof}

\subsection{Flips and reordered blowups}\label{sec:flips}

We prove that every small face contraction admits a $D$-flip for every $\RR$-Cartier divisor $D$ that is relatively anti-ample.
We then characterize when a different ordering of the same centers gives the flipped model.

\begin{defi}\label{def:flip}
	Let $g\colon X\to Y$ be a projective small birational morphism of normal varieties, and let $D$ be an $\RR$-Cartier divisor on $X$ such that $-D$ is $g$-ample.
	A \emph{$D$-flip} of $g$ is a projective small birational morphism $g^+\colon X^+\to Y$ such that $X^+$ is normal and the strict transform $D^+$ of $D$ to $X^+$ is $\RR$-Cartier and $g^+$-ample.
\end{defi}

\begin{rem}\label{rem:flipcriterion}
	Assume that $D$ is integral.
	Then by \cite[Lemma~6.2]{KollarMori}, the $D$-flip exists if and only if the graded $\Ob Y$-algebra $\bigoplus_{m\ge0}g_*\Ob X(mD)$ is finitely generated.
	Whenever it exists, it is unique over $Y$ and is given by
	\[
		X^+=\Proj_Y\bigoplus_{m\ge0}g_*\Ob X(mD).
	\]
\end{rem}

In the situation of \cref{def:flip}, assume that $\rho(X/Y)=1$, and let $D'$ be another $\RR$-Cartier divisor on $X$ such that $-D'$ is $g$-ample.
Then the negativity lemma \cite[\S2.3]{BirkarFlips} implies that a $D$-flip and a $D'$-flip agree whenever both exist, but in general a $D'$-flip need not exist even if a $D$-flip exists.
In our setting, if a $D$-flip exists, then it is also a $D'$-flip.

\begin{lem}\label{lem:flipindependence}
	In \cref{set:main}, let $R$ be an extremal ray of $\NEb(\tilde X/X)$ whose contraction $g_R\colon\tilde X\to\tilde X_R$ is small.
	Let $D$ be an $\RR$-Cartier divisor on $\tilde X$ with $D\cdot R<0$.
	Assume that the $D$-flip $g_R^+\colon\tilde X^+\to\tilde X_R$ exists.
	Then it is also a $D'$-flip for any $\RR$-Cartier divisor $D'$ with $D'\cdot R<0$.
\end{lem}

\begin{proof}
	Let $D'$ be an $\RR$-Cartier divisor on $\tilde X$ with $D'\cdot R<0$.
	We can choose a real number $t>0$ such that $(D'-tD)\cdot R=0$.
	By \cref{cor:contract}, there is an $\RR$-Cartier divisor $L$ on $\tilde X_R$ such that $D'-tD-g_R^*L$ is numerically trivial over $X$.
	Applying \cref{lem:divisordecomposition} by $\RR$-linearity, we can write
	\[
		D'-tD-g_R^*L=\pi^*D_0+\sum_i a_iE_i,
	\]
	where $D_0$ is an $\RR$-Cartier divisor on $X$ and $a_i\in\RR$.
	Since $D'-tD-g_R^*L$ is numerically trivial over $X$, we have $\sum_i a_i[E_i]=0$ in $N^1(\tilde X/X)$, so $a_i=0$ for every $i$.
	Let $h\colon\tilde X_R\to X$ be the structure morphism.
	Replacing $L$ by $L+h^*D_0$, we obtain
	\[
		D'=tD+g_R^*L.
	\]
	Since $g_R$ and $g_R^+$ are small, taking strict transforms on $\tilde X^+$ gives
	\[
		(D')^+=tD^++(g_R^+)^*L.
	\]
	The divisor $D^+$ is $\RR$-Cartier and $g_R^+$-ample, and $t>0$.
	Hence $(D')^+$ is also $\RR$-Cartier and $g_R^+$-ample, so $g_R^+$ is a $D'$-flip.
\end{proof}

We next prove the existence of flips for coordinate blowups.

\begin{lem}\label{lem:localflipfan}
	In \cref{set:coordinate}, let $F$ be a face of $\NEb(\tilde Y/\AA^d)$ whose contraction $g_F\colon\tilde Y\to\tilde Y_F$ is small.
	Let $D=\sum_{i=1}^n a_iF_i$ be an integral divisor such that $-D$ is $g_F$-ample.
	Then the $D$-flip $g_F^+\colon\tilde Y^+\to\tilde Y_F$ of $g_F$ exists.
\end{lem}

\begin{proof}
	The contraction $g_F$ is toric, and $D$ is a torus-invariant Cartier divisor.
	Thus \cite[Theorem~4.8]{FujinoSato} gives the $D$-flip.
\end{proof}

We now pass from the coordinate statement to the general case.

\begin{thm}\label{thm:flip}
	In \cref{set:main}, let $F$ be a face of $\NEb(\tilde X/X)$ whose contraction $g_F\colon\tilde X\to\tilde X_F$ is small.
	Let $D$ be an $\RR$-Cartier divisor on $\tilde X$ such that $-D$ is $g_F$-ample.
	Then the $D$-flip
	\[
		g_F^+\colon\tilde X^+\longrightarrow\tilde X_F
	\]
	of $g_F$ exists.
\end{thm}

\begin{proof}
	Applying \cref{lem:divisordecomposition} by $\RR$-linearity, we can write $D=\pi^*D_0+\sum_ia_iE_i$, where $D_0$ is an $\RR$-Cartier divisor on $X$ and $a_i\in\RR$.
	Since $\pi^*D_0$ is the pullback of an $\RR$-Cartier divisor on $\tilde X_F$, we may assume that $D=\sum_ia_iE_i$.

	Choose finitely many coordinate neighborhoods $U_p$ which, together with $X_0\coloneq X\setminus\bigcup_iZ_i$, cover $X$.
	For each chosen $p$, set $F_p\coloneq\iota_p^{-1}(F)\subseteq\NEb(\tilde Y_p/\AA^d)$ and $D_p\coloneq\sum_i a_iF_i$ on the coordinate model $\tilde Y_p$, where $\iota_p$ is the injection of \cref{lem:NEinj}.
	By \cref{lem:localcontract}, the restriction of $g_F$ over $U_p$ is the base change of the contraction of $F_p$
	\[
		g_{F_p}\colon\tilde Y_p\longrightarrow\tilde Y_{p,F}
	\]
	along the \'etale morphism $\tilde\phi_{p,F}\colon\tilde U_{p,F}\to\tilde Y_{p,F}$.
	The divisor $-D_p$ is $g_{F_p}$-ample, since $\iota_p$ induces $F_p\hookrightarrow F$ and preserves intersection vectors.
	Over $X_0$, the morphism $g_F$ is an isomorphism and $D$ restricts to zero.

	We show that $g_{F_p}$ is small.
	Suppose, to the contrary, that $\Exc(g_{F_p})$ contains a prime divisor $B$.
	Since $g_{F_p}$ is toric, $B$ is torus-invariant, so its image in $\AA^d$ is a nonempty closed torus-invariant subset and therefore contains the origin $\phi_p(p)$.
	The square in \cref{prop:localmodel} is cartesian, so $B$ meets the image of $\tilde\phi_p$.
	Since $\tilde\phi_p$ is \'etale, the inverse image $\tilde\phi_p^{-1}(B)$ is a nonempty divisor on $\tilde U_p$, and it is contained in $\Exc(g_F)$ by \cref{lem:localcontract}.
	This contradicts the smallness of $g_F$.

	We now construct the flip when every $a_i$ is rational.
	Replacing $D$ by a positive integral multiple, we may assume that every $a_i$ is an integer.
	Set
	\[
		\mathcal A\coloneq\bigoplus_{m\ge0}(g_F)_*\Ob{\tilde X}(mD),
		\qquad
		\mathcal A_p\coloneq\bigoplus_{m\ge0}(g_{F_p})_*\Ob{\tilde Y_p}(mD_p).
	\]
	The coordinate model $\tilde Y_p$ satisfies \cref{set:coordinate}, and $D_p$ is an integral combination of its exceptional divisors.
	Since $g_{F_p}$ is small and $-D_p$ is $g_{F_p}$-ample, \cref{lem:localflipfan} gives the $D_p$-flip $g_{F_p}^+\colon\tilde Y_p^+\to\tilde Y_{p,F}$.
	By \cref{rem:flipcriterion}, the algebra $\mathcal A_p$ is finitely generated and $\tilde Y_p^+=\Proj_{\tilde Y_{p,F}}\mathcal A_p$.

	By \cref{prop:localmodel}, the restriction of $D$ to $\tilde U_p$ is $\tilde\phi_p^*D_p$.
	Since $\tilde\phi_{p,F}$ is flat and the squares of \cref{lem:localcontract} are cartesian, flat base change in each degree gives
	\[
		\restr{\mathcal A}{\tilde U_{p,F}}
		\isom\tilde\phi_{p,F}^*\mathcal A_p.
	\]
	Hence $\mathcal A$ is finitely generated over each $\tilde U_{p,F}$.
	It is also finitely generated over the inverse image of $X_0$, since $g_F$ is an isomorphism and $D$ restricts to zero there.
	Thus $\mathcal A$ is finitely generated, and \cref{rem:flipcriterion} gives the $D$-flip
	\[
		g_F^+\colon\tilde X^+\coloneq\Proj_{\tilde X_F}\mathcal A\longrightarrow\tilde X_F.
	\]
	This proves the theorem when every $a_i$ is rational.
	We also record that the restriction of $\tilde X^+$ over $U_p$ is $\tilde Y_p^+\times_{\tilde Y_{p,F}}\tilde U_{p,F}$, since relative $\Proj$ commutes with base change.

	Now let $a\coloneq(a_1,\dotsc,a_n)\in\RR^n$ be arbitrary.
	For a chosen point $p$, let $\tilde\Sigma_p$ (resp.~$\tilde\Sigma_{p,F}$) be the fan of the toric variety $\tilde Y_p$ (resp.~$\tilde Y_{p,F}$).
	Since $g_{F_p}$ is small, $\tilde\Sigma_p$ is a refinement of $\tilde\Sigma_{p,F}$ with the same rays.
	Let $\hat\Sigma$ be a refinement of $\tilde\Sigma_{p,F}$ with the same rays, and set $\hat Y\coloneq X_{\hat\Sigma}$.
	There are only finitely many such refinements, since every cone of $\hat\Sigma$ is generated by rays of $\tilde\Sigma_{p,F}$.
	Note that the identity of $N=\ZZ^d$ induces a small proper birational toric morphism $\hat Y\to\tilde Y_{p,F}$, and that
	\[
		\hat\Sigma(1)=\tilde\Sigma_{p,F}(1)=\tilde\Sigma_p(1)
		=\sset{\rho_{e_j}}{1\le j\le d}\cup\sset{\rho_{v_i}}{1\le i\le n,\;p\in Z_i}.
	\]

	For $b\in\RR^n$, let $D(b)\coloneq\sum_ib_iF_i$ on $\tilde Y_p$, and let $\hat D(b)$ be its strict transform on $\hat Y$.
	The convention $F_i=0$ for $p\notin Z_i$ gives
	\[
		D(b)=\sum_{\rho\in\tilde\Sigma_p(1)}c_\rho(b)D_\rho,
		\qquad
		c_\rho(b)\coloneq
		\begin{cases}
			b_i & \rho=\rho_{v_i},\\
			0 & \rho=\rho_{e_j}
		\end{cases}
	\]
	on $\tilde Y_p$, and hence $\hat D(b)=\sum_\rho c_\rho(b)D_\rho$ on $\hat Y$ as well.

	\begin{claim}\label{claim:rationalflip}
		There exists $a'=(a'_1,\dotsc,a'_n)\in\QQ^n$ such that $-\sum_ia'_iE_i$ is $g_F$-ample and such that, for every chosen point $p$ and every refinement $\hat\Sigma$ of $\tilde\Sigma_{p,F}$ with the same rays, if $\hat D(a')$ is $\QQ$-Cartier and ample over $\tilde Y_{p,F}$, then $\hat D(a)$ is $\RR$-Cartier and ample over $\tilde Y_{p,F}$.
	\end{claim}

	\begin{proof}[Proof of \cref{claim:rationalflip}]
		Fix a chosen point $p$ and a refinement $\hat\Sigma$ of $\tilde\Sigma_{p,F}$ with the same rays.
		Since $\abs{\hat\Sigma}=\Cone(e_1,\dotsc,e_d)$ is a full-dimensional convex cone, every maximal cone $\tau$ of $\hat\Sigma$ is full-dimensional.
		Choose a set $B_\tau\subseteq\tau(1)$ of $d$ rays with linearly independent generators, and let $m_\tau(b)\in M_\RR$ be the unique element such that
		\[
			\innprod{m_\tau(b)}{u_\rho}=-c_\rho(b)\qquad(\rho\in B_\tau).
		\]
		Then $m_\tau\colon\RR^n\to M_\RR$ is a linear map defined over $\QQ$.
		For $b\in\QQ^n$, the divisor $\hat D(b)$ is $\QQ$-Cartier if and only if, for every maximal cone $\tau$ of $\hat\Sigma$, some $m\in M_\QQ$ satisfies $\innprod{m}{u_\rho}=-c_\rho(b)$ for every $\rho\in\tau(1)$.
		Such an $m$ must be $m_\tau(b)$, so this condition means that
		\[
			\innprod{m_\tau(b)}{u_\rho}+c_\rho(b)=0\qquad(\rho\in\tau(1))
		\]
		for every maximal cone $\tau$ of $\hat\Sigma$.
		In this case, $\hat D(b)$ is ample over $\tilde Y_{p,F}$ if and only if its support function is strictly convex on every maximal cone of $\tilde\Sigma_{p,F}$, that is,
		\[
			\innprod{m_\tau(b)}{u_\rho}+c_\rho(b)>0\qquad(\rho\in\sigma_\tau(1)\setminus\tau(1))
		\]
		for every maximal cone $\tau$ of $\hat\Sigma$, where $\sigma_\tau$ is the unique maximal cone of $\tilde\Sigma_{p,F}$ containing $\tau$.
		Moreover, if $b\in\RR^n$ satisfies these conditions, then $\hat D(b)$ is $\RR$-Cartier and ample over $\tilde Y_{p,F}$, since $b$ is a positive linear combination of rational vectors satisfying them.

		Consider the linear forms
		\[
			f_{p,\hat\Sigma,\tau,\rho}\colon\RR^n\longrightarrow\RR,\qquad
			b\longmapsto\innprod{m_\tau(b)}{u_\rho}+c_\rho(b)
		\]
		above for all the chosen points $p$, all the refinements $\hat\Sigma$, all their maximal cones $\tau$, and all the rays $\rho$ of $\hat\Sigma$.
		Then
		\[
			V\coloneq\bigcap_{f_{p,\hat\Sigma,\tau,\rho}(a)=0}\Ker f_{p,\hat\Sigma,\tau,\rho}\subseteq\RR^n
		\]
		is a rational linear subspace containing $a$.
		Since $-D=-\sum_ia_iE_i$ is $g_F$-ample and $F$ is rational polyhedral, the $g_F$-ampleness of $-\sum_ib_iE_i$ for $b\in\RR^n$ is given by finitely many strict inequalities in $b$ that hold at $b=a$.
		We choose a rational vector $a'\in V$ so close to $a$ that all the forms $f_{p,\hat\Sigma,\tau,\rho}$ not vanishing at $a$ keep their signs and $-\sum_ia'_iE_i$ is $g_F$-ample.
		Then every form has the same sign at $a$ and at $a'$.
		Since $a'$ is rational, the claim follows from the description above.
	\end{proof}

	Let $a'$ be as in \cref{claim:rationalflip}, and set $D'\coloneq\sum_ia'_iE_i$.
	Since $-D'$ is $g_F$-ample, the rational case gives the $D'$-flip $g_F^+\colon\tilde X^+\to\tilde X_F$.
	We show that it is also a $D$-flip.
	For a chosen point $p$, let $g_{F_p}^+\colon\tilde Y_p^+\to\tilde Y_{p,F}$ be the $D(a')$-flip from the rational case.
	It is a small projective birational toric morphism, so $\tilde Y_p^+=\hat Y$ for one of the refinements $\hat\Sigma$ above.
	The divisor $\hat D(a')$ is $\QQ$-Cartier and $g_{F_p}^+$-ample.
	Hence $\hat D(a)$ is $\RR$-Cartier and $g_{F_p}^+$-ample by \cref{claim:rationalflip}.

	Let $D^+$ be the strict transform of $D$ on $\tilde X^+$.
	By the rational case, the restriction of $\tilde X^+$ over $U_p$ is the base change of $\tilde Y_p^+$ along $\tilde\phi_{p,F}$.
	Since both flips are small, the restriction of $D^+$ is the pullback of $\hat D(a)$.
	Hence $D^+$ is $\RR$-Cartier and $g_F^+$-ample over each $U_p$.
	Over $X_0$, the morphism $g_F^+$ is an isomorphism and $D^+$ is zero.
	Therefore $D^+$ is $\RR$-Cartier and $g_F^+$-ample, and $g_F^+$ is a $D$-flip.
\end{proof}

We next determine when this flipped model can be obtained by changing the order of the blowups.

\begin{thm}\label{thm:reorderflip}
	In \cref{set:main}, let $F$ be a face of $\NEb(\tilde X/X)$ whose contraction $g_F\colon\tilde X\to\tilde X_F$ is small.
	Let $\sigma\in\mathfrak S_n$ be a permutation such that the reordered centers
	\[
		Z_{\sigma(1)},\dotsc,Z_{\sigma(n)}
	\]
	satisfy the assumptions of \cref{set:main}, that is, $Z_{\sigma(h)}\not\supseteq Z_{\sigma(i)}$ whenever $h<i$.
	Let $\pi^\sigma\colon\tilde X^\sigma\to X$ be the successive blowup along these centers in this order.
	For $1\le i\le n$, write $E_i^\sigma$ for the pullback to $\tilde X^\sigma$ of the exceptional divisor of the blowup along the strict transform of $Z_i$.

	Let $D$ be an $\RR$-Cartier divisor on $\tilde X$ such that $-D$ is $g_F$-ample, and let $H$ be an $\RR$-Cartier supporting divisor of $F$.
	Let $D^\sigma$ and $H^\sigma$ be the strict transforms of $D$ and $H$ on $\tilde X^\sigma$, respectively.
	Then $D^\sigma$ and $H^\sigma$ are $\RR$-Cartier, and the following are equivalent.
	\begin{enumerate}
		\item\label{item:reordernef}
			The divisor $H^\sigma$ is $\pi^\sigma$-nef, and
			\[
				D^\sigma\cdot\alpha>0\qquad
				\text{for every }\alpha\in\NEb(\tilde X^\sigma/X)\cap(H^\sigma)^\perp\setminus\{0\}.
			\]
		\item\label{item:reorderflip}
			There is a commutative diagram over $X$
			\[
			\begin{tikzcd}[column sep=small,row sep=small]
				\tilde X\arrow[rr,dashed]\arrow[dr,"g_F"']
					&& \tilde X^\sigma\arrow[dl,"g_F^\sigma"]\\
				& \tilde X_F &
			\end{tikzcd}
			\]
			in which $g_F^\sigma$ is the $D$-flip of $g_F$.
	\end{enumerate}
	If $F$ is an extremal ray, then \cref{item:reordernef} holds if and only if $H^\sigma$ is $\pi^\sigma$-nef and the birational map $\tilde X\dashrightarrow\tilde X^\sigma$ over $X$ is not an isomorphism.
\end{thm}

\begin{proof}
	Let $\varphi\colon\tilde X\dashrightarrow\tilde X^\sigma$ be the birational map over $X$.
	As in the proof of \cref{lem:valuation}, for each $i$, both $E_i$ and $E_i^\sigma$ define the $\Ic_{Z_i}$-adic valuation of $k(X)$.
	All other prime divisors on $\tilde X$ and on $\tilde X^\sigma$ are strict transforms of prime divisors on $X$.
	Thus $\varphi$ is an isomorphism in codimension one and sends each $E_i$ to $E_i^\sigma$.
	Applying \cref{lem:divisordecomposition} by $\RR$-linearity, we can write every $\RR$-Cartier divisor on $\tilde X$ as $\pi^*D_0+\sum_ib_iE_i$, where $D_0$ is an $\RR$-Cartier divisor on $X$ and $b_i\in\RR$.
	Since $E_i$ and $E_i^\sigma$ define the same valuation, its strict transform on $\tilde X^\sigma$ is $(\pi^\sigma)^*D_0+\sum_ib_iE_i^\sigma$.
	Hence $D^\sigma$ and $H^\sigma$ are $\RR$-Cartier.

	Subtracting a pullback from $X$, we write $H=-\sum_ia_iE_i$ with $a_i\in\RR$.
	We may replace $H$ by an integral combination of the $E_i$ without changing its zero face on either model or the $\pi^\sigma$-nefness of $H^\sigma$.
	Indeed, by \cref{thm:cone} applied to both orderings, the cones $\NEb(\tilde X/X)$ and $\NEb(\tilde X^\sigma/X)$ are generated by finitely many integral classes.
	The intersection numbers of these classes with $-\sum_ib_iE_i$ and $-\sum_ib_iE_i^\sigma$ are rational linear forms in $b=(b_1,\dotsc,b_n)$.
	The forms vanishing at $a$ define a rational linear subspace containing $a$.
	A rational vector in this subspace sufficiently close to $a$ gives every other form the same sign as $a$, and clearing its denominators gives the required integral divisor.
	Thus we may assume that every $a_i$ is an integer, and $a_i\ge0$ by \cref{lem:fiberclass}.

	Assume that $H^\sigma$ is $\pi^\sigma$-nef, and set
	\[
		F^\sigma\coloneq\NEb(\tilde X^\sigma/X)\cap(H^\sigma)^\perp.
	\]
	Then $H^\sigma=-\sum_ia_iE_i^\sigma$ is a supporting divisor of the face $F^\sigma$.
	By \cref{thm:targetn} applied to both orderings, the contractions of $F$ and $F^\sigma$ have the same target
	\[
		\tilde X_F\isom\overline{\Bl}_{\Ic^{(a)}}X
	\]
	over $X$.
	Both $\pi$ and $\pi^\sigma$ are isomorphisms over $X\setminus\bigcup_iZ_i$, so the contraction $g_F^\sigma\colon\tilde X^\sigma\to\tilde X_F$ of $F^\sigma$ satisfies $g_F^\sigma\circ\varphi=g_F$ as rational maps.
	Since $g_F$ is small and $\varphi$ is an isomorphism in codimension one, $g_F^\sigma$ is also small.
	Since $\NEb(\tilde X^\sigma/\tilde X_F)=F^\sigma$, the positivity condition in \cref{item:reordernef} holds if and only if $D^\sigma$ is $g_F^\sigma$-ample, that is, if and only if $g_F^\sigma$ is the $D$-flip of $g_F$.
	This proves \cref{item:reordernef}$\Rightarrow$\cref{item:reorderflip}.

	Conversely, assume that \cref{item:reorderflip} holds.
	By the proof of \cref{thm:targetn}, there is a Cartier divisor $A$ on $\tilde X_F$ that is ample over $X$ and satisfies $H\lineq g_F^*A$.
	Since $g_F$ and $g_F^\sigma$ are small, taking strict transforms gives $H^\sigma\lineq(g_F^\sigma)^*A$.
	Hence $H^\sigma$ is $\pi^\sigma$-nef, and $\NEb(\tilde X^\sigma/X)\cap(H^\sigma)^\perp=\NEb(\tilde X^\sigma/\tilde X_F)$.
	Since $D^\sigma$ is $g_F^\sigma$-ample, it is positive on every nonzero class of this cone.
	This proves \cref{item:reorderflip}$\Rightarrow$\cref{item:reordernef}.

	Finally, assume that $F=R$ is an extremal ray.
	If \cref{item:reordernef} holds, then so does \cref{item:reorderflip}, and $\varphi$ is not an isomorphism, since $-D$ is $g_R$-ample and $D^\sigma$ is $g_R^\sigma$-ample.
	Conversely, assume that $H^\sigma$ is $\pi^\sigma$-nef and $\varphi$ is not an isomorphism.
	We show that \cref{item:reorderflip} holds.
	By \cref{lem:flipindependence}, we may assume that $D=\sum_ib_iE_i$ with $b_i\in\ZZ$.
	The argument above gives the small contraction $g_R^\sigma\colon\tilde X^\sigma\to\tilde X_R$ of $F^\sigma$ with $g_R^\sigma\circ\varphi=g_R$.
	Strict transform identifies $N^1(\tilde X/X)$ with $N^1(\tilde X^\sigma/X)$, sending $[E_i]$ to $[E_i^\sigma]$.
	Since $g_R$ and $g_R^\sigma$ are small, it also identifies the images of $N^1(\tilde X_R/X)$ under $g_R^*$ and $(g_R^\sigma)^*$.
	By \cref{cor:contract} applied to both models, $N^1(\tilde X^\sigma/\tilde X_R)$ is one-dimensional, and $F^\sigma$ is an extremal ray.
	The strict transform $D^\sigma=\sum_ib_iE_i^\sigma$ has nonzero class in $N^1(\tilde X^\sigma/\tilde X_R)$, so either $D^\sigma$ or $-D^\sigma$ is $g_R^\sigma$-ample.
	Suppose that $-D^\sigma$ is $g_R^\sigma$-ample.
	Since $g_R$ and $g_R^\sigma$ are small and $\varphi$ is an isomorphism in codimension one, normality shows that $\bigoplus_{m\ge0}(g_R)_*\Ob{\tilde X}(-mD)$ and $\bigoplus_{m\ge0}(g_R^\sigma)_*\Ob{\tilde X^\sigma}(-mD^\sigma)$ are the same algebra.
	Both $\tilde X$ and $\tilde X^\sigma$ are the relative $\Proj$ of this algebra over $\tilde X_R$, so $\varphi$ is an isomorphism.
	This is a contradiction.
	Therefore $D^\sigma$ is $g_R^\sigma$-ample, and $g_R^\sigma$ is the $D$-flip of $g_R$.
\end{proof}

\subsection{Exceptional loci and contraction types}\label{sec:types}

In this subsection, we describe the exceptional locus of the contraction of an extremal ray $R$ and show that the contraction is either divisorial or small in the sense of \cref{ss:conventions}.

We first recall the following result of Sato \cite{Sato} on toric extremal contractions.

\begin{thm}\label{thm:sato}
	Let $X_\Sigma$ be a $\QQ$-factorial toric variety, let $X_\Sigma\to S$ be a projective toric morphism to an affine toric variety, and let $R$ be an extremal ray of $\NEb(X_\Sigma/S)$.
	Assume that its contraction $g_R\colon X_\Sigma\to (X_\Sigma)_R$ exists and is birational.
	Put
	\[
		\mathcal N(R)\coloneq\sset{\rho\in\Sigma(1)}{D_\rho\cdot R<0},\qquad
		\tau_R\coloneq\Cone(u_\rho\mid\rho\in\mathcal N(R)).
	\]
	Then $\tau_R$ is a cone of $\Sigma$, and
	\[
		\Exc(g_R)=V(\tau_R),\qquad
		\codim_{X_\Sigma}\Exc(g_R)=\card{\mathcal N(R)}.
	\]
\end{thm}

\begin{proof}
	This is the birational case of \cite[Theorem~3.4]{Sato}, as described in \cite[\nopp 3.6]{Sato}.
	By \cite[Proposition~2.7]{Sato}, the cone $w'$ and the integer $m$ in that description are precisely $\tau_R$ and $\card{\mathcal N(R)}$, respectively.
\end{proof}

We next introduce the index sets used to describe the exceptional locus.

\begin{defi}\label{def:raysets}
	In \cref{set:main}, let $R$ be an extremal ray of $\NEb(\tilde X/X)$.
	The \emph{source} of $R$ is the minimal $i_0$ such that $E_{i_0}\cdot R\ne0$.
	Set
	\[
		I^-(R)\coloneq\sset{i\in\{1,\dotsc,n\}}{E_i\cdot R<0},
		\qquad
		Z_R\coloneq\bigcap_{E_i\cdot R\ne0}Z_i\subseteq X.
	\]
	Let $p\in\bigcup_iZ_i$ be a closed point, and fix a coordinate model at $p$, with coordinate sets $J_i$ for the centers $Z_i$ containing $p$.
	For $1\le j\le d$, set $I_j\coloneq\sset{i}{p\in Z_i,\ j\in J_i}$, and
	\[
		J^-_p(R)\coloneq\biggl\{j\in\{1,\dotsc,d\}\biggm|\biggl(\sum_{i\in I_j}E_i\biggr)\cdot R>0\biggr\}.
	\]
	In \cref{set:coordinate}, we use the same notation for an extremal ray $R$ of $\NEb(\tilde Y/\AA^d)$, with $F_i$ in place of $E_i$, and we omit the subscript $p$.
\end{defi}

Note that the source $i_0$ of $R$ belongs to $I^-(R)$.
Indeed, by \cref{thm:cone}, the ray $R$ is spanned by the class of an elementary curve, and the entry of this class at $i_0$ is $-1$.
In \cref{set:coordinate}, \cref{lem:D_e_j} shows that $J^-(R)$ is the set of $j$ with $D_{e_j}\cdot R<0$.

We now apply \cref{thm:sato} to coordinate blowups.

\begin{lem}\label{lem:coordlocus}
	In \cref{set:coordinate}, let $R$ be an extremal ray of $\NEb(\tilde Y/\AA^d)$, and let $g_R\colon\tilde Y\to\tilde Y_R$ be its contraction.
	Then
	\[
		\Exc(g_R)=\bigcap_{i\in I^-(R)}F_i\cap\bigcap_{j\in J^-(R)}D_{e_j},\qquad
		\pi_Y(\Exc(g_R))=\bigcap_{F_i\cdot R\ne0}W_i,
	\]
	and the intersection in the first equality is transverse.
\end{lem}

\begin{proof}
	The contraction $g_R$ exists by \cref{cor:contract}, and it is birational because $\pi_Y$ factors through it.
	Since $\tilde Y$ is smooth by \cref{lem:inherit} and $\pi_Y$ is a projective toric morphism, we may apply \cref{thm:sato} to $R$.

	By \cref{lem:validstar,lem:D_e_j}, the rays $\rho\in\tilde\Sigma(1)$ with $D_\rho\cdot R<0$ are the rays $\rho_{v_i}$ with $i\in I^-(R)$ and $\rho_{e_j}$ with $j\in J^-(R)$.
	Hence, by \cref{thm:sato}, the cone
	\[
		\tau_R=\Cone(v_i\mid i\in I^-(R))+\Cone(e_j\mid j\in J^-(R))
	\]
	belongs to $\tilde\Sigma$, and $\Exc(g_R)=V(\tau_R)$.
	Since $\tilde\Sigma$ is smooth, $V(\tau_R)$ is the transverse intersection of the divisors $F_i=D_{v_i}$ and $D_{e_j}$ for these rays.
	This proves the first equality and the transversality.

	For the second equality, let $\sigma$ be the smallest face of $\Cone(e_1,\dotsc,e_d)$ containing $\tau_R$.
	Since $\pi_Y$ is induced by the identity map of $N$, it maps $V(\tau_R)$ onto $V(\sigma)$.
	Since $v_i=\sum_{j\in J_i}e_j$, we have
	\[
		\sigma=\Cone(e_j\mid j\in J),
		\qquad
		J\coloneq\bigcup_{i\in I^-(R)}J_i\cup J^-(R).
	\]
	Thus $V(\sigma)=V(x_j\mid j\in J)$, and it remains to show that $J=\bigcup_{F_i\cdot R\ne0}J_i$.

	Let $j\in J$.
	If $j\in J_i$ for some $i\in I^-(R)$, then $F_i\cdot R\ne0$.
	If $j\in J^-(R)$, then $(\sum_{i\in I_j}F_i)\cdot R>0$, so $F_i\cdot R>0$ for some $i\in I_j$, and $j\in J_i$ for this $i$.
	Hence $J\subseteq\bigcup_{F_i\cdot R\ne0}J_i$.

	Conversely, let $j\in J_i$ with $F_i\cdot R\ne0$.
	If $F_i\cdot R<0$, then $i\in I^-(R)$, so $j\in J$.
	Assume that $F_i\cdot R>0$.
	If $j\in J_h$ for some $h\in I^-(R)$, then $j\in J$.
	Otherwise, $F_h\cdot R\ge0$ for every $h\in I_j$.
	Since $i\in I_j$ and $F_i\cdot R>0$, it follows that $j\in J^-(R)\subseteq J$.
	Hence $J=\bigcup_{F_i\cdot R\ne0}J_i$, and
	\[
		\pi_Y(\Exc(g_R))=V(x_j\mid j\in J)=\bigcap_{F_i\cdot R\ne0}W_i. \qedhere
	\]
\end{proof}

The next proposition transfers \cref{lem:coordlocus} to $\tilde X$ through the coordinate models.

\begin{prop}\label{prop:locus}
	In \cref{set:main}, let $R$ be an extremal ray of $\NEb(\tilde X/X)$, and let $g_R\colon\tilde X\to\tilde X_R$ be its contraction.
	Then the following hold.
	\begin{enumerate}
		\item\label{item:locuslocal}
			Let $p\in\pi(\Exc(g_R))$ be a closed point, and let $\pi_{Y_p}\colon\tilde Y_p\to\AA^d$ together with $\phi_p\colon U_p\to\AA^d$ be a coordinate model at $p$.
			For $1\le j\le d$, let $\tilde H_j\subseteq\tilde U_p$ be the strict transform of the hypersurface $H_j\coloneq\phi_p^{-1}V(x_j)\subseteq U_p$.
			Then
			\[
				\Exc(g_R)\cap\tilde U_p
				=\bigcap_{i\in I^-(R)}(E_i\cap\tilde U_p)\cap\bigcap_{j\in J^-_p(R)}\tilde H_j,
			\]
			and the intersection on the right is transverse.
			In particular, $\Exc(g_R)\cap\tilde U_p$ is smooth of codimension $\card{I^-(R)}+\card{J^-_p(R)}$ in $\tilde U_p$.
			Moreover, $\pi(\Exc(g_R)\cap\tilde U_p)=Z_R\cap U_p$.
		\item\label{item:locusglobal}
			The exceptional locus $\Exc(g_R)$ is smooth and contained in $\bigcap_{i\in I^-(R)}E_i$.
			Its image $\pi(\Exc(g_R))$ is a union of connected components of $Z_R$.
	\end{enumerate}
\end{prop}

\begin{proof}
	Put $R_p\coloneq\iota_p^{-1}(R)$, where $\iota_p\colon\NEb(\tilde Y_p/\AA^d)\hookrightarrow\NEb(\tilde X/X)$ is the injection of \cref{lem:NEinj}.
	By \cref{lem:localcontract}, $R_p$ is an extremal ray with $\iota_p(R_p)=R$, and
	\[
		\Exc(g_R)\cap\tilde U_p=\tilde\phi_p^{-1}(\Exc(g_{R_p})),
	\]
	where $g_{R_p}\colon\tilde Y_p\to\tilde Y_{p,R}$ is the contraction of $R_p$.
	Since $\iota_p$ preserves intersection vectors, we have $I^-(R)=I^-(R_p)$ and $J^-_p(R)=J^-(R_p)$, and $F_i\cdot R_p\ne0$ if and only if $E_i\cdot R\ne0$.
	Moreover, $\restrdiv{E_i}{\tilde U_p}=\tilde\phi_p^{\,*}F_i$ by \cref{prop:localmodel}, and $\tilde H_j=\tilde\phi_p^{-1}(D_{e_j})$ because strict transforms commute with flat base change.
	Applying \cref{lem:coordlocus} to $R_p$ and pulling back along the \'etale morphism $\tilde\phi_p$ gives the displayed equality and the transversality.
	Since the square in \cref{prop:localmodel} is cartesian, \cref{lem:coordlocus} also gives
	\[
		\pi(\Exc(g_R)\cap\tilde U_p)
		=\phi_p^{-1}\biggl(\bigcap_{F_i\cdot R_p\ne0}W_i\biggr)
		=Z_R\cap U_p.
	\]
	This proves \cref{item:locuslocal}.

	The open subsets $\tilde U_p$ for $p\in\pi(\Exc(g_R))$ cover $\Exc(g_R)$.
	Hence \cref{item:locuslocal} shows that $\Exc(g_R)$ is smooth and contained in $E_i$ for every $i\in I^-(R)$.
	By \cref{item:locuslocal}, we also have $\pi(\Exc(g_R))\cap U_p=Z_R\cap U_p$ for every $p\in\pi(\Exc(g_R))$.
	Thus $\pi(\Exc(g_R))$ is an open subset of $Z_R$.
	It is also closed, since $\pi$ is proper.
	Hence it is a union of connected components of $Z_R$.
\end{proof}

\begin{rem}
	The image $\pi(\Exc(g_R))$ in \cref{prop:locus} can be a proper subset of $Z_R$; see \cref{exa:locusproper}.
\end{rem}

Recall from \cref{ss:conventions} that we call a birational extremal ray contraction divisorial if its exceptional locus is a prime divisor.

\begin{cor}\label{cor:typen}
	In \cref{set:main}, let $R$ be an extremal ray of $\NEb(\tilde X/X)$ with source $i_0$, and let $g_R\colon\tilde X\to\tilde X_R$ be its contraction.
	Then the following are equivalent.
	\begin{enumerate}
		\item\label{item:gRdivisorial}
			The contraction $g_R$ is divisorial.
		\item\label{item:divevery}
			The set $I^-(R)$ is $\{i_0\}$, and $J^-_p(R)=\emptyset$ for every closed point $p\in\pi(\Exc(g_R))$ and every coordinate model at $p$.
		\item\label{item:divsome}
			The set $I^-(R)$ is $\{i_0\}$, and $J^-_p(R)=\emptyset$ for some closed point $p\in\pi(\Exc(g_R))$ and some coordinate model at $p$.
	\end{enumerate}
	If these conditions hold, then $\Exc(g_R)=E_{i_0}$ and $\pi(\Exc(g_R))=Z_{i_0}$.
	Otherwise, $g_R$ is small.
\end{cor}

\begin{proof}
	Assume first that $g_R$ is divisorial, so that $\Exc(g_R)$ is a prime divisor.
	Let $p\in\pi(\Exc(g_R))$ be a closed point, and choose a coordinate model at $p$.
	Then $\Exc(g_R)\cap\tilde U_p$ has codimension one in $\tilde U_p$.
	By \cref{prop:locus}~\cref{item:locuslocal},
	\[
		\card{I^-(R)}+\card{J^-_p(R)}=1.
	\]
	Since $i_0\in I^-(R)$, it follows that $I^-(R)=\{i_0\}$ and $J^-_p(R)=\emptyset$.
	This proves \cref{item:gRdivisorial}$\Rightarrow$\cref{item:divevery}.
	The implication \cref{item:divevery}$\Rightarrow$\cref{item:divsome} holds because $\Exc(g_R)$ is nonempty.

	Conversely, assume that \cref{item:divsome} holds for a closed point $p\in\pi(\Exc(g_R))$ and a coordinate model at $p$.
	By \cref{prop:locus}~\cref{item:locuslocal},
	\[
		\Exc(g_R)\cap\tilde U_p=E_{i_0}\cap\tilde U_p,
	\]
	and this set is nonempty.
	By \cref{prop:locus}~\cref{item:locusglobal}, $\Exc(g_R)$ is a closed subset of the prime divisor $E_{i_0}$, and it contains the nonempty open subset $E_{i_0}\cap\tilde U_p$ of $E_{i_0}$.
	Hence $\Exc(g_R)=E_{i_0}$, so $g_R$ is divisorial and $\pi(\Exc(g_R))=Z_{i_0}$.
	This proves \cref{item:divsome}$\Rightarrow$\cref{item:gRdivisorial}, together with the descriptions of $\Exc(g_R)$ and its image.

	Finally, assume that these conditions fail.
	Since $i_0\in I^-(R)$, we have $\card{I^-(R)}+\card{J^-_p(R)}\ge2$ for every closed point $p\in\pi(\Exc(g_R))$ and every coordinate model at $p$.
	By \cref{prop:locus}~\cref{item:locuslocal}, $\Exc(g_R)\cap\tilde U_p$ then has codimension at least two in $\tilde U_p$.
	These open subsets cover $\Exc(g_R)$, so $g_R$ is small.
\end{proof}

\subsection{Examples of contractions and flops}\label{sec:examples}

We first describe the contractions for two centers, using the relative cone of curves computed in \cref{prop:twocenters}.
This recovers the two-center construction of \cite[Theorem~3.7 and Corollary~3.8]{MasamuraYoshida}.

\begin{cor}\label{cor:MY}
	In \cref{set:main}, assume that $n=2$.
	Then the following hold.
	\begin{enumerate}
		\item\label{item:MYtransverse}
			Assume that $Z_1$ and $Z_2$ meet transversely.
			Then the contractions of the extremal rays $\RRp(-1,0)$ and $\RRp(0,-1)$ are divisorial, with exceptional loci $E_1$ and $E_2$ and targets $\Bl_{Z_2}X$ and $X^{(1)}=\Bl_{Z_1}X$, respectively.
		\item\label{item:MYcontraction}
			Assume that some irreducible component of $Z_1\cap Z_2$ has excess dimension, and let $Z^+\subseteq Z_1\cap Z_2$ be the union of these components.
			Then the contraction of the extremal ray $\RRp(0,-1)$ is the blowdown $\pi_2\colon\tilde X\to X^{(1)}$.
			Moreover, the contraction $g_R\colon\tilde X\to\tilde X_R$ of the other extremal ray $R\coloneq\RRp(-1,1)$ satisfies the following.
			\begin{enumerate}
				\item\label{item:MYtarget}
					The target $\tilde X_R$ is isomorphic to $\Bl_{Z_1\cup Z_2}X$ over $X$, and
					\[
						\Exc(g_R)=(\pi_2)_*^{-1}\pi_1^{-1}(Z^+).
					\]
				\item\label{item:MYtype}
					The contraction $g_R$ is small if and only if $Z_1\not\subseteq Z_2$.
					Moreover, if $K_X$ is $\QQ$-Cartier, then $-K_{\tilde X}$ is $g_R$-ample if and only if $c_1>c_2$.
				\item\label{item:MYflip}
					Assume that $Z_1\not\subseteq Z_2$ and that $K_X$ is $\QQ$-Cartier, and let $\tilde X^\sigma$ be the successive blowup along $Z_2$ and then $Z_1$.
					If $c_1>c_2$, then the induced birational map $\tilde X\dashrightarrow\tilde X^\sigma$ is the flip of $g_R$.
					If $c_1=c_2$, then it is a flop.
			\end{enumerate}
	\end{enumerate}
\end{cor}

\begin{proof}
	We first prove \cref{item:MYtransverse}.
	By \cref{prop:twocenters}, the extremal rays of $\NEb(\tilde X/X)$ are $R_1\coloneq\RRp(-1,0)$ and $R_2\coloneq\RRp(0,-1)$.
	The divisors $-E_2$ and $-E_1$ are supporting divisors of $R_1$ and $R_2$, respectively.
	By \cref{thm:targetn}, the targets of their contractions are $\Bl_{Z_2}X$ and $\Bl_{Z_1}X$, since these blowups are normal by \cref{lem:inherit}.
	We have $I^-(R_1)=\{1\}$ and $I^-(R_2)=\{2\}$.
	Since no entry of $(-1,0)$ or $(0,-1)$ is positive, we also have $J^-_p(R_1)=J^-_p(R_2)=\emptyset$ for every closed point $p$ and every coordinate model at $p$.
	Thus \cref{cor:typen} shows that both contractions are divisorial, with exceptional loci $E_1$ and $E_2$, respectively.
	This proves \cref{item:MYtransverse}.

	We next prove \cref{item:MYcontraction}.
	By \cref{prop:twocenters},
	\[
		\NEb(\tilde X/X)=\RRp[(-1,1),(0,-1)].
	\]
	The divisor $-E_1=\pi_2^*(-E_1^{(1)})$ is a supporting divisor of $\RRp(0,-1)$, and $-E_1^{(1)}$ is $\pi_1$-ample.
	Hence the blowdown $\pi_2$ is the contraction of $\RRp(0,-1)$.
	Moreover, $-E_1-E_2$ is a supporting divisor of $R$.
	By \cref{thm:targetn}, the target $\tilde X_R$ is the normalization of $\Bl_{\Ic_{Z_1}\cap\Ic_{Z_2}}X=\Bl_{Z_1\cup Z_2}X$.
	This blowup is normal by the proof of \cite[Theorem~3.7]{MasamuraYoshida}.
	Thus $\tilde X_R\isom\Bl_{Z_1\cup Z_2}X$.

	We next compute the exceptional locus.
	Let $p$ be a closed point of $Z_1\cap Z_2$.
	If $p\in Z^+$, then the proof of \cref{prop:twocenters} gives a curve $C\subseteq\pi^{-1}(p)$ with intersection vector $(-1,1)$, so $p\in\pi(\Exc(g_R))$.
	If $p\notin Z^+$, then the proof of \cref{prop:twocenters} shows that $C_2\le0$ for every curve $C\subseteq\pi^{-1}(p)$.
	Hence $p\notin\pi(\Exc(g_R))$.
	Since $Z_R=Z_1\cap Z_2$, \cref{prop:locus}~\cref{item:locusglobal} gives $\pi(\Exc(g_R))\subseteq Z_1\cap Z_2$, and therefore $\pi(\Exc(g_R))=Z^+$.
	Let $p\in Z^+$ be a closed point, choose a coordinate model at $p$, and set $H_j\coloneq\phi_p^{-1}V(x_j)$ for $1\le j\le d$.
	Since $Z_1\cap Z_2\cap U_p=\phi_p^{-1}(W_1\cap W_2)$ has codimension $\card{J_1\cup J_2}<c_1+c_2$ in $U_p$, it is contained in $Z^+$.
	Since $R=\RRp(-1,1)$, we have $I^-(R)=\{1\}$ and $J^-_p(R)=J_2\setminus J_1$.
	Put $G_p\coloneq\bigcap_{j\in J_2\setminus J_1}H_j$.
	This subvariety meets $Z_1\cap U_p$ transversely along $Z_1\cap Z_2\cap U_p$, so its strict transform $G_p^{(1)}$ on $X^{(1)}$ satisfies
	\[
		E_1^{(1)}\cap G_p^{(1)}=\pi_1^{-1}(Z^+\cap U_p).
	\]
	Taking strict transforms under $\pi_2$ and applying \cref{prop:locus}~\cref{item:locuslocal}, we obtain
	\begin{align*}
		\Exc(g_R)\cap\tilde U_p
		&=(E_1\cap\tilde U_p)\cap\bigcap_{j\in J_2\setminus J_1}\tilde H_j\\
		&=(\pi_2)_*^{-1}\pi_1^{-1}(Z^+)\cap\tilde U_p.
	\end{align*}
	Since these open subsets $\tilde U_p$ cover $\pi^{-1}(Z^+)$, this proves \cref{item:MYtarget}.

	By \cref{cor:typen}, the contraction $g_R$ is divisorial if and only if $J_2\subseteq J_1$ at some point of $Z^+$, which is equivalent to $Z_1\subseteq Z_2$.
	Otherwise, $g_R$ is small.
	Moreover, if $K_X$ is $\QQ$-Cartier, then so is $K_{\tilde X}$, and the blowup formula gives
	\[
		K_{\tilde X}\cdot(-1,1)=-(c_1-1)+(c_2-1)=c_2-c_1.
	\]
	Since $\rho(\tilde X/\tilde X_R)=1$ by \cref{cor:contract}, this proves \cref{item:MYtype}.

	Finally, assume that $Z_1\not\subseteq Z_2$ and that $K_X$ is $\QQ$-Cartier, as in \cref{item:MYflip}.
	Then the order $Z_2,Z_1$ also satisfies \cref{set:main}.
	Let $E_1^\sigma$ and $E_2^\sigma$ be the exceptional divisors of $\tilde X^\sigma$ over $Z_1$ and $Z_2$, respectively, and identify each class $\alpha\in N_1(\tilde X^\sigma/X)$ with the vector $(E_1^\sigma\cdot\alpha,E_2^\sigma\cdot\alpha)$.
	Applying \cref{prop:twocenters} to the order $Z_2,Z_1$, we obtain
	\[
		\NEb(\tilde X^\sigma/X)=\RRp[(1,-1),(-1,0)].
	\]
	Thus $H^\sigma=-E_1^\sigma-E_2^\sigma$ is $\pi^\sigma$-nef.
	Since the two cones are different, the induced map $\varphi\colon\tilde X\dashrightarrow\tilde X^\sigma$ is not an isomorphism.
	By \cref{thm:reorderflip}, the reordered blowup gives the $D$-flip of $g_R$ for every $\RR$-Cartier divisor $D$ with $D\cdot R<0$.
	If $c_1>c_2$, then taking $D=K_{\tilde X}$ shows that $\varphi$ is the flip of $g_R$.
	If $c_1=c_2$, then $K_{\tilde X}$ is numerically trivial over $\tilde X_R$, and taking $D=E_1$ shows that $\varphi$ is a flop.
	This proves \cref{item:MYflip}.
\end{proof}

The following example shows that $Z^+$ in \cref{cor:MY}~\cref{item:MYcontraction} can be a proper subset of $Z_1\cap Z_2$.

\begin{exa}\label{exa:locusproper}
	Let
	\[
		X=\AA^4_{x,y,z,w},\qquad Z_1=V(x,y),\qquad Z_2=V\bigl(x-z(z-1),y-z w\bigr).
	\]
	These centers are smooth of codimension two, and their intersection
	\[
		Z_1\cap Z_2=V(x,y,z)\sqcup V(x,y,z-1,w)
	\]
	is the disjoint union of a line and a point.
	Hence \cref{prop:twocenters}~\cref{item:twoexcess} gives
	\[
		\NEb(\tilde X/X)=\RRp[(-1,1),(0,-1)].
	\]
	By \cref{cor:MY}, the contraction $g_R$ of $R\coloneq\RRp(-1,1)$ satisfies
	\[
		\pi(\Exc(g_R))=V(x,y,z)\subsetneq Z_1\cap Z_2=Z_R.
	\]
\end{exa}

We now compute contraction targets, exceptional loci, and flops in two examples with three centers.
The first flop cannot be obtained by reordering the centers, whereas the second can.

\begin{exa}\label{exa:branch}
	Let $X=\AA^5_{x_1,\dotsc,x_5}$ and
	\[
		Z_1=V(x_1,x_2,x_3),\qquad
		Z_2=V(x_2,x_4),\qquad
		Z_3=V(x_3,x_5).
	\]
	These centers satisfy \cref{set:coordinate}, so $\tilde X$ is a toric variety.
	By \cref{thm:toriccone} and the wall relations \eqref{eq:wallrelation},
	\[
		\NEb(\tilde X/X)=\RRp[(-1,1,1),(0,-1,0),(0,0,-1)].
	\]
	The remaining wall classes are $(-1,1,0)$ and $(-1,0,1)$, and each of them is the sum of two of these generators.
	Thus $R\coloneq\RRp(-1,1,1)$ is an extremal ray.
	The divisor $H\coloneq-2E_1-E_2-E_3$ supports $R$, so \cref{cor:contract,thm:targetn} give the contraction
	\[
		g_R\colon\tilde X\longrightarrow\tilde X_R
		\isom\overline{\Bl}_{\Ic_{Z_1}^{2}\cap\Ic_{Z_2}\cap\Ic_{Z_3}}X.
	\]

	Since $R=\RRp(-1,1,1)$, we have $I^-(R)=\{1\}$, $J^-_0(R)=\{4,5\}$, and $Z_R=Z_1\cap Z_2\cap Z_3=\{0\}$.
	Hence \cref{prop:locus} gives
	\[
		\Exc(g_R)=(\pi_2\circ\pi_3)_*^{-1}\pi_1^{-1}(0)\isom\PP^2,
		\qquad
		\pi(\Exc(g_R))=\{0\}.
	\]
	The exceptional locus has codimension $3$ in $\tilde X$, so $g_R$ is small.

	Since $(c_1,c_2,c_3)=(3,2,2)$, the blowup formula gives
	\[
		K_{\tilde X}\cdot(-1,1,1)=-2+1+1=0.
	\]
	Thus $g_R$ is a flopping contraction.
	Since $E_1\cdot R<0$, its flop exists by \cref{thm:flip} with $D=E_1$.

	This flop cannot be obtained by reordering the centers.
	No center contains another, so every ordering satisfies \cref{set:main}.
	For a permutation $\sigma$, index the exceptional divisors of $\tilde X^\sigma$ by the original centers, and put $H^\sigma=-2E_1^\sigma-E_2^\sigma-E_3^\sigma$.
	Assume that $Z_2$ precedes $Z_1$.
	Choose a closed point $p$ in the nonempty set $(Z_1\cap Z_2)\setminus Z_3$.
	Near $p$, only the blowups along $Z_2$ and $Z_1$ occur.
	The exceptional fiber of the blowup along $Z_2$ over $p$ is a line meeting the strict transform of $Z_1$ at one point, so its strict transform has intersection vector $(1,-1,0)$.
	This curve has $H^\sigma$-degree $-1$.
	Similarly, if $Z_3$ precedes $Z_1$, then a point of the nonempty set $(Z_1\cap Z_3)\setminus Z_2$ gives a curve with intersection vector $(1,0,-1)$ and $H^\sigma$-degree $-1$.
	Thus $H^\sigma$ is $\pi^\sigma$-nef only if $Z_1$ comes first.

	After blowing up $Z_1$, the strict transforms of $Z_2$ and $Z_3$ meet transversely.
	Indeed, on the $x_1$-chart they are defined by $(x_2/x_1,x_4)$ and $(x_3/x_1,x_5)$, respectively, and on each of the other two charts one of them is empty.
	Hence the blowups along them commute, and for both orderings with $Z_1$ first, the map $\varphi\colon\tilde X\dashrightarrow\tilde X^\sigma$ is an isomorphism.
	By the last assertion of \cref{thm:reorderflip}, no ordering satisfies \cref{thm:reorderflip}~\cref{item:reordernef}, so none gives the flop.
\end{exa}

\begin{exa}\label{exa:contain}
	Consider the successive blowup $\tilde X\to X$ of \cref{exa:forced}.
	By \cref{thm:toriccone} and the wall relations \eqref{eq:wallrelation},
	\[
		\NEb(\tilde X/X)=\RRp[(-1,1,-1),(0,-1,1),(0,0,-1)].
	\]
	The only other wall class is $(0,-1,0)=(0,-1,1)+(0,0,-1)$.
	Thus $R\coloneq\RRp(-1,1,-1)$ is an extremal ray.
	The divisor $H\coloneq-E_1-2E_2-E_3$ supports $R$, so \cref{cor:contract,thm:targetn} give the contraction
	\[
		g_R\colon\tilde X\longrightarrow\tilde X_R
		\isom\overline{\Bl}_{\Ic_{Z_1}\cap\Ic_{Z_2}^{2}\cap\Ic_{Z_3}}X.
	\]

	Since $R=\RRp(-1,1,-1)$, we have $I^-(R)=\{1,3\}$, $J^-_0(R)=\emptyset$, and $Z_R=Z_1\cap Z_2\cap Z_3=\{0\}$.
	Hence \cref{prop:locus} gives
	\[
		\Exc(g_R)=E_1\cap E_3,\qquad
		\pi(\Exc(g_R))=\{0\}.
	\]
	The exceptional locus has codimension $2$ in $\tilde X$, so $g_R$ is small.

	Since $(c_1,c_2,c_3)=(2,3,2)$, the blowup formula gives
	\[
		K_{\tilde X}\cdot(-1,1,-1)=-1+2-1=0.
	\]
	Thus $g_R$ is a flopping contraction.

	We obtain its flop by reordering the centers.
	Let $\pi^\sigma\colon\tilde X^\sigma\to X$ be the successive blowup in the order
	\[
		Z_2,\quad Z_1,\quad Z_3,
	\]
	which satisfies \cref{set:main}, and index its exceptional divisors by the original centers.
	By \cref{thm:toriccone} and the wall relations,
	\[
		\NEb(\tilde X^\sigma/X)
		=\RRp[(-1,0,0),(0,0,-1),(1,-1,1)].
	\]
	The strict transform
	\[
		H^\sigma=-E_1^\sigma-2E_2^\sigma-E_3^\sigma
	\]
	is $\pi^\sigma$-nef and vanishes on $\NEb(\tilde X^\sigma/X)$ exactly along $\RRp(1,-1,1)$.
	The two relative cones are different, so, as in the proof of \cref{cor:MY}, the induced map $\varphi\colon\tilde X\dashrightarrow\tilde X^\sigma$ is not an isomorphism.
	Since $E_1\cdot R<0$, \cref{thm:reorderflip} with $D=E_1$ gives a commutative diagram
	\[
	\begin{tikzcd}[column sep=small,row sep=small]
		\tilde X\arrow[rr,dashed,"\varphi"]\arrow[dr,"g_R"'] && \tilde X^\sigma\arrow[dl,"g_R^\sigma"]\\
		& \tilde X_R &
	\end{tikzcd}
	\]
	in which $g_R^\sigma$ is the $E_1$-flip, and hence the flop, of $g_R$.
\end{exa}

\section{Relative Cox sheaves and small modifications}\label{sec:cox}

In this section, we compute the relative Cox sheaf of $\pi\colon\tilde X\to X$ and prove that it is finitely generated over $\Ob X$.
We also show that $\pi$ is a Mori dream morphism when $X$ is $\QQ$-factorial and quasi-projective.

\subsection{The relative Cox sheaf}\label{ss:cox}

We first define relative Cox sheaves, which are graded by the relative Picard group.

\begin{defi}\label{def:relativecox}
	Let $f\colon X\to S$ be a projective morphism of varieties with $X$ normal, and assume that $\Pic(X/S)$ is finitely generated and free.
	Let $K\subseteq\operatorname{CDiv}(X)$ be a subgroup of the group of Cartier divisors on $X$ such that the natural map $K\to\Pic(X/S)$ is an isomorphism.
	The \emph{relative Cox sheaf} of $f$ with respect to $K$ is the $\Pic(X/S)$-graded $\Ob S$-algebra
	\[
		\Rc\coloneq\bigoplus_{D\in K}f_*\Ob X(D),
	\]
	with multiplication induced by multiplication of rational functions.
	The summand $f_*\Ob X(D)$ has degree $[D]$.
\end{defi}

To apply \cref{def:relativecox} to $\pi$, we next identify the relative Picard group.

\begin{lem}\label{lem:relativepic}
	In \cref{set:main}, the group $\Pic(\tilde X/X)$ is free with basis $[E_1],\dotsc,[E_n]$.
	In particular, the natural map
	\[
		\Pic(\tilde X/X)_\RR
		\longrightarrow N^1(\tilde X/X)
	\]
	is an isomorphism.
\end{lem}

\begin{proof}
	The classes $[E_1],\dotsc,[E_n]$ generate $\Pic(\tilde X/X)$ by \cref{lem:divisordecomposition}.
	They are linearly independent, since their images in $N^1(\tilde X/X)$ form a basis.
	The same fact gives the last assertion.
\end{proof}

By \cref{lem:relativepic}, we can apply \cref{def:relativecox} to $\pi$, taking as $K$ the subgroup generated by $E_1,\dotsc,E_n$.
The Cox construction for proper birational morphisms in \cite[\S1.1]{WlodarczykCox} uses the same grading by exceptional divisors, and we compute $\Rc$ using the results there.

\begin{thm}\label{thm:cox}
	In \cref{set:main}, let $\Rc$ be the relative Cox sheaf of $\pi$ with respect to the subgroup generated by $E_1,\dotsc,E_n$, and identify $\Pic(\tilde X/X)$ with $\ZZ^n$ by \cref{lem:relativepic}.
	Then the following hold.
	\begin{enumerate}
		\item\label{item:coxpieces}
			For every $m\in\ZZ^n$, the degree-$m$ part of $\Rc$ is
			\[
				\Rc_m=\bigcap_{i=1}^n\Ic_{Z_i}^{\max\{-m_i,0\}}
				\subseteq\Ob X.
			\]
		\item\label{item:coxlocal}
			Let $p\in\bigcup_iZ_i$ be a closed point, and choose a coordinate model of $\pi\colon\tilde X\to X$ at $p$, with \'etale morphism $\phi_p\colon U_p\to\AA^d$.
			Put $I\coloneq\sset{i}{p\in Z_i}$ and $I_j\coloneq\sset{i\in I}{j\in J_i}$ for $1\le j\le d$, and let $\varepsilon_1,\dotsc,\varepsilon_n$ be the standard basis of $\ZZ^n$.
			Then
			\[
				\restr\Rc{U_p}\isom
				\Ob{U_p}\otimes_{k[x_1,\dotsc,x_d]}
				k[z_1,\dotsc,z_d,\;t_i\;(i\in I),\;t_h^{\pm1}\;(h\notin I)],
			\]
			where the algebra structure and grading on the right are given by
			\[
				x_j\longmapsto z_j\prod_{i\in I_j}t_i,
				\qquad
				\deg t_i=\varepsilon_i,
				\qquad
				\deg z_j=-\sum_{i\in I_j}\varepsilon_i.
			\]
		\item\label{item:coxfg}
			The sheaf $\Rc$ is a finitely generated $\Ob X$-algebra.
	\end{enumerate}
\end{thm}

\begin{proof}
	Let $m\in\ZZ^n$.
	By \cite[Lemma~1.1.5]{WlodarczykCox}, we have $\Rc_m\subseteq\Ob X$, and a regular function $f$ is a section of $\Rc_m$ if and only if $v_{E_i}(f)\ge-m_i$ ($1\le i\le n$), where $v_{E_i}$ is the valuation of $E_i$, and this condition is automatic if $m_i\ge0$.
	Hence
	\[
		\Rc_m=\pi_*\Ob{\tilde X}\biggl(-\sum_i\max\{-m_i,0\}E_i\biggr),
	\]
	and \cref{item:coxpieces} follows from \cref{lem:valuation}.

	For \cref{item:coxlocal}, let $F_1,\dotsc,F_n$ be the exceptional divisors of the chosen coordinate model $\tilde Y_p\to\AA^d$.
	By \cref{prop:localmodel} and flat base change along $\phi_p$, each graded piece of $\restr\Rc{U_p}$ is the pullback of $H^0(\tilde Y_p,\Ob{\tilde Y_p}(\sum_im_iF_i))$, compatibly with multiplication.
	Since $F_h=0$ for $h\notin I$, these pieces do not depend on $m_h$ for $h\notin I$, which gives the torus factors $t_h^{\pm1}$.
	For $i\in I$, the divisor $F_i$ is the torus-invariant divisor $D_{v_i}$ by \cref{lem:validstar}, so its valuation takes the value $1$ on $x_j$ for $j\in J_i$ and $0$ on the other coordinates.
	Thus \cite[Lemma~2.3.5~(1)]{WlodarczykCox}, applied to the toric morphism $\tilde Y_p\to\AA^d$, gives the remaining factor $k[z_1,\dotsc,z_d,t_i\ (i\in I)]$ in \cref{item:coxlocal}, with the stated algebra structure and grading.
	Here the variable $t_i$ of \cite{WlodarczykCox} is our $t_i^{-1}$, since $F_i$ corresponds to $t_i^{-1}$ there.
	This proves \cref{item:coxlocal}.

	Over $X\setminus\bigcup_iZ_i$, the sheaf $\Rc$ is $\Ob X[t_1^{\pm1},\dotsc,t_n^{\pm1}]$.
	This open subset and the coordinate neighborhoods $U_p$ cover $X$, so \cref{item:coxlocal} shows that $\Rc$ is finitely generated over $\Ob X$.
	This proves \cref{item:coxfg}.
\end{proof}

\subsection{Small modifications and Mori dream morphisms}\label{ss:moridream}

In this subsection, we show that the varieties obtained from $\tilde X$ by flips over $X$ have rational polyhedral relative cones of curves, and that there are only finitely many of them up to isomorphism over $X$.
We also show that $\pi$ is a Mori dream morphism when $X$ is $\QQ$-factorial and quasi-projective.

\begin{prop}\label{prop:flipmodels}
	In \cref{set:main}, let $\hat\pi\colon\hat X\to X$ be obtained from $\pi\colon\tilde X\to X$ by a finite sequence of $D$-flips of small extremal ray contractions over $X$, where the $\RR$-Cartier divisor $D$ may change from step to step.
	Then the following hold.
	\begin{enumerate}
		\item\label{item:flipmodelcone}
			The cone $\NEb(\hat X/X)$ is rational polyhedral, and every $\hat\pi$-nef $\QQ$-Cartier divisor on $\hat X$ is $\hat\pi$-semiample.
		\item\label{item:flipmodelsteps}
			Every face of $\NEb(\hat X/X)$ admits a contraction over $X$.
			If the contraction $g_R\colon\hat X\to\hat X_R$ of an extremal ray $R$ is small, then the $D$-flip of $g_R$ exists for every $\RR$-Cartier divisor $D$ on $\hat X$ with $D\cdot R<0$.
		\item\label{item:flipmodelfinite}
			There are only finitely many such varieties $\hat X$ up to isomorphism over $X$.
			Moreover, every sequence of such $D$-flips terminates if each $D$ is the strict transform of one fixed $\RR$-Cartier divisor on $\tilde X$.
	\end{enumerate}
\end{prop}

\begin{proof}
	Choose finitely many coordinate neighborhoods $U_p$ which, together with $X_0\coloneq X\setminus\bigcup_iZ_i$, cover $X$.
	For $\hat X$ as in the statement, write $\hat E_i$ for the strict transform of $E_i$ on $\hat X$.
	We first prove \cref{item:flipmodelcone,item:flipmodelsteps} in the following stronger form.

	\begin{claim}\label{claim:flipmodelchart}
		Let $\hat X$ be as in the statement.
		Then \cref{item:flipmodelcone,item:flipmodelsteps} hold for $\hat X$, and $\hat X$ has the following local description.
		The morphism $\hat\pi$ is an isomorphism over $X_0$.
		For every chosen $p$, there is a commutative diagram, in which both squares are cartesian,
		\[
		\begin{tikzcd}
			\hat X\arrow[d,"\hat\pi"']
				& \hat U_p\arrow[r,"\hat\phi_p"]\arrow[d]\arrow[l,hook']
				& \hat Y_p\arrow[d]\\
			X
				& U_p\arrow[r,"\phi_p"']\arrow[l,hook']
				& \AA^d
		\end{tikzcd}
		\]
		where $\hat Y_p\to\AA^d$ is a projective birational toric morphism whose fan is simplicial and has the same rays as the fan of $\tilde Y_p$.
		Moreover,
		\[
			\restrdiv{\hat E_i}{\hat U_p}=\hat\phi_p^{\,*}\hat F_i
			\qquad(1\le i\le n),
		\]
		where $\hat F_i$ is the torus-invariant prime divisor of $\hat Y_p$ corresponding to the ray $\rho_{v_i}$ if $p\in Z_i$, and $\hat F_i\coloneq0$ if $p\notin Z_i$.
	\end{claim}

	\begin{proof}[Proof of \cref{claim:flipmodelchart}]
		We prove \cref{claim:flipmodelchart} by induction on the number of flips.
		By \cref{prop:localmodel}, the variety $\tilde X$ has the local description.
		Now assume that $\hat X$ has the local description.
		It then suffices to show that \cref{item:flipmodelcone,item:flipmodelsteps} hold for $\hat X$ and that every flip of a small extremal ray contraction of $\hat X$ again has the local description.

		Each $\hat E_i$ is $\QQ$-Cartier.
		Indeed, $\hat E_i$ lies over $Z_i$, and $\restrdiv{\hat E_i}{\hat U_p}=\hat\phi_p^{\,*}\hat F_i$ for every chosen $p$, where $\hat F_i$ is $\QQ$-Cartier because the fan of $\hat Y_p$ is simplicial.

		We first prove \cref{item:flipmodelcone} for $\hat X$.
		By \cite[Theorem~1.6 and Proposition~1.11]{CvR}, $\NEb(\hat Y_p/\AA^d)$ is generated by the classes of finitely many wall curves, all of which lie over $0$.
		The diagram for $\hat X$ identifies $\hat\pi^{-1}(p)$ with the fiber of $\hat Y_p$ over $0$, and the equality $\restrdiv{\hat E_i}{\hat U_p}=\hat\phi_p^{\,*}\hat F_i$ identifies intersection numbers with $\hat E_i$ with those with $\hat F_i$.
		As in \cref{lem:NEinj}, we obtain an injective linear map
		\[
			\hat\iota_p\colon N_1(\hat Y_p/\AA^d)\longhookrightarrow N_1(\hat X/X)
		\]
		that preserves these intersection numbers.
		Every curve contracted by $\hat\pi$ lies over a point of some chosen $U_p$, and its class is the image under $\hat\iota_p$ of the class of a curve in a fiber of $\hat Y_p\to\AA^d$.
		Hence $\NEb(\hat X/X)$ is generated by the images of the finitely many wall classes, and it is rational polyhedral.

		Let $L$ be a $\QQ$-Cartier divisor on $\hat X$.
		Since $X$ is smooth along $\bigcup_iZ_i$ and $\hat\pi$ is an isomorphism over $X_0$, the pushforward $L_0\coloneq\hat\pi_*L$ is $\QQ$-Cartier, and hence
		\[
			L=\hat\pi^*L_0+\sum_i a_i\hat E_i
			\qquad(a_i\in\QQ).
		\]
		Assume that $L$ is $\hat\pi$-nef.
		Replacing $L$ by $L-\hat\pi^*L_0$, we may assume that $L=\sum_ia_i\hat E_i$.
		Put $L_p\coloneq\sum_ia_i\hat F_i$.
		Then $\restrdiv{L}{\hat U_p}=\hat\phi_p^{\,*}L_p$.
		The divisor $L_p$ is nef over $\AA^d$, since $\hat\iota_p$ identifies its degrees on the wall curves with degrees of $L$.
		By \cite[Proposition~4.6]{FujinoSato}, a sufficiently divisible multiple of $L_p$ is Cartier and base-point-free over $\AA^d$.
		Pulling back to $\hat U_p$ and taking a common multiple for the finitely many chosen $p$, we see that $L$ is $\hat\pi$-semiample.
		This proves \cref{item:flipmodelcone} for $\hat X$.

		We next prove \cref{item:flipmodelsteps} for $\hat X$.
		Let $F$ be a face of $\NEb(\hat X/X)$.
		Since $\NEb(\hat X/X)$ is rational polyhedral, $F$ has a supporting $\QQ$-Cartier divisor, which is $\hat\pi$-semiample by \cref{item:flipmodelcone}.
		As in the proof of \cref{cor:contract}, this divisor defines the contraction $g_F\colon\hat X\to\hat X_F$ of $F$, and the sequence
		\[
		\begin{tikzcd}
			0\arrow[r]
			& N^1(\hat X_F/X)\arrow[r,"g_F^*"]
			& N^1(\hat X/X)\arrow[r]
			& N^1(\hat X/\hat X_F)\arrow[r]
			& 0
		\end{tikzcd}
		\]
		is exact.
		Over $U_p$, the contraction $g_F$ is the base change of the toric contraction of the face $\hat\iota_p^{-1}(F)$ of $\NEb(\hat Y_p/\AA^d)$, as in \cref{lem:localcontract}.

		Now let $g_R\colon\hat X\to\hat X_R$ be the contraction of an extremal ray $R$ of $\NEb(\hat X/X)$, and assume that $g_R$ is small.
		For each chosen $p$, put $R_p\coloneq\hat\iota_p^{-1}(R)$, and let $g_{R_p}$ be the toric contraction of $R_p$.
		Over $U_p$, the morphism $g_R$ is the base change of $g_{R_p}$.
		If $R_p=0$, then $g_R$ is an isomorphism over $U_p$.
		If $R_p\ne0$, then $R_p$ is an extremal ray and $g_{R_p}$ is small, by the argument in the proof of \cref{thm:flip}.

		The classes $[\hat E_i]$ are linearly independent in $N^1(\hat X/X)$.
		Indeed, since the birational map $\tilde Y_p\dashrightarrow\hat Y_p$ over $\AA^d$ is an isomorphism in codimension one between $\QQ$-factorial varieties, strict transform gives an isomorphism $N^1(\tilde Y_p/\AA^d)\isomarrow N^1(\hat Y_p/\AA^d)$ sending $[F_i]$ to $[\hat F_i]$.
		The nonzero classes $[F_i]$ form a basis of the source, and each $Z_i$ contains one of the chosen points.

		The proof of \cref{lem:flipindependence}, with the exact sequence above and the linear independence of the classes $[\hat E_i]$, shows the following.
		If the $D$-flip of $g_R$ exists for an $\RR$-Cartier divisor $D$ on $\hat X$ with $D\cdot R<0$, then it is also a $D'$-flip for every $\RR$-Cartier divisor $D'$ on $\hat X$ with $D'\cdot R<0$.
		Hence it suffices to construct the $D$-flip for one such $D$ and to show that it has the local description.
		We can choose $D=\sum_ia_i\hat E_i$ with $D\cdot R<0$ such that $D_p\coloneq\sum_ia_i\hat F_i$ is Cartier on $\hat Y_p$ for every chosen $p$.

		If $R_p\ne0$, then $-D_p$ is $g_{R_p}$-ample.
		By \cite[Theorem~4.8]{FujinoSato}, the $D_p$-flip of $g_{R_p}$ then exists and is $\QQ$-factorial, and its fan has the same rays as the fan of $\hat Y_p$.
		The proof of \cref{thm:flip} now gives the $D$-flip $g_R^+\colon\hat X^+\to\hat X_R$ and shows that $\hat X^+$ has the local description.
		This proves \cref{item:flipmodelsteps} for $\hat X$, and hence \cref{claim:flipmodelchart}.
	\end{proof}

	We now prove \cref{item:flipmodelfinite}.
	For each chosen $p$, there are only finitely many simplicial fans with support $\Cone(e_1,\dotsc,e_d)$ whose rays are the rays of the fan of $\tilde Y_p$.
	Let $\hat X$ and $\hat X'$ be two varieties as in the statement such that $\hat Y_p$ and $\hat Y'_p$ have the same fan for every chosen $p$.
	Then the diagrams in \cref{claim:flipmodelchart} give isomorphisms $\hat U_p\isom\hat U'_p$ over $U_p$, and these restrict to the identity over $X_0\cap U_p$.
	Hence they glue to an isomorphism $\hat X\isom\hat X'$ over $X$.
	Thus there are only finitely many such varieties $\hat X$ up to isomorphism over $X$.

	Finally, let $D$ be an $\RR$-Cartier divisor on $\tilde X$.
	By \cref{lem:divisordecomposition}, we may assume that $D=\sum_ia_iE_i$ with $a_i\in\RR$.
	We first assume that $a_i\in\QQ$ for every $i$.
	Over each chosen $U_p$, every flip of a variety $\hat X$ for the strict transform $\sum_ia_i\hat E_i$ of $D$ is either an isomorphism or the base change of a toric flip of $\hat Y_p$ for $\sum_ia_i\hat F_i$.
	Every sequence of such toric flips terminates by \cite[Theorem~4.9]{FujinoSato}.
	Each flip fails to be an isomorphism over at least one of the finitely many chosen $U_p$, so every sequence of flips for the strict transforms of $D$ terminates.

	In general, suppose that there is an infinite sequence of flips for the strict transforms of $D$.
	By the first part of \cref{item:flipmodelfinite}, the varieties in the sequence are isomorphic over $X$ to only finitely many varieties, and each of them has only finitely many extremal rays by \cref{item:flipmodelcone}.
	Hence the flipped extremal rays $R$ give only finitely many inequalities $\sum_ib_i\hat E_i\cdot R<0$ in $(b_1,\dotsc,b_n)\in\RR^n$.
	They define an open subset of $\RR^n$ containing $(a_1,\dotsc,a_n)$, so we can choose $b_i\in\QQ$ satisfying the same inequalities.
	Every flip in the sequence is then also a flip for the strict transform of $\sum_ib_iE_i$.
	This contradicts the case of rational coefficients.
	This proves \cref{item:flipmodelfinite}.
\end{proof}

Let $f\colon X\to S$ be a projective morphism of varieties with $X$ normal.
A \emph{small $\QQ$-factorial modification over $S$} is a birational map $X\dashrightarrow X'$ over $S$ that is an isomorphism in codimension one, where $X'$ is a normal $\QQ$-factorial variety projective over $S$.

A $\QQ$-Cartier divisor $D$ on $X$ is \emph{movable over $S$} if there is a positive integer $m$ such that $mD$ is Cartier and the support of the cokernel of the natural map $f^*f_*\Ob X(mD)\to\Ob X(mD)$ has codimension at least two in $X$.
The \emph{movable cone} $\Mov(X/S)\subseteq N^1(X/S)$ is the closure of the convex cone generated by the classes of movable divisors on $X$ over $S$.

We recall the definition of a Mori dream morphism from \cite[Definition~3.1]{Ohta}.

\begin{defi}\label{def:mdm}
	Let $f\colon X\to S$ be a projective surjective morphism of normal quasi-projective varieties with $f_*\Ob X=\Ob S$.
	The morphism $f$ is called a \emph{Mori dream morphism} if the following hold.
	\begin{enumerate}
		\item\label{item:mdmQfactorial}
			The variety $X$ is $\QQ$-factorial.
		\item\label{item:mdmpic}
			The natural map $\Pic(X/S)_\QQ\to N^1(X/S)_{\QQ}$ is an isomorphism.
		\item\label{item:mdmnef}
			The nef cone $\Nef(X/S)$ is rational polyhedral and generated by classes of divisors semiample over $S$.
		\item\label{item:mdmmov}
			There are finitely many small $\QQ$-factorial modifications $\psi_i\colon X\dashrightarrow X_i$ over $S$, each satisfying \cref{item:mdmQfactorial,item:mdmpic,item:mdmnef}, such that
			\[
				\Mov(X/S)=\bigcup_i\psi_i^*\Nef(X_i/S).
			\]
	\end{enumerate}
\end{defi}

\begin{cor}\label{cor:mdm}
	In \cref{set:main}, assume that $X$ is $\QQ$-factorial and quasi-projective.
	Then $\pi\colon\tilde X\to X$ is a Mori dream morphism.
\end{cor}

\begin{proof}
	We check the conditions of \cref{def:mdm}.
	By \cref{lem:inherit}, the variety $\tilde X$ is $\QQ$-factorial.
	It is quasi-projective because $\pi$ is projective and $X$ is quasi-projective, and $\pi_*\Ob{\tilde X}=\Ob X$ because $X$ is normal.
	Condition \cref{item:mdmpic} follows from \cref{lem:relativepic}.
	Condition \cref{item:mdmnef} follows from \cref{thm:cone,thm:bpf}, since $\Nef(\tilde X/X)$ is dual to the rational polyhedral cone $\NEb(\tilde X/X)$ and is therefore generated by classes of $\pi$-nef Cartier divisors, which are $\pi$-semiample.

	For \cref{item:mdmmov}, we use the varieties $\hat X$ of \cref{prop:flipmodels}, together with the induced birational maps $\tilde X\dashrightarrow\hat X$, which are isomorphisms in codimension one.
	We first show that each $\hat X$ satisfies \cref{item:mdmQfactorial,item:mdmpic,item:mdmnef}.
	As in the proof of \cref{prop:flipmodels}, each $\hat E_i$ is $\QQ$-Cartier.
	Let $D$ be a $\QQ$-Weil divisor on $\hat X$, and put $D_0\coloneq\hat\pi_*D$.
	Since $X$ is $\QQ$-factorial, $D_0$ is $\QQ$-Cartier, and we can write $D=\hat\pi^*D_0+\sum_ia_i\hat E_i$ with $a_i\in\QQ$.
	Hence $D$ is $\QQ$-Cartier, and $\hat X$ is $\QQ$-factorial.
	If $D$ is numerically trivial over $X$, then all the $a_i$ vanish, since the classes $[\hat E_i]$ are linearly independent by the proof of \cref{prop:flipmodels}.
	Thus $D=\hat\pi^*D_0$, which gives \cref{item:mdmpic} for $\hat X$.
	Condition \cref{item:mdmnef} for $\hat X$ follows from \cref{prop:flipmodels}~\cref{item:flipmodelcone}.

	We next show that $\Mov(\tilde X/X)$ is contained in the union of the strict transforms of the cones $\Nef(\hat X/X)$.
	Let $D$ be a $\QQ$-divisor on $\tilde X$ that is movable over $X$, and assume that $D$ is not $\pi$-nef.
	Since $\NEb(\tilde X/X)$ is rational polyhedral, there is an extremal ray $R$ with $D\cdot R<0$.
	The contraction of $R$ is small.
	Indeed, otherwise its exceptional locus contains a prime divisor covered by curves with class in $R$, and every local section of a Cartier multiple of $D$ vanishes along this divisor, contrary to movability.
	Thus the $D$-flip exists by \cref{prop:flipmodels}~\cref{item:flipmodelsteps}.
	The strict transform of $D$ is again movable, since the two varieties are isomorphic in codimension one and the relative section sheaves of corresponding divisors are the same.
	Repeating this argument on the new variety, we obtain by \cref{prop:flipmodels}~\cref{item:flipmodelfinite} a variety $\hat X$ on which the strict transform of $D$ is $\hat\pi$-nef.
	Since positive rational combinations of movable divisors are movable, $\Mov(\tilde X/X)$ is the closure of the set of classes of movable $\QQ$-divisors.
	There are only finitely many varieties $\hat X$ by \cref{prop:flipmodels}~\cref{item:flipmodelfinite}, so the union of the strict transforms of their nef cones is closed.
	Hence it contains $\Mov(\tilde X/X)$.

	Conversely, the strict transform on $\tilde X$ of a $\hat\pi$-ample $\QQ$-divisor on any $\hat X$ is movable over $X$, and hence the strict transform of $\Nef(\hat X/X)$ is contained in the closed cone $\Mov(\tilde X/X)$.
	This proves \cref{item:mdmmov}.
\end{proof}

\begin{rem}\label{rem:coxrelation}
	When $k=\CC$ and $X$ is $\QQ$-factorial and quasi-projective, \cref{cor:mdm} also follows from \cite[Theorem~1.2]{Ohta} together with \cref{thm:cox,lem:relativepic}, and Mori dream theory then gives rational polyhedrality of $\NEb(\tilde X/X)$, semiampleness, and the existence of contractions and flips \cite[\S5]{Ohta}.
	It does not give the elementary generators of \cref{thm:cone}, base-point-freeness of nef divisors, or the descriptions of contraction targets, exceptional loci, and reordered flips in \cref{sec:contractions}.
\end{rem}

\printbibliography

\end{document}